\documentclass{amsart} 
\usepackage{amsmath, amsthm, amscd, amsfonts,amssymb,graphicx}
\numberwithin{equation}{section}

\newcommand{\teq}{\arabic{section}.\arabic{equation}}
\newcommand{\teql}{\Alph{section}.\arabic{equation}}

\newcommand{\sqr}[2]{{\vcenter{\vbox{\hrule height.#2pt\hbox{\vrule width.#2pt
height#1pt \kern#1pt\vrule width.#2pt}\hrule height.#2pt}}}}

\newcounter{eqcount}

\renewcommand{\labelenumi}{{{\rm (\teq \alph{enumi})}}} 
\newenvironment{edesc}{\refstepcounter{equation}\begin{enumerate}}%
{\end{enumerate}}
\newenvironment{triv}{\refstepcounter{equation}\begin{list}%
{{\hbox{\rm(\teq)\ }}} \item }{\end{list}}
\newenvironment{trivl}{\refstepcounter{equation}\begin{list}%
{{\hbox{\rm(\teql)\ }}} \item }{\end{list}}

\newcommand{\ring}[1]{{\mathbb #1}}
\newcommand\bZ{{\ring{Z}}}

\newcommand\bC{{\ring{C}}} \newcommand\bR{{\ring{R}}}
\newcommand\bF{{\ring{F}}} \newcommand\bQ{{\ring{Q}}}

\newcommand{\csp}[1]{{\mathbb #1}}
\newcommand{\tsp}[1]{{\mathcal #1}}

\newcommand{\prP}{\csp{P}}
\newcommand{\afA}{\csp{A}}
 
\newcommand{\sC}{{\tsp{C}}} 
\newcommand{\sO}{{\tsp{O}}} \newcommand{\sI}{{\tsp{I}}}
\newcommand{\sQ}{\tsp{Q}}
\newcommand{\sP}{{\tsp {P}}} 
\newcommand{\sL}{{\tsp {L}}} 
\newcommand{\sT}{{\tsp {T}}} \newcommand{\sH}{{\tsp {H}}}
 
 \newcommand{\sD}{{\tsp {D}}}
 
\newcommand{\sR}{{\tsp {R}}} 
 
 \newcommand{\bN}{{\csp {N}}}
\renewcommand{\ni}{\texto{Ni}}

\newcommand{\eql}[2]{{\rm (\ref{#1}\ref{#2})}} 

\newcommand{\vect}[1]{{\pmb #1}} 
\newcommand{\ba}{{\vect{a}}} \newcommand{\bg}{\vect{g}}
 
\newcommand{\bp}{{\vect{p}}} \newcommand{\bx}{{\vect{x}}}
 
 \newcommand{\bz}{{\vect{z}}}
  
  \newcommand{\bu}{{\vect{u}}}
\newcommand{\row}[2]{{#1_1,\ldots,#1_{#2}}}

\newcommand{\smatrix}[4]{{\big(\begin{array}{cc}
\!\lower2pt\hbox{$\scriptstyle#1$} &\lower2pt\hbox{$\scriptstyle#2$}\!
\\\! \raise2pt\hbox{$\scriptstyle#3$} &\raise2pt\hbox{$\scriptstyle#4$}
\!\end{array}\big)}}

\newcommand{\texto}[1]{{\textr{#1}}}
\newcommand{\GL}{\texto{GL}} \newcommand{\SL}{\texto{SL}}
 \newcommand{\ind}{\texto{ind}}
\newcommand{\PSL}{\texto{PSL}} \newcommand{\PGL}{\texto{PGL}}
 
 \renewcommand{\ni}{\texto{Ni}}
\newcommand{\Spec}{\texto{Spec}} 
\newcommand{\textr}[1]{{\text{\rm #1}}}
\newcommand{\tr}{\textr{tr}} 
\newcommand{\abs}{\textr{abs}}  
 
 \newcommand{\inn}{\textr{in}}
 \newcommand{\Aut}{\textr{Aut}}
\newcommand{\pr}{\textr{pr}}

\newcommand{\rd}{\texto{rd}}

\newcommand{\sph}{{\vphantom 1}}

\newcommand{\textb}[1]{{\text{\bf #1}}}
\newcommand{\bfC}{{\textb{C}}}
\newcommand{\longmapright}[2]{\smash{\mathop{\hbox to
#2pt{\rightarrowfill}}\limits^{#1}}}
\newcommand{\longmapleft}[2]{\smash{\mathop{\hbox to
#2pt{\leftarrowfill}}\limits^{#1}}}

\newcommand{\mapright}[1]{\smash{\mathop{\longrightarrow}\limits^{#1}}}

\newcommand{\np}{{+}}   \newcommand{\nm}{{-}}

\newcommand{\lrang}[1]{{\langle #1\rangle}}

\newcommand{\eqdef}{\stackrel{\text{\rm def}}{=}}

\newfont{\sevenrm}{cmr7}
\newfont{\bsevenrm}{cmbx7}
\newfont{\mathseven}{cmsy7}
\newfont{\bigmath}{cmsy10 scaled 1200}
\newfont{\fiverm}{cmr5}
\newfont{\bfiverm}{cmbx5}
\newfont{\hel}{cmbx10 scaled 1200}
\newfont{\eu}{eufb10}
\newfont{\sseu}{eufm5}
\newfont{\seu}{eufm7}
\newfont{\Cal}{cmmib10}
\newfont{\sCal}{cmmib7}
\newfont{\zch}{eusb10}

\theoremstyle{plain}
\newtheorem{thm}{Theorem}[section] 
\newtheorem{lem}[thm]{Lemma}

\newtheorem{prop}[thm]{Proposition}

\newtheorem{res}[thm]{Result}
\theoremstyle{definition}
\newtheorem{defn}[thm]{Def}
\newtheorem{exmp}[thm]{Example}
\newtheorem{guess}[thm]{Conj}

\newtheorem{prob}[thm]{Problem}

\theoremstyle{remark}
\newtheorem{rem}[thm]{Remark}

\newcommand{\xs}{\times^s\!}

\def\pic #1 by #2 (#3){\vbox to #2{\hrule width 
#1 height 0pt depth 0pt\vfill\special{picture #3}}}
\def\scaledpicture#1
by #2 (#3 scaled #4){{\dimen0=#1 \dimen1=#2\divide\dimen0 by
 1000 \multiply\dimen0 by #4\divide\dimen1 by 1000 \multiply\dimen1 
by #4\pic \dimen0 by \dimen1 (#3 scaled #4)}}

\newcommand{\comm}[1]{{}}
\newcommand{\sB}{{\tsp {B}}}

\renewcommand{\phi}{\varphi}
\newcommand{\Fr}{\text{Fr}}
\newcommand{\sV}{\tsp V}

\newcommand{\lcm}{\text{lcm}}
\newcommand{\set}{\text{set}}

\newcommand{\C}{{\text{\rm C}}}

\newcommand{\psigma}{{\pmb \sigma}}
\newcommand{\ab}{{{}_{\text{\rm ab}}}}
\newcommand{\geng}{{{\text{\bf g}}}}
\newcommand{\sh}{{{\text{\bf sh}}}}
\newcommand{\br}{{{\text{\rm br}}}}
\newcommand{\red}{{{\text{\rm rd}}}}

\newcommand{\sF}{{\tsp F}}
\newcommand{\by}{{\pmb y}}
\newcommand{\tP}{{\text{\rm P}}}
\newcommand{\Sup}{{\text{\rm Sup}}}

\begin{document}

\title[Variables Separated]{Variables separated equations: \\ \small{Strikingly different roles for the Branch Cycle Lemma and the Finite Simple Group Classification}}

\author[M.~D.~Fried]{Michael
D.~Fried${}^*$}

\date{\today}

\address{Emeritus, UC Irvine, Irvine, CA 92697, USA}
\email{mfried@math.uci.edu}

\begin{abstract}
\!\!\!\!H.~Davenport's Problem asks: What can we expect of two polynomials, over $\bZ$, with the same ranges on almost all residue class fields?  This stood out among many separated variable problems posed by  Davenport, D.J.~Lewis and A.~Schinzel.  

By bounding the degrees, but expanding the maps and variables in Davenport's Problem, {\sl Galois stratification\/} enhanced the separated variable theme,  solving an Ax and Kochen problem from their Artin Conjecture work. J.~Denef and F.~Loeser applied this to add {\sl Chow motive coefficients\/} to previously introduced zeta functions on a diophantine statement. 

By restricting the variables, but leaving the degrees unbounded, we found the striking distinction  between Davenport's problem over $\bQ$, solved by applying the Branch Cycle Lemma, and its generalization over any number field, solved using the simple group classification. This encouraged J.~Thompson  to formulate the {\sl genus\/} 0 problem on rational function monodromy groups. R.~Guralnick and Thompson led its solution  in stages. 

We look at at two developments since the solution of Davenport's problem.   \begin{itemize} \item Stemming from C.~MacCluer's 1967 thesis, identifying a general class of problems, including Davenport's, as  {\sl monodromy precise\/}. 
\item R(iemann) E(xistence) T(heorem)'s role as a converse to problems generalizing Davenport's, and Schinzel's (on reducibility).\end{itemize}  We use these to  consider: Going beyond the simple group classification to handle imprimitive groups; and what is the role of covers and correspondences  in going from algebraic equations to zeta functions with Chow motive coefficients. 
\end{abstract}



\thanks{This is a mathematically more complete version of \cite{UM08}, based on themes from my  graduate student time at University of Michigan ('64-'67). }
\maketitle

\tableofcontents


\section{Davenport's Problem} 
Algebraic equations occur in many modern data problems. They represent relations between variables 
defining data. The data variable gives us a monodromy (or Galois) group with a faithful 
permutation representation.  Data-variable problems  should have 
convenient coefficients, such as ordinary fractions, $\bQ$. Then, there are two monodromy 
groups: the {\sl arithmetic\/} (over $\bQ$) and a normal subgroup of it, the {\sl geometric\/} (over the 
algebraic closure, $\bar \bQ$).   There is then an encompassing {\sl inverse problem\/}. Suppose you are given such a pair of groups 
(with their compatible permutation representations), one normal in the other.  Find an equation and 
data-variable over $\bQ$, having that (arithmetic, geometric) monodromy pair.   

\subsection{Relating four problems} \label{secI.0} 
The data-variable has many applications. For example, mappings of the sphere to the sphere, are everywhere in Cryptography:  Over infinitely 
many prime residue classes, an {\sl exceptional\/} rational function $f$ maps one-one from the data to its values. The {\sl Schur Conjecture\/}, Prop.~\ref{schurConj}, was the proposed classification of such covers where $f$  is a polynomial.     Davenport's problem was, essentially, to classify polynomials over 
$\bQ$ by their ranges on almost all residue class fields. \S\ref{App.1b} explains the notation for residue class fields, $\sO_K/\bp$, of a number field $K$, defined by primes $\bp$. Problems that interested Davenport seem extremely different from those that attract algebraists interested in {\sl motives\/}. Yet Davenport's  {\sl very\/} specific problem  led to two general results: The genus 0 problem, and the encoding of all diophantine statements into zeta functions. 

Davenport's problem, restricted to 
polynomials not composable ({\sl indecomposable}) from lower degree polynomials, gave two very 
different conclusions.    
\begin{edesc} \label{davpart} \item  \label{davparta} D$_1$: Over $\bQ$, two polynomials with the same range are linearly equivalent: 
obtainable, one from the other, by a linear change of variables.  
\item  \label{davpartb} D$_2$: Linearly inequivalent polynomials can have the same ranges 
on all residue classes of a number field, but  a fixed constant (31) bounds the degrees of these exceptions.  \end{edesc} 

Schur's conjecture was a stop on the way to completing Davenport's Problem. Still, in  an analog of Schur for rational functions (\cite[\S2]{Fr78}, \cite[\S6.1--6.3]{Fr05b}, \cite{GMS03}) reinterpret Serre's O(pen)I(mage)T(heorem) to connect the monodromy method (\S\ref{prodBCs})  
 to modular curves \cite{Se68}. Since the rational functions here have dihedral geometric monodromy group, you might think their analysis trivial. That's not so, for the properties of their natural families gives the depth of the story. Likewise, one reason for returning to Davenport's problem is to document modern methods that simplify describing the families that occur there (\S\ref{App.5}).

 Two tools for investigating equations came early in the monodromy method: 
 \begin{itemize} \item the 
{\sl B(ranch)C(ycle)L(emma)\/}; and 
\item the {\sl Hurwitz monodromy group}. \end{itemize} By walking through Davenport's problem 
with hindsight, we see why the -- rarely acknowledged -- preoccupation with variables separated 
equations gave important lessons on these tools. To simplify the presentation of the BCL, we have broken its use into deciding when covers can't be over $\bQ$ (\S\ref{secVI.1}), and figuring out the natural cyclotomic field that appears (\S\ref{famdeffield}).   Davenport's Problem  {\sl explicitly\/} used both aspects, by comparison with general applications starting with \cite{Fr77}.  That shows in what the  solutions of \eql{davpart}{davpartb} contributed to the Genus 0 Problem. We call attention to the use of function theory in these results through these lessons:    
\begin{edesc} \label{bigProblems} \item What allows us to produce branch cycles \S\ref{secVI.4}. 
\item What is the relation between covers and Chow motives \S\ref{secVII.2}.  
\item What 'in nature' (a phrase from \cite{So01}, see ¤\ref{secI.3}) gives today's challenges to group 
theory  \S\ref{secVII.3}.  \end{edesc} 
 Each phrase addresses an aspect of formulating problems based on equations. That is, many 
disciplines seem to need algebraic equations. Yet why, and how much do we lose in using more easily 
manipulated surrogates for them? 

\S\ref{secIII.2} says that some conclusions drawn from applying  Cebotarev's density theorem, can be made precise. Chebotarev usually gives a crude translation between statements over finite fields and the monodromy group of a cover, rarely capturing the diophantine statement over the ramification locus. We develop two aspects of Davenport's problem that generalize to support {\sl monodromy precision\/}  (\S\ref{MPres}) and an {\sl RET Converse\/} (\S\ref{RETConv}). My summary starts with a Davenport-Lewis 
paper \cite{DL63}. We interpret this as the first special case of Monodromy Precision: About exceptional polynomials, but  now known to apply  far more generally.

Our examples tie theory to 
the enterprise of writing explicit equations.   It continues an Abel-Galois-Riemann 
tradition of solving problems where 
algebraic covers fall in continuous (connected) families. Often we complete the problem  by distinguishing (reduced Hurwitz space) components containing the desired solutions.  

Aspects of Davenport's problem would have surprised even Abel, Galois and 
Riemann. My examples: How it used the classification of finite simple groups; and how it led to the the genus 0 problem.  \cite{De99}, not as historical or elementary, and less connected to group theory, concentrates on how the Hilbert-Siegel problem of \S\ref{limitAlt} motivated using Hurwitz spaces. I wrote \cite{Fr73a} for an audience often discomfited by Grothendieck's geometry. So, unlike the French school, I often limited statements to rational functions in one variable (genus 0 covers).  

To amend that App.~\ref{whatCover} reminds of the Grothendieck cover definition: a finite, flat morphism. It then notes that most proofs not referring to branch cycles, work very generally. Example: \cite[Prop.~2]{Fr73a} is a much cited  lemma from Davenport's problem. It reverts factorization of separated variable equations to where two covers have identical Galois closures. It applies far beyond genus 0 covers. We call attention to this in how \S\ref{laurent} and \S\ref{fibproduniv} refer to extending Lem.~\ref{pullbackcomps}. 

Another subtlety raised by App.~\ref{whatCover}  occurs because I insist on restricting covers to normal varieties.  The subtleties arise only when their dimension exceeds one.  That affects our Galois Stratification vs Chow motives topic when, say, we consider the monodromy precision property on Davenport pairs. 

\subsection{Introduction to Davenport's Problem} \label{secI.1} Davenport stated his problem at a conference at Ohio State during my 2nd year of graduate school. The anchors for this story are the result I proved, D$_2$ \eql{davpart}{davpartb}, and the problem  that "seized" John Thompson -- his own words -- that came from it (G$_1$(0) below). 

D$_1$ said that two polynomials $f$ and $g$ over $\bQ$ with the same ranges on almost all finite fields, with $f$ indecomposable, must be related by an inner change of variables $\alpha$, $f(\alpha(x))=g(x)$, $\alpha(x)=ax+b$ a degree 1 (linear, or  affine)  transformation. That is, the conclusion is that  $f $ and $g$ are {\sl affine equivalent}. Davenport didn't include the indecomposable hypothesis. It translates to a primitive monodromy group (\S\ref{secIV}); progress would have been slow without it (see 
M\"uller's Conjecture \ref{MullConj}).

If $f$ and $g$ are a pair of rational functions having the same ranges for almost all primes $p$, then so will $\alpha\circ f$ and $\alpha\circ g$, their outer composition with $\alpha$ an affine transformation with coefficients in $\bQ$. If you compose $f$ with both inner and outer M\"obius transformations, we say the result is {\sl M\"obius equivalent\/} to $f$.

Problem G$_1$(0), {\sl The genus 0 Problem}, posed that all genus 0 {\sl primitive\/} covers have special covering (monodromy) groups and associated permutation representations (\S\ref{secI.3}). This has three parts: 

\begin{itemize} \item Genus 0 monodromy groups related to alternating, symmetric, dihedral and cyclic groups come in large families; 
\item within  alternating and symmetric related groups, large families occur only with a restricted set of  associated permutation representations; and  
\item there are but finitely many genus 0 monodromy groups outside these.\end{itemize} 

 All higher rank projective linear groups -- examples of almost simple groups (\S \ref{App.1}) -- over finite fields might have yielded solutions countering the expected Davenport conclusion. Yet, function theory showed only finitely many contribute to  D$_2$. Further, the most striking exceptional genus 0 monodromy groups appeared either from Davenport's problem, or from genus 0 upper half plane quotients that are \lq close-to\rq\ modular curves. Those from problem G$_2$(0) (\S \ref{secI.2}). 

Abel and Galois were aware of long monographs by Lagrange and his students. Here is one quick summary of much early 1800s mathematics. Galois showed the impossibility of uniformizing the function fields of the modular curves $X_0(p)$ (introduced by Abel), $p$ a prime $\ge  5$,  by radicals.

The 20th century didn't much use the phrase \rq\rq uniformized by radicals.\rq\rq\ Yet, despite attempts to avoid such an old formulation, a variant of it dominated published results in the 1960s. The algebraic equations I heard most about in graduate school had separated variables:

\begin{triv} \label{*} $f(x)-g(y) = 0$ with $f$ and $g$ polynomials, whose degrees we take (respectively) to be $m$ and $n$. \end{triv} 

By introducing a pair of covers of the Riemann sphere, we open the territory to using group theory. Rewrite \eqref{*} by introducing $z$ so as to split the variables:
\begin{equation} \label{*2}       f(x)-z = 0 \text{ and } g(y)-z = 0. \end{equation} 

Questions on solutions of \eqref{*},  in \eqref{*2} form, are equivalent to those with 
$(\alpha\circ f (\alpha'(x)), \alpha \circ g(\alpha''(y)))$ replacing $(f(x),g(y))$, with $\alpha$, $\alpha'$ and $\alpha''$ affine transformations. We say the former pair is affine equivalent to the latter.

Using  \eqref{*2} interprets \eqref{*} as relating two genus 0 covers (\S \ref{secII.1}).

Is it surprising that there are still mysteries about genus 0 covers? We will be precise about the most jarring ingredient from R(iemann)'s E(xistence) T(heorem) (\S \ref{secVI.4}). That is, how covers of the Riemann sphere relate to {\sl branch cycles}.  When the covers have genus 0 and appear naturally, many feel uncomfortable -- as did Kronecker and Weierstrass -- without an explicit uniformization. 

\subsection{Detecting a {\sl few\/} exceptions} \label{secI.2} The major surprise in Davenport's Problem was that D$_2$ was almost --  for all but finitely many degrees -- true. 
The explication  to the community of D$_2$ -- its finitely many exceptional degrees -- gave three results relating finite group theory to algebraic equations. \cite{Fr80} emphasized the connection between these problems and the (finite) simple group classification. Thm.~\ref{DS4} lists those exceptional degrees. \S\ref{App.5} emphasizes how group theory detected those exceptional degrees. It is brief considering how much comes from it. 

My  unofficial group theory background came from talking with many affiliated with the UM mathematics department. A year lapsed between graduate school and the conversation with Tom Storer (\S\ref{secVI}) -- during the summer of 1968. I needed that year to distinguish between D$_1$ and D$_2$. 

Also, Thompson's name was attached to another genus 0 problem, which I call G$_2$(0). This said that the $j$-line covers appearing in G$_2$(0) should have explicit uniformization by (upper half-plane) automorphic functions attached to representations of the Monster Simple group. They called it {\sl Monstrous Moonshine\/} and its resolution won Borchards a Fields Medal.

I first heard of G$_2$(0) during the group theory conference called Santa Cruz. (Its proceedings included \cite{Fr80}.) Even more time elapsed between the paper purporting to connect G$_1$(0) and G$_2$(0). The conversation I had with Thompson, while walking to lunch not long after my arrival at U.~of Florida (\S \ref{secVII.1}), was a planned serendipity.

\subsection{The Genus 0 Problem} \label{secI.3} The first tentative statement of the Genus 0 problem -- motivated by the solution of Davenport's problem (Prop.~\ref{DS5}) -- is in the last introduction paragraph  of \cite[p.~41]{Fr73b}. \cite[\S 7.2]{Fr05b} has its precise statement and its background (the attached html file has references and context). Roughly, due to the nature of branch cycles (\S \ref{secVI.4}), monodromy groups of rational functions fall -- with rare  exception -- among groups known to most mathematicians. Those exceptions, as in Davenport's Problem, have a serious impact.  The Genus 0 Problem formulation, and much work on it, is due to Bob Guralnick. 

Still, for those monodromy groups (with their representations) that do arise  in abundance,  researcher must ask if in their particular problems they occur often or not. \S \ref{limitDih} considers the arising of dihedral (and related) groups, and \S\ref{limitAlt} of alternating (and related) groups. These   example results put us in territory not -- at first -- limited by the genus 0 problem conclusion. 

Separated variable equations appeared with hyperelliptic curves, say, where Riemann first proved a generalization of Abel's Theorem. The Genus 0 Problem is a key step in considering what attributes of these equations qualify as special.
 
From the solution of D$_2$, three genus 0 curves, each a natural upper half-plane quotient and $j$-line cover, though not a modular curve, arise as parameter spaces for Davenport pairs.  From \S \ref{App.5} (for $n = 7$, 11, 13),  each space has an attached group representation of a projective linear group.  Does this lead to an explicit automorphic uniformizer, as in the Uniformizer Problem of  \S \ref{App.3}, for each?

Group theory can be demandingly intricate. What I, not a group theorist, found is that it can accomplish goals that would be worse than tiresome with equation manipulation. Much of modern group theory has little to do with permutation representations, though much of group theory's birth does. While \cite[p. 315--317]{So01}  does give a view on early group theory, even in its referral to Galois it differs much from mine. An audience question to  Ron Solomon, when he gave his history lecture \cite{So01} at UF, was [roughly], \lq\lq How did Galois' work survive?\rq\rq\ 

I suggested then an elaboration of the path  through Jacobi's interest in the uniformization result mentioned in \S \ref{secI.1}. I also mentioned that savior -- Crelle -- to both Abel and Galois. That was prior  to the renovation and update of Galois' work -- crucial to its survival -- by the  brilliant Jordan. 

Continuing an attempt at dialog with group theorists, I cite  \cite[p. 347]{So01}: \begin{quote} \dots\  experience shows that most of the finite groups which occur 'in nature' -- in the broad sense not simply of chemistry and physics, but of number theory, topology, combinatorics, etc. -- are 'close' either to simple groups or to groups such as dihedral groups [include affine groups as in  \S \ref{App.1}?], Heisenberg groups, etc, which arise naturally in the study of simple groups.\end{quote}  \S \ref{secVII.3} considers a more precise question: Do rational functions occur \lq\lq in nature?\rq\rq\  Nothing is more important to algebra than rational functions. To avoid trivial assurances of \lq\lq Yes, they do!\rq\rq consider that \cite[Chap. 3, \S 7.2.3]{chpfund.pdf} demonstrates the  
traditional renderings of rational function covers in $\bR^3$, say as in \cite[p. 243]{Con78}, are illusions, albeit one that Riemann himself used. Even with degree 2  covers. 

\subsection{UM affiliates and later work} \label{secI.4} 
With a superscript "a" for visiting junior faculty, "v" for visiting senior faculty, and "s" for (fellow) student, this  is a review of how the mathematicians A.~Brumer$^a$, R.~Bumby$^a$, H.~Davenport$^v$, D.J.~Lewis, W.~Leveque, R.~Lyndon, C.~MacCluer$^s$, R.~MacRae$^a$, 
R. Misera$^s$, J.~Mclaughlin, A.~Schinzel$^v$, J. Smith$^a$ influenced me as I sought to  meld a set of problems into a coherent story. They were at University of Michigan during my three years -- 1964--67 -- of graduate school. Soon after T.~Storer played a crucial role. 

The list above includes early influences on papers from my first three years (though later for publication) out of graduate school. The first six sections go over tools that solved Davenport's problem with emphasis on their relation to others' later work. \S\ref{secVI} on branch cycles and \S \ref{App.4} on the braid group -- the least used ingredients from Davenport's problem -- epitomize the {\sl monodromy method}.  

The long \S \ref{secVII} connects work of other authors -- R.~Abhyankar, R.M.~Avanzi,  J.~Ax, J.-M. Couveignes,  P.~D\`ebes, J.~Denef, W.~Feit, R.~Guralnick, I.~Gusi\'c, F.~Loeser, P.~M\"uller, F.~Pakovich, J.~Saxl, J.P.~Serre, J.~Thompson and U.~Zannier 
 -- to  the group, equation and function themes of \eqref{bigProblems}. These are people I've talked to (by e-mail at least), and here quote substantially. (I've left out direct reference to those -- unbeknownst to them -- whose papers I've refereed.) Often, however, I say those connections differently than do they. The biggest difference between \S \ref{secVII}  and the earlier sections is in the minimal use of branch cycles by others. Maybe this is my fault or that the applications aren't \lq\lq mainstream.\rq\rq\  Maybe, but \S\ref{nobcs} begs  to differ by offering two historical observations that affected all of mathematics. 

\section{Separated variables equations and group theory} \label{secII} By using  form \eqref{*2} in place of \eqref{*} we  relate two covers by Riemann spheres, $f: \prP^1_x \to \prP^1_z$ and $g: \prP^1_y \to   \prP^1_z$, of the Riemann sphere $\prP^1_z$. Recall: $\prP^1_z$  is just projective 1-space. The subscript  $z$ indicates an explicit isomorphism with affine 1-space union a point, $\infty$, at infinity. We always assume $f$ and $g$ are nonconstant.

\subsection{The effect of splitting the variables} \label{secII.1} Equation \eqref{*} defines an algebraic curve in affine 2-space. It has a completion in projective 2-space, with homogeneus variables $(x,y,w)$, by forming the curve $w^u(f(x/w)-g(y/w))=0$, $u$ the maximum of $m$ and $n$. This, however, is  likely singular. 

An advantage of \eqref{*2} is that it geometrically describes such singularities. They correspond to the pairs $(x',y')$  that both ramify in the respective maps $f$ and $g$ to $\prP^1_z$. That is, regard \eqref{*} as the fiber product -- set of pairs $(x',y')$  with $f(x')=g(y')$ -- of the two maps $f$ and $g$, but extend the fiber product over $\infty$. Papers use the notation $\prP^1_x\times^\set_{\prP^1_z}\prP^1_y$ for this set theoretic fiber product. We call any $z$ value over which there is a ramified point on $\prP^1_x$ a {\sl branch point\/} of $f$. Note: If $f = g$ (and $m > 1$), then the fiber product has at least two components, one the diagonal.

Also, $\prP^1_x\times^{\set}_{\prP^1_z}\prP^1_y$ is projective:  a closed subset of 
$\prP^1_x\times \prP^1_y$. Still, this contains \eqref{*} as a subset, so might be singular.  This is relevant, since Thm.~\ref{DS1} (DS$_1$) reduces consideration to $f $ and $g$ (not affine equivalent) with $m=n$  where  $f$ and $g$ have exactly the same branch points.  We will always use the projective {\sl normalization\/} of $\prP^1_x\times^{\set}_{\prP^1_z}\prP^1_y$,  denoting this object by $\prP^1_x\times_{\prP^1_z}\prP^1_y$. In our 1-dimensional curve case, this is the unique  nonsingular projective model of \eqref{*}.  It maps naturally to  $\prP^1_x\times^\set_{\prP^1_z}\prP^1_y$;  one-one  (an immersion) except over singular points. Yet, as with modular curves (a special case),  finding equations for the unique normalization is nontrivial.

\cite[\S 3.3.2, \S 4.2.2 and \S 4.3]{chpfund.pdf}  discuss these compactifications in much more detail, including elaborating on the following remarks.

\begin{edesc} \label{algGeom}  \item \label{algGeoma} Any closed subscheme (covered by affine pieces) of projective space is the zero set of homogeneous algebraic equations \cite[Cor. 5.16]{Har77}.
\item \label{algGeomb} The normalization of any projective variety is projective and its connected components correspond to its algebraic components: Segre's Embedding \cite[Thm. 4, p. 400]{Mum66}. \end{edesc} 

\S\ref{whatCover} reminds of cover basics,  the generality of fiber products and of their universal property \eqref{UnivPropGen}. A particular case might start with this hypothesis. Suppose a cover of nonsingular curves $\phi_W: W\to \prP^1_z$ factors through both  $f$ and $g$. 
\begin{triv} \label{UnivPropCurves} Then, $\phi_W$ factors through the fiber product $\prP^1_x\times_{\prP^1_z}\prP^1_y$. \end{triv} It notes, also, that if the varieties have dimension 1 (are curves) and they are irreducible and normal over a characteristic 0 field (so nonsingular), then any nonconstant morphism is automatically a cover. 

\subsection{From classical to modern} \S\ref{secII.2} is part of the history behind Davenport's problem, while \S\ref{secII.2b} notes two modern techniques that came from its solution. 

\subsubsection{Formulations between the 1920's and the 1960's} \label{secII.2} Equations like \eqref{*} (sometimes $f$ and $g$  are rational functions), combined with questions about solutions, say in the rationals $\bQ$, explains many  papers of that time.  Here are examples fitting this paradigm I heard from Davenport, Leveque, Lewis and Schinzel my second year of graduate school. All assumed $f$ and $g$ had coefficients in $\bQ$: $\bZ/p$ refers to the integers modulo a prime $p$.

   \begin{edesc} \label{equats} \item \label{equats1} Which equations \eqref{*} have infinitely many solutions in $\bZ$ (or $\bQ$)?
\item  \label{equats2} Schur (1921): If $f(x)=x$ in \eqref{*}, when are there infinitely many primes $p$ satisfying this: For each $x' \in \bZ/p$  there is  $y' \in \bZ/p$ satisfying \eqref{*}?
\item  \label{equats3} Davenport (1966, at Ohio State): For which equations  \eqref{*}  and almost all primes $p$ does the following hold. For each $x' \in \bZ/p$  (resp.~$y' \in \bZ/p$) there is $y' \in \bZ/p$ (resp.~$x' \in \bZ/p$) satisfying \eqref{*}. 
\item \label{equats4} Schinzel (papers from the late '50s): Which equations \eqref{*} factor into lower degree polynomials in $x$ and $y$ \cite{Sc71}? \end{edesc} 

In referring to these below, I will always assume the hypotheses hold nontrivially. For example: exclude $g(x)=f(ax+b)$ in Davenport's problem, for then the conclusion to his question is obviously \lq\lq yes\rq\rq\ if $a, b$ are in $\bQ$. Lem.~\ref{affineChange} shows we often need not assume $a,b\in  \bQ$ (for example, when $f$ is indecomposable), and that it follows automatically from $g\in \bQ[x]$. Refer to a nontrivial pair $(f,g)$ satisfying \eql{equats}{equats3} as a {\sl Davenport pair\/} (over $\bQ$). Using almost all residue class fields of a number field $K$  gives meaning to a Davenport pair over $K$.

For $f=\sum_{i=0}^m c_ix^i\in K[x]$, $K$  a field, denote $\{i >0 \mid c_i\ne 0\}$ by $I_f$. 
\begin{triv} \label{coefficRule} If $K$ has characteristic prime to $\deg(f)$, then $f(x-c_{m-1}/mc_m)$ has penultimate coefficient 0.  \end{triv} 

\S\ref{App.1a} reminds of the trace function, $\tr$,  from a representation of a group. If $G\le S_n$ then we denote the subgroup of $G$ fixing $i$ by $G(i)$. When a group has several permutation representations, distinguishing them requires more notation. 

\begin{lem} \label{affineChange} Suppose, for $a\not\in K$ and $b$ constants, $f(x), f(ax+b)\eqdef g(x)\in K[x]$. Then, $f=h(x^{k_f})$ with $h\in K[x]$ and $k_f$, the \text{\rm gcd}  of $I_f$, exceeds 1.  

Denote $a^m$ by $a'$. Assume further that  $(f, g)$ form a Davenport pair (over $K$). Then, either  $\deg(h)>1$ or, $a'x^{k_f} -1$ has a zero modulo almost all residue class fields of $K$. If $K=\bQ$, then $f$ must be decomposable. \end{lem} 

\begin{proof} Apply \eqref{coefficRule} to each of $f(x)$ and $g(x)$ to assume their penultimate coefficients are 0. The results are still affine equivalent. Also, translating by some $c\in K$ doesn't change the domain; if they started as a Davenport pair (over $K$), they remain such. 

Using \eqref{coefficRule},  the penultimate coefficient of $g(x)$ is $(mc_mb+c_{m-1})a^{m\nm 1}=0$. So, we can assume  $b=0$. Therefore, $f(ax)\in K[x]$, a statement equivalent to 
\begin{equation} \label{coefficConc1} \{a^i\mid i\in I_f\}\subset K^*.\end{equation} 

Write $k_f$ as a linearly combination $\sum_{i\in I_f} u_ii$ with the collection of $u_i\,$s relatively prime to draw these conclusions: 
\begin{triv}  $a^{k_f}\in K^*$ $\Leftrightarrow$ \eqref{coefficConc1}; and since $a\not\in K$, $f(x)=h(x^{k_f})$, and $k_f> 1$. \end{triv} 

Now assume  $(f,g)$ is a Davenport pair over $K$, but $\deg(h)=1$ and $k_f$ is prime.  
Denote a residue class field of a prime $\bp$ of $K$ by $O_K/\bp$. Being a Davenport pair implies the following for almost all $\bp$. For each $x_0\in O_K/\bp$, there is $y_0\in O_K/\bp$ with $(x_0)^{k_f}=a'(y_0)^{k_f}$. Conclude:  $a'x^{k_f} -1$ has a zero mod $\bp$ for almost all $\bp$. 

We now give a major use of  Cebotarev's theorem. We use the cover version later (see the Chebotarev discussion of \S\ref{secIII.2}). 
Assume $f$ is an irreducible polynomial (resp.~$\phi: X\to Z$ is an irreducible  cover) over a number field $K$.  
\begin{edesc} \label{irrCeb} \item  \label{irrCeba}  Then $f$ (resp.~$\phi$) has a transitive Galois (resp.~monodromy) group. 
\item  \label{irrCebb} In any transitive subgroup  $T: G\to S_n$, there is an element $\sigma$ that fixes no letter of the permutation action: $\sigma \not\in \cup_{i=1}^n G(i)$, or $\tr(T(\sigma))=0$. 
\item  \label{irrCebc} For infinitely many primes $\bp$ of $K$, $f\mod \bp$ has no zero (resp.~$\phi$ is not onto as a map on residue class fields). \end{edesc}  The polynomial version implies $a'x^{k_f} -1$ is reducible. Finally, if $k_f$ is a prime, and $K=\bQ$, then it is well-known that $a'x^{k_f} -1$ is irreducible. This contradition completes the proof of the lemma.  \end{proof}

In addition to the problems above, a  H(ilbert)'s I(rreducibility) T(heorem) variant kept appearing. Archetypal of problems unsolved at the time was this:

\begin{triv} \label{*3} For which $f$ are there infinitely many  $z' \in \bZ$ for which $f(x) - z'$ factors over $\bQ$, but it has no $\bQ$ zero (the Hilbert-Siegel Problem of  Prop.~\ref{DeFrMC})? \end{triv} 

\begin{exmp}[Davenport pair?] \label{x^8exmp} Lem.~\ref{affineChange} ended with  $m_{a,k}(x)=ax^k-1\in \bQ[x]$ with a zero mod $p$ for almost all $p$, but no zero in $\bQ$. Then, $m_{16,8}(x)=16x^8-1$ is an example. See this by factoring $m_{16.8}$ into quadratics. From multiplicative properties of the Legendre symbol:  $f(x)=h(m_{16,8}(x))$ and $g(x)=h(x^8)$ form a Davenport  pair. But, with  $g(x)=f(a'x)$, $a'\not\in \bQ$. As in Def.~\ref{excAnaloga} with $f=f_d=T_{8,d}(x)$ and $g=g_d=f_d(\sqrt{2}x)$. 
\cite[Thm]{Mu06}  uses the Legendre symbol, as above, to show  $(f_d,g_d)$, $d\in \bQ$, form a Davenport pair. He also shows for degree 8, this gives them all, up to our usual equivalence. 
Conj.~\ref{MullConj} states this example is serious. Yet, rather than suggesting D${}_1$ in \eqref{davpart} is wrong, it suggests, even if $f$ is decomposable, it  might actually be true. \end{exmp}

\subsubsection{Extrapolating from Davenport's Problem} \label{secII.2b} In treating variants of Schur's or Davenport's problems, papers of the time considered special polynomials $f$ and $g$,  concluding these problems negatively. Example: For $f$ in some specific set of polynomials, the answer to \eql{equats}{equats2} would be that none had Schur's property. 

Extending Chebotarev's theorem to function fields  was necessary to consider  Davenport's Problem in such detail. Yet, it was the mysteries of algebraic equations over number fields that guided developments, especially Riemann's approach to algebraic functions. That is, inverse results gave the greatest motivation. 
 
Sometimes the essence of algebraic equations, in two variables, is caught by the isomorphism class of the equation, represented by a point on the moduli space of curves of a given genus. Sometimes, not! For that doesn't hint at the relations (correspondences) between equations. 

Further, equations that -- with a change of variables -- have coefficients in the algebraic numbers, maybe even in $\bQ$, differ extremely from those that do not. Using zeta functions attached to Chow motives -- we can ask about their behavior when the variables assume values in, say, finite fields.  As in \S\ref{secVII.2} it is historically accurate to use Davenport's problem to illustrate this. 

Most significant for developments, was that over certain number fields there were Davenport pairs in great abundance. That is, they formed nontrivial algebraic families of such pairs. In depicting those, especially in describing efficient parameters, I ran up against how few algebraists had any  experience with a moduli problem. 

\S \ref{App.5} recounts the three families of Davenport pairs -- degrees 7, 13 and 15 -- and the  equivalences on those pairs that gave parameters describing them. These parameter spaces  each have a genus 0 curve at their core. The techniques for describing these are now so efficient they can be used for many problems. 

\subsection{Galois Theory and Fiber Products} \label{secII.3} Groups appeared little  in \S \ref{secI.1} problems  up to 1967. Yet, progress came quickly after introducing them. Here is how they enter. For simplicity assume $f$ and $g$ over $\bQ$. Each of the maps $f: \prP^1_x \to \prP^1_z$ and $g: \prP^1_y \to \prP^1_z$ has a Galois closure cover over $\bQ$, 
$\hat f: \hat X  \to \prP^1_x$ and $\hat g: \hat Y \to \prP^1_y$. 

So, they have Galois groups $^aG_f$ and $^aG_g$ -- their respective (arithmetic) {\sl monodromy\/} groups -- the automorphism groups of these covers. Indeed, the Galois closure of $f$ has a natural description. Take  (normalization of) any connected component (over $\bQ$) of the $m$-fold fiber product of $f$ minus the (fat) diagonal components \cite[\S 8.3.2]{chpfund.pdf}.

The small \lq\lq a\rq\rq\ at the left stands for a(rithmetic), and indicates  one complication.  Consider situations like Schur's or Davenport's problems, where the polynomials $f$ and $g$ are far from general. Then, an absolutely irreducible component (over $\bar \bQ$; see \S\ref{secIII.1}) of the cover $\hat X$ may have equations over a field  $\hat \bQ_f$, larger than $\bQ$.

It was standard in the literature of the time to assume $\hat \bQ_f = \bQ$. In the general problems I faced, that didn't hold. Especially in Problem \S \ref{secII.2} \eql{equats}{equats2}, and the connection of that problem to one of Serre's {\sl Open Image Theorems}.

There is also a minimal Galois cover of $\prP^1_z$ that factors through both $\hat X$ and $\hat Y$. Its group, $^aG_{f,g}$, is naturally a fiber product. Indeed, define $\hat W$ to be the largest (nonsingular) Galois cover of $\prP^1_z$, over $\bQ$, through which both $\hat f$ and $\hat g$ factor. So, there is $\hat f_w: \hat X \to \hat W$ and $\hat g_w: \hat Y \to \hat W$ factoring through the maps to $\prP^1_z$. Each automorphism $\sigma$ of $\hat X$ or $\hat Y$  induces an automorphism $^r\sigma$ of $\hat W$. (The superscript "r" stands for restriction.)

Then, $^aG_{f,g}$ is the fiber product, $$\{(\sigma_1,\sigma_2) \in {}^aG_f\times {}^aG_g \mid {}^r\sigma_1= {}^r\sigma_2 \text{ on } \hat W\}.$$ With $m$ and $n$ the respective degrees of $f$ and $g$, then $^aG_{f,g}$ naturally has permutation representations $T_f$ and $T_g$ of degree $m$ and $n$. Also, a {\sl tensor\/} representation $T_{f,g}$ of degree $m\cdot n$ on the pairs of  letters for  the two representations $T_f$ and $T_g$. 

\section{Moving from Chebotarev translation to Riemann Surfaces} \label{secIII} Brumer taught Algebraic Number Theory while Lewis was in England, Fall semester of my 2nd year. Brumer attended a course by McLaughlin on group theory and included comments on groups during our private black board discussions.

\subsection{My Choice of Thesis Topic} \label{secIII.1} In Brumer's course I learned the fiber product construction of the group of the composite of two Galois extensions of a field. His treatment of the standard (number field) Cebotarev density theorem included a version of the Chebotarev statement in  Lem.~\ref{affineChange} and using groups to interpret it.

During Lewis' algebraic curve course  (Spring 1966) my thesis topic congealed on properties of a collection of polynomials $g_1, \dots, g_t$. We always assume algebraic sets are locally closed subsets of some projective space: a quasiprojective variety. Recall an algebraic set  $X$ over a field $K$ is {\sl absolutely irreducible} if it is irreducible over $K$ and remains so over the algebraic closure of $K$. If $K$ is a number field, for all but finitely many of its primes $\bp$ we can reduce the coefficients defining $X$ and consider it as an algebraic set $X_\bp$ over the residue field. Over any field you may consider the points $X(K)$ on $X$ with coordinates in $K$.  

Use the acronym a.a. (resp.~i.m.) for {\sl almost all\/} (resp.~{\sl infinitely many\/}) primes $p$. We refer to the following statements below for a.a. and for i.m.~$p$. 

\begin{edesc} \label{*4} \item \label{*4a} Characterize a polynomial $f$ whose range on $\bZ/p$ is in the union of the ranges of $\bZ/p$ under $g_1, \dots, g_t$.  
\item \label{*4b} More generally, consider covers $f_i: X_i\to Z$, $i=1,2$ of normal varieties over $\bQ$, with $Z$ absolutely irreducible.  Characterize  that the range of $f_2$ on $X_{2,p}(\bZ/p)$ contains the range of  $f_1$ on $X_{1,p}(\bZ/p)$. \end{edesc}  The first distinguishing property of any cover (say, $f$ in \eql{*4}{*4a}) is transparently its degree. Its  {\sl monodromy group\/}  is subtler.  

To see \eql{*4}{*4b} generalizes \eql{*4}{*4a} take $f=f_1$ and  $f_2$ the natural map from the simultaneous fiber product of $g_i: \prP^1_{w_i}\to \prP^1_z$, $i=1,\dots, t$. Generalize \eqref{*2} to consider $^aG_{f,g_1,\dots,g_t}$, the monodromy formed from many fiber products, with  representations $T_f$ and $T_{g_1},\dots ,T_{g_t}$ of respective degrees $m$ and $n_1,\dots, n_t$. That generalizes --  the same fiber product construction -- to form $^aG_{f_1,f_2}$ in \eql{*4}{*4b}. 

\subsection{Precise Versions of Chebotarev's Theorem} \label{secIII.2} Chebotarev's  theorem -- for function fields over number fields -- says that \eqref{*4} implies a statement on $^aG_{f,g_1,\dots,g_t}$. We explain two possible converses for Chebotarev. Respectively, these are {\sl Monodromy Precision\/} and an {\sl RET Converse}. Much of this section is on the former, though much of the paper's remainder is on the latter. 

\subsubsection{Monodromy Precision} \label{MPres} Generally, in applying Cebotarev, you expect implications in only one direction.  Yet, for  \eqref{*4} monodromy groups are precise: a monodromy statement implies \eqref{*4}. 

MacCluer's Thesis \cite{Mac67} answered the main question of \cite{DL63}   by showing, for tamely ramified polynomial covers, that the property of being  exceptional (\S\ref{secI.0})  over a finite field is  monodromy precise. (He said it differently.) We indicate the growth of this result -- using examples from  the special cases that dominate this paper -- below Thm.~\ref{Chebp}.  \S\ref{secVII.2} elaborates on the point of this Cebotarev strengthening. Everything applies over a general number field $K$.  We simplify by taking $K=\bQ$, and using the notation of \eql{*4}{*4a}  though it applies equally to \eql{*4}{*4b}.  

Assume a component of the cover whose group is $^aG_{f,g_1,\dots,g_t}=^aG_{f,\bg}$ (the arithmetic monodromy) has definition field $\hat \bQ_{f,\bg}$.  Then, $^aG_{f,\bg}$ maps surjectively to the Galois group $G(\hat \bQ_{f,\bg}/\bQ)$.  The kernel is the geometric monodromy, $G_{f,\bg}$, of the cover.
For $\tau\in  G(\hat \bQ_{f,\bg}/\bQ)$ denote the $^aG_{f,\bg}$ coset  mapping to $\tau$ by $\tau {}^aG_{f,\bg}$. We call \eqref{*5} the {\sl monodromy conclusion}. Again, $\tr$ denotes the trace (\S\ref{App.1a}). 

\begin{thm} \label{Chebp} Assume \eqref{*4} holds for i.m. ~(resp.~a.a.) primes $p$. Then, for some (resp.~for each) coset $\tau {}^aG_{f,\bg}$, and  for each $\sigma \in \tau {}^aG_{f,\bg}$: 

\begin{triv} \label{*5} $\tr(T_f(\sigma))>0$  if and only if for some $i$, $\tr(T_{g_i}(\sigma))>0$. \end{triv} Further,  the converse holds: \eqref{*5} implies \eqref{*4}. Finally, all these statements apply directly with the field $\bQ$ replaced by $\bZ/p$. \end{thm} 

\begin{proof}[Comments]  Denote the fixed field of $\tau$ in $\hat \bQ_{f,\bg}$ by $\hat \bQ_{f,\bg}^\tau$. 
The implication \eqref{*4} $\implies$ \eqref{*5} is a combination of Cebotarev -- actually not then in the literature \cite[p.~212-13] {Fr76}, or \cite[Chap.~5]{FrJ86}$_1$ --  for number fields and for function fields.  The subtlety of the extension of constants $\hat \bQ_{f,\bg}$ not being $\bQ$ is precisely treated in \cite[\S2]{Fr74a} under the title:  \lq\lq Non-regular Analog of the Cebotarev Theorem.\rq\rq\  

The converse is from \cite[Cor.~3.6]{Fr05b}, a conclusion from pr-exceptionality (comments on \eql{DavEx}{DavExe} below), where pr stands for {\sl possibly reducible\/} cover. 

\cite{Fr05b} shows this applies for any prime $p$ satisfying these two properties: 
\begin{edesc} \label{achieveExc} \item \label{achieveExca} $\tau$ is the Frobenius element in $\hat Q_{f,\bg}$; and 
\item  \label{achieveExcb} the subgroup of ${}^aG_{f,\bg}$ fixing $\hat Q_{f,\bg}^\tau$ naturally equals  the analog of ${}^aG_{f,\bg}$ over $\bZ/p$ obtained by reducing all polynomials mod $p$. \end{edesc}  \cite[Lem.~1]{Fr74a} says \eql{achieveExc}{achieveExcb} holds for a.a.~ $p$  
with  $\tau$ the Frobenius in $G(\hat \bQ_{f,\bg}/\bQ)$. (There are i.m.~such $p$ by  Chebotarev's theorem.)  This requires avoiding a potentially large -- but finite -- set of primes, including those dividing denominators of coefficients, or for which some polynomial becomes inseparable. 
\end{proof} 

Now consider these special cases of \eqref{*4} as concluded by Thm.~\ref{*5}.
\begin{edesc} \label{DavEx} \item \label{DavExa} All $f_1\,$s work  in \eql{*4}{*4b}: The range of $f_2$ is the complete set, $Z_p(\bZ/p)$, of $\bZ/p$ points on $Z$ for i.m.~ (resp.~a.a.) $p$. 
\item \label{DavExb} {\sl Exceptional functions}: In \eql{*4}{*4b}, $f_1$ is trivial (degree 1), and $X_2$ is absolutely irreducible: For i.m.~$p$, the range of $f_2$ is $Z(\bZ/p)$. 
\item \label{DavExc} {\sl Exceptional polynomials}: $t=1$, and $f$ is trivial in \eql{*4}{*4a}: For i.m.~$p$, the range of $g=g_1$ on $\bZ/p\cup \{\infty\}$ is $\bZ/p\cup \{\infty\}$. 
\item \label{DavExd} {\sl pr-exceptional functions}: Exactly the same as \eql{DavEx}{DavExc}, except we allow $X_2$ to have more than one component. 
\item \label{DavExe} {\sl Davenport pairs}:  For a.a.~$p$, \eql{*4}{*4b} holds as stated, but it also holds after switching  $f_1$ and $f_2$. \end{edesc} 
 
\subsubsection{Monodromy Precision Comments} \label{MComm} We comment on the cases of  \eqref{DavEx} using the notation  $\bF_q$ for the finite field of cardinality $q=p^t$ for some prime $p$. 

{\sl Comments on \eql{DavEx}{DavExa}}: If $X_2$ is absolutely irreducible, then it remains absolutely irreducible for almost all  $p$: a case of \cite[Lem.~1]{Fr74a}. From \eql{irrCeb}{irrCebb}  the geometric monodromy group $G_{f_2}$ contains some $\sigma$ fixing no letter of the permutation set.  

For primes  $\bp$ of $K$ where $\sigma$ is in the arithmetic coset  of \eqref{*5}, $\sigma$ (according to the monodromy conclusion) prevents $f_2$ from being an onto map over $\sO_K/\bp$. So, $X_2$ has several components if \eql{DavEx}{DavExa} holds for a.a.~$\bp$. Allowing $X_2$ to  have several components -- to be {\sl  p(ossibly)r(educible)-exceptional\/} -- put Davenport pairs and exceptional covers (comments on \eql{DavEx}{DavExd} and \eql{DavEx}{DavExe}) under one umbrella  \cite{Fr05b}. 

{\sl Comments on \eql{DavEx}{DavExb}: Exceptionality sets}: Suppose   $\phi: X\to Z$ is a cover of absolutely irreducible varieties over  $\bF_q$. Denote the extension of constants field in the arithmetic monodromy, $^a G_\phi$ (its corresponding representation is $T_\phi$), by $\hat \bF_q$. 

Denote  the coset in  $^a G_\phi$ that restricts to the  $q$-power map, $\Fr_q$, on $\hat \bF_q$  by $\Fr_q{}^aG_\phi$. If one of the notions of \eqref{excChar} -- all equivalent according to \cite[Cor.~3.6]{Fr05b} -- hold, call $\phi$ an $\bF_q$ {\sl exceptional cover}.  
\begin{edesc} \label{excChar} \item \label{excChara} $\phi: X(\bF_{q^t})\to Z(\bF_{q^t})$ is onto (resp.~injective) for infinitely many $t$. 
\item \label{excCharb} The fiber product $X \times_Z X$ with the diagonal component removed has no absolutely irreducible $\bF_q$ components. 
\item \label{excCharc} With $\sigma$ running over $\Fr_q{}^aG_\phi$, then, $\tr(T_\phi(\sigma))> 0$ (resp. $\tr(T_\phi(\sigma))\le 1$).  \end{edesc} 

Each of \eql{excChar}{excChara} and \eql{excChar}{excCharc} are a pair of characterizations.  The former says $\phi$ is one-one and onto $\bF_{q^t}$ points for infinitely many $t$. The latter says $\tr(T_\phi(\sigma))=1$ for all  $\sigma$ extending the Frobenius.  With $(Z,\bF_q)$ fixed, \cite[Prop.~4.3]{Fr05b} says the collection of exceptional covers of $Z$ over $\bF_q$ form a category with fiber products. 

We explain. If $\phi$ and $\phi':X'\to Z$ are two such covers, then the  fiber product $X\times_Z X'$ has exactly one absolutely irreducible $\bF_q$ component, though it may have many $\bF_q$ components. That absolutely irreducible component is the fiber product of $\phi$ and $\phi'$ in this category. (Note:  \S\ref{roleFlatness} says, if  $\dim(Z)>1$, then we may have to extend the notion of cover.)

Return to the notation of  \eql{DavEx}{DavExb}. We say $f_2$ is exceptional (over $\bQ$; but it applies to any number field), if there are i.m.~$p$ with -- upon applying \eqref{achieveExc} -- the reduction of $f_2$ mod $p$ exceptional as above. Denote the set of such  {\sl primes of exceptionality\/} for $f_2$ by $E_{f_2}$. There may be primes $p$ for which $f_2$ on $\bZ/p$ points is onto, but they don't fit the \eqref{excChar} criterion of exceptional. Still, if $f_2$ on $\bZ/p$ points is onto for infinitely many $p$, all but finitely many will be in $E_{f_2}$. That is, even using Chebotarev roughly, for $p$ large the ontoness forces the monodromy statement of \eqref{*5},  equivalent -- in this case -- to the other criteria of \eqref{excChar}.

{\sl Comments on \eql{DavEx}{DavExc}: }: Since $\prP^1_w$ is absolutely irreducible, the comment on \eql{DavEx}{DavExa} says $E_g$ {\sl excludes\/} infinitely many primes.  Assume $g$ has reduction mod $p$ giving a tamely ramified polynomial with \eqref{achieveExc} holding. Then, \cite{Mac67} showed the converse statement \eqref{*5}, but in the fiber product form \eql{excChar}{excCharb}. 
 
Yet, Thm.~\ref{Chebp} says tame ramification, even that $g$ is a polynomial, was unnecessary. Comments on  \eql{DavEx}{DavExd} start the discussion on explicit (algebraic) equations, versus avoiding equations as in \S\ref{secV}, \S\ref{writeEquats} and  \S\ref{App.1b}.

Suppose $g:\prP^1_w \to \prP^1_z$ is an exceptional rational function over a number field $K$.  Suppose, further, $\ell_1,\ell_2\in \PGL_2(\bar \bQ)$ (linear fractional transformations), but   $\ell_1\circ g \circ \ell_2^{-1} \eqdef g_{\ell_1,\ell_2}\in K(w)$. We say $g$ and  $g_{\ell_1,\ell_2}$ are M\"obius equivalent  (over $\bar \bQ$). If $\ell_1,\ell_2\in \PGL_2(K)$, then clearly  $g_{\ell_1,\ell_2}$ is also exceptional. This trivial production of new exceptional covers encourages regarding  
M\"obius equivalence classes over $K$ as essentially the same.  

{\sl Comments on \eql{DavEx}{DavExd}}: Isn't \eql{excChar}{excCharb} a pleasanter  characterization of exceptionality  than using group theory? Yet, it was groups that precisely characterized  tamely ramified exceptional polynomials:  \cite{Fr70}, or \cite[\S5]{FrGS93}, or \cite[Prop.~5.1]{Fr05b}.  

In Davenport's problem the efficiency of using groups is even more striking. As \S\ref{secII.3} reminds, anything using the Galois closure of a cover is about fiber products. Still, there is no simple analog of \eql{excChar}{excCharb} for pr-exceptional covers. In analogy for exceptional covers, there is a set of primes,  $E_{f_2}$, for pr-exceptionality.  

{\sl Comments on \eql{DavEx}{DavExe}}:  Consider the natural fiber product projections  
$$X_1\times_ZX_2\mapright{\pr_i} X_i, 1=1,2.$$  A special case of \cite[Cor.~3.6]{Fr05b} tells us that $(f_1,f_2)$ form a Davenport pair if and only if both $\pr_i\,$s  are pr-exceptional covers with exceptionality sets consisting of a.a.~$p$. This is what gives the Monodromy Converse for Davenport pairs. 

\subsubsection{Using equations and Chebychev conjugates}  \S\ref{depSch} -- on displaying Davenport pairs -- deepens our distinction  between using branch cycles  (\S\ref{secVI.4}) and using equations to describe covers. Prop.~\ref{schurConj} gives the result/conjecture that attracted so much number theory attention to Cheybchev polynomials. This allows me a preliminary contrast of my techniques with a traditional use of explicit equations. 

A functional equation defines the $n$th Chebychev polynomial, $T_n$: 
\begin{equation} \label{Chebychev} T_n((x+1/x)/2)=(x^n+x^{-n})/2.\end{equation}
For $a\in \bF_q^*$ and  $a=u^2$, $u\in \bF_{q^2}^*$, denote multiplication by $u$ by $m_u$. 

\begin{defn}[Dickson analogs of $T_n$] \label{excAnaloga} Convolution by $m_u$ gives a {\sl Chebychev conjugate\/}  $T_{n,a}=m_u\circ T_n \circ m_{u}^{-1}$, scaling the branch points from $\pm1$ to $\pm u$.  Chebychev conjugates are constants times {\sl Dickson polynomials\/} \cite[Prop.~5.3]{Fr05b}. \end{defn}

\begin{prop}[Schur's 'Conjecture'] \label{schurConj}   With $K$ a number field, the $f\in \sO[x]$ for which $E_{f,K}$ is infinite are compositions with maps $a\mapsto ax +b$ (affine) over $K$ with polynomials of the following form running over odd primes $u$: 
\begin{triv} \label{excType} cyclic $x^u$ or Chebychev conjugates of $T_u$, $u> 3$ \cite[Thm.~2]{Fr70}. \end{triv}  \end{prop} 

\cite[p.~49]{Fr10} -- essentially Lem.~\ref{ChebChar}  characterizing the Chebychev conjugates -- used the name Chebychev for all Chebychev conjugates (instead of Dickson analogs).   \cite{LMT93} is dedicated to using explicit expressions for Chebychev and closely related functions.  If exceptionality is important, it behooves us to know precisely over what finite fields Chebychev conjugates are exceptional. Our comments on Lem.~\ref{monPrecCheb} are an example of monodromy precision close to MacCluer's motivation in \cite{Mac67} (before the proof of Prop.~\ref{schurConj}). 

\begin{lem} \label{monPrecCheb} Assume $n$ is odd (and prime to $p$), and $a\in \bF_q^*$, with $q$ odd. Then, the Chebychev conjugate  $T_{n,a}$ is exceptional if and only if $(n,q^2-1)=1$. 
\end{lem}

\begin{proof}[Comments on two different types of proof]  \cite[Lem.~13]{Fr70} for $a=1$, but there is a typo in the statement: $N(\bp)-1$ should be $N(\bp)^2-1$. The proof of sufficiency of $(n,q^2-1)=1$ for all $a$ -- the first part of \cite[Thm.~3.2]{LMT93} -- is exactly the same. Except rather than stating this gives exceptionality for these primes, they say only that it  the Chebychev conjugate is a {\sl permutation polynomial\/} -- it maps one-one -- on $\bF_q$. \cite[p.~39]{LMT93} does the converse -- if $(n,q^2-1)=d > 1$, then a Chebychev conjugate is not exceptional/permutation -- based also on  \eqref{Chebychev}. 

We do the converse using the monodromy precise characterization in \eql{excChar}{excCharb}.  
From the $T_{n,a}$ characterization of Lem.~\ref{ChebChar}, we know $\prP^1_x\times_{\prP^1_z} \prP^1_x\setminus \Delta$ consists of $\frac{n\nm1}2$ absolutely irreducible components of degree 2 over $\prP^1_z$. With $\zeta_n$ a  primitive $n$th root of 1 over $\bF_q$, each component has definition field the symmetric functions in $U_j=\{\zeta_n^j,  \zeta_n^{-j}\}$. As  $q^2-1\equiv 0 \mod d$,  the $q$th power map -- Frobenius -- acts as either $+1$ or $-1$ on the elements of $U_{n/d}$: $U_{n/d}\mapsto \{\zeta_n^{qn/d},  \zeta_n^{-qn/d}\}=U_{n/d}$. So, a component corresponding to $U_{n/d}$ is defined over $\bF_q$, and $T_{n,a}$ is not exceptional.  \end{proof} 

\begin{rem}[Continuing on Lem.~\ref{monPrecCheb}]  The proof of $(n,q^2-1)=1$ being exact for Chebychev conjugate exceptionality is what I gave as a referee of  \cite{Mat84}. That was to algorithmically, from degrees, find which compositions of cyclics and Chebychev conjugates in Prop.~\ref{schurConj} are exceptional.  Maybe the  version from \cite[Thm.~3.2]{LMT93} is  more comforting than using monodromy precision. 

Yet, could equation manipulation work on the exceptional primes arising from Serre's Open Image Theorem \S\ref{limitDih}, as does monodromy precision \cite[\S6.2, esp.~Prop.~6.6]{Fr05b}? There is a structure to exceptionality that imediately differentiates it from accidents when $f$, though not exceptional, might  permute elements of $\bF_q$. I could find no reference to exceptionality in \cite{LMT93}. \cite{FrL87} -- on forming higher dimensional Chebychev analogs, so exceptional covers, using Weil's restriction of scalars -- indicates I did try to communicate about such matters. 
\end{rem} 

\subsubsection{RET Converse}  \label{RETConv} We return to \eql{DavEx}{DavExe}, using notation of   \eql{*4}{*4a}:  $f=f_1$, $g=f_2$, with $(f,g)$ a Davenport pair of polynomials. 

C$_1$: {\sl Formulating a geometric converse}: A converse of the group version  \eqref{*5} might  ask this. Given any group statement of this ilk, are there $(f,g)$ that produce the group conditions. This question is appropriate far beyond Davenport's problem. 

C$_2$: {\sl Formulating an arithmetic converse}: Statement \eqref{*5} has a group version about an arithmetic monodromy. A converse might give two groups $^aG$ and $G$ satisfying a statement like \eqref{*5}, then ask: Are there covers realizing these groups as their arithmetic/geometric monodromy  over some number field? 

R(iemann)'sE(xistence)T(heorem), \S \ref{secVI.4}, can invert these statements. Constraining permutation representations to produce polynomial covers (or rational functions) is in the group theory, through R(iemann)-H(urwitz) \eqref{*10}.

We model our method for deciding over what number fields the arithmetic inversion is achievable on the two pieces D${}_1$ and D${}_2$ to Davenport's Problem \eqref{davpart}. Other related problems, like Serre's OIT,  show how a few precise problems can coral considerable progress, despite the surrounding unknown territory.  

\subsection{Meeting UM Faculty and going to $\infty$} \label{secIII.4} The graduate student population was over 200 at UM in those years. I later realized that the department was large, too, compared to other departments in which I ever held a position. Therefore, seminars -- not driven by the research  of a resident faculty -- often started with many attending, but dropped rapidly each week.

\subsubsection{How fiber products and other tools arose} 
I  learned fiber products at UM from a seminar on Diudonne's version, EGA, of Grothendieck's writing, summer 1965. From the 50$^+$ who first showed, soon there was just Brumer, Bumby and me. I recall practicing sheaves, direct limits and projective limits especially from a famous Grothendieck paper -- Tohoku-- under their tutelage. Bumby, especially,  guided my intuition on much profinite homological algebra. 

Lewis arranged for my attendence at two Bowdoin college NSF-funded summers. Eight weeks each on Algebraic Number Theory (summer of 1966) and Algebraic Geometry (summer of 1967).  Both summers I learned everything put in front of me.  I also learned I would be subject to pejoratives for not having the background prevalent then at Harvard, MIT or Princeton. It never intimidated me.

Brumer left for Columbia at the start of my 3rd year. Imitating Brumer, I engaged McLaughlin directly in blackboard discussions when I could catch him, about permutation representations.  Roger Lyndon and I lectured in his seminar on Discontinuous groups acting on the upper half plane. Also, I read notes of Brumer on modular curves from lectures of Gunning. As with theta functions, this became my hidden tool,  augmented sharply by two years around Shimura while I was at the I(nstitute for)A(dvanced)S(tudy) 1967-1969. 

\subsubsection{Grabbing a thesis and learning from it} \label{secIII.3} I was aware, by Summer 1966 that the implication \eql{*4}{*4a} $\implies$ \eqref{*5} would receive little regard  for these reasons. 

\begin{enumerate} 
  \item  It said nothing about the polynomials involved, not even suggesting what, of significance, one might say.
 \item The problem didn't register with the MIT-Princeton-Harvard students at the 1966 Bowdoin Conference on Algebraic Number Theory.  \end{enumerate} 

If a mature algebraic geometer had cued my next step -- say Artin or Mumford,  I later knew both  -- it wouldn't have resonated. Through, however, my student eyes it opened a new way of thinking. Later I realized it was a stride even for Riemann. Lefschetz admitted he finally understood Picard from something similar. 

Yet, isn't this elementary?: I looked at $\infty$,  Christmas morning 1966, at a time I despaired at finding any structure to Problem \eqref{*4}. I saw a finger circling $\infty$ on the Riemann sphere, clockwise (so, unlike their use by many, my loops go clockwise around points to this day), and then coming back to a basepoint -- at my feet. 

Here's what it meant for the values of a polynomial $f:  \prP^1_x \to  \prP^1_z$. You knew for certain one element, $\sigma_\infty$, in $G_f$ (and so in $^aG_f$): an $n$-cycle coming from the cover totally ramifying over $\infty$.  The proof of Prop.~\ref{Gusic} displays $\sigma_\infty$ in an  elementary way. Recall, $\infty$ was not initially considered a value of  $f$, but that is irrelevant.  

\subsubsection{Combining data at $\infty$ with Chebotarev} \label{secIII.5} That finger circling $\infty$ corresponded to a path on the punctured sphere. So, in considering \eql{*4}{*4a}  it corresponds to a generator, $\sigma_\infty$, for the inertia group over $\infty$ for the fiber product of all covers given by $f$ and the $g_i\,$s. In each corresponding permutation representation $\sigma_\infty$ appears  respectively as an $m$-cycle or an $n_i$-cycle.

\begin{prop} \label{useInfty}  Apply the conclusion of Chebotarev  in \eqref{*5}: With $N$ the least common multiple of the $n_i\,$s, $m$ divides $N$. In particular, in Davenport's problem \eql{equats}{equats3} the degrees of a Davenport pair $( f, g)$ must be the same. \end{prop} 

\begin{proof} 
The element $\sigma_\infty^N$ fixes every letter in $T_{g_i}$ (corresponding to $g_i$). So, from  \eqref{*5},  $T_f(\sigma_\infty^N)$ must fix something. Yet, unless $m$ divides $N$, as $T_f(\sigma_\infty^N)$ is an $m$-cycle to 
the $N$-th power; it fixes nothing. This contradiction shows the result. \end{proof} 

From here on we take this common degree of a Davenport pair as $n$. In fact, there is a stronger  conclusion, which Lem.~\ref{pullbackcomps} explains more fully. 

\begin{thm}[DS$_1$] \label{DS1} Suppose $f$ and $g$ nontrivially satisfy Davenport's hypothesis. Then their Galois closure covers are the same \cite[Prop.~2]{Fr73a}. \end{thm} 

\subsection{Double transitivity versus primitivity} \label{secIV} Unless you are a group theorist, or have,  through a particular problem met groups seriously, then you likely know finite groups only through their permutation representations. So, you wouldn't know there is an intimate relation between primitive groups and simple groups (\S \ref{App.2}) -- excluding primitive affine groups (\S \ref{App.1a}), which may resist any classification. 

I didn't know these things, which came partly from \cite{AOS85}, when I started either. I luckily could skirt the easier edge of the doubly transitive/primitive divide. This section runs lightly over \cite{Fr70} to   review how the primitive group property arose early. The more intense analysis of \cite{Fr73a} starts in \S\ref{secV}.   

\subsubsection{Translating Primitivity} \label{secIV.1} The monodromy group ${}^aG_f$ of a cover $f: X \to \prP^1_z$ over a field $K$ is {\sl primitive\/} if and only if the cover does not properly factor  through another cover (over $K$). Also, ${}^aG_f$ is doubly transitive if and only if the fiber product $X\times_{\prP^1_z}X$ has exactly two irreducible $K$ components (one is the diagonal). 

When $X=\prP^1_x$, primitive means $f$  doesn't decompose (over $K$) as $f_1\circ f_2$ with both $\deg(f_i)\,$s  exceeding 1.  Doubly transitive translates as follows:  $(f(x)-f(y)/(x-y)$ is, after clearing denominators by multiplying -- with $h$ the denominator of $f$ -- by $h(x)h(y)$, an irreducible polynomial in two variables over $K$. 
Galois theory translates these  respective statements as conditions on ${}^aG_f$ under the permutation representation $T_f$. For a group $G$ under a degree $n$ representation $T$,  $G(i)$ indicates the subgroup of $G$ fixing $i$. 

\begin{edesc} \label{primdt} 
    \item \label{primdta} $(G,T)$ is {\sl Primitive}: No group lies properly between $G$ and $G(1)$.
    \item \label{primdtb} $(G,T)$ is {\sl Doubly Transitive}: $G(1)$ is transitive on $\{2,\dots,n\}$. \end{edesc} 

If $G_f$ is primitive, then so is $^aG_f$, but the converse does not in general hold. Still, we have the following. Denote the characteristic of $K$ by $\text{Char}(K)$. 

\begin{lem}[Polynomial Primitivity]  \label{decompPoly} If $f\in K[x]$, of degree prime to $\text{\rm Char}(K)$, decomposes over $\bar K$, then it decomposes over $K$ \cite[Prop.~3.2]{FrM69}. In this case, if it is indecomposable, then $G_f$ is doubly transitive  unless it is affine equivalent over $\bar K$ to a cyclic  ($x^n$) or Chebychev polynomial (as in \eqref{Chebychev}) \cite[Thm.~1]{Fr70}. \end{lem}  

In Schur's Conjecture we can revert to primitivity quickly. A composite of polynomials gives a one-one map on a finite field, if and only if each does. Polynomial Primitivity \ref{decompPoly} then reverts to the case $G_f$ (the geometric group) is primitive. Two famous group theory results from early in the 20th century help immensely.

\begin{itemize}
\item {\sl Schur}: If $G_f$ is primitive and $n$ is composite, since $G_f$   contains an $n$-cycle under $T_f$, it must be doubly transitive.
\item {\sl  Burnside}: If $n$ is a prime, and $G_f$ is not doubly transitive, then it is a subgroup of the semi-direct product $\bZ/n \xs (\bZ/n)^*$ (\S\ref{App.1a}). 
    \end{itemize}
Lem.~\ref{ChebChar} gives the branch cycle characterization of Chebychev polynomials, an easy forerunner of the branch cycle characterization of Davenport pairs as in \S\ref{secVI.5}. 

\subsubsection{Group Theory in Grad School} \label{secIV.2} After 35 years of evidence that we know all simple groups, unless a permutation group is primitive, even the classification isn't so helpful (\S\ref{secVII.3}). Still,  primitive groups aren't \lq\lq simple\rq\rq\ ( irony intended). 

Richard Misera, a fellow graduate student -- I never saw him again after getting my degree -- was studying with Don Higman. After once seeing me discuss the distinction between permutation representations and group representations with McLaughlin he volunteered an  example  that became a powerful partner when I was ready to solve Davenport's Problem (\S \ref{secV.3}).

Soon after graduate school, I knew enough to solve Schur's Conjecture (\S\ref{secI.2}). Still, it was John Smith, whom I thought I saw by accident at IAS -- he actually came to discuss a problem with me -- who told me of Schur's and Burnside's Theorems. Smith was the 3rd  (and last, including MacRae and Schinzel) affiliate of Michigan  during my graduate years with whom I wrote papers (in each case two).

My Erd\"os number is 2 because Schinzel's is 1. 

\section{Equation properties without writing equations} \label{secV}  Rare among algebraic equation papers, even those using the monodromy method, solving Davenport's problem used general principles, not equation manipulation. For a Davenport pair, $(f,g)$, list the zeros $x_i$ of $f(x)-z$ (resp.~$y_i$ of $g(y)-z$), $i=1,\dots,n$, in an algebraic closure of $K(z)$. Do a penultimate normalization: change $x$ to $x+b$, $b \in K$, so the coefficient of $x^{m\nm 1}$ is 0 (similarly for $g(y)$).

\subsection{A linear relation in Davenport's problem} \label{secV.1}  DS$_1$ (Thm.~\ref{DS1}) says $$K(x_i, i=1,\dots,n)= K(y_i, i=1,É,n).$$ Yet, \eql{DavRes}{DavResa} is an even stronger relation. \cite[Thm.~1]{Fr73a} gives \eql{DavRes}{DavResa} and \eql{DavRes}{DavResc} with the converse statement in \eql{DavRes}{DavResb} a special case of Thm.~\ref{Chebp}. 

\begin{thm}[DS$_2$] \label{DS2} Assume $f$ and $g$ nontrivially satisfy Davenport's hypothesis, with f  indecomposable.
\begin{edesc} \label{DavRes} 
\item \label{DavResa} Then, $T_f$ and $T_g$ are inequivalent permutation representations of \\ $^aG_f = {}^aG_g$. Yet, they are equivalent as group representations.
\item \label{DavResb} Further, the converse holds: Such $T_f$ and $T_g$, equivalent as representations, imply $f$ and $g$ satisfy Davenport's hypothesis; and (for a.a.~$p$) $f $ and $g$ assume each value mod $p$ with exactly the same multiplicity.
\item \label{DavResc} Finally, since $f$ is indecomposable, so is $g$ and \eql{DavRes}{DavResa} is equivalent to $f(x)-g(y)$ being reducible (Shinzel's problem, \eql{equats}{equats4}). \end{edesc} \end{thm}

What DS$_2$ says is that $x_i$ is a sum of distinct $y_j\,$s times a nonzero element $a \in K$. With no loss, take $a=1$, and write

\begin{triv} \label{*7} $x_1=y_1+ y_{\alpha_2}+\cdots +y_{\alpha_k}$,  with $2 \le  k \le (n-1)/2$ (because the complementary sum of $y_i\,$s now works as well). \end{triv}

Let $f(x)$ and $g(y)$ be rational functions over a field $K$ (assume $\text{Char}K=0$, or that the covers given by $f$ and $g$ are separable). Suppose  $f$ (resp.~$g$)  decomposes as $f_1\circ f_2$ (resp.~$g_1\circ g_2$). Write the projective normalization of the fiber product of the covers $(f,g)$ (resp.~$(f_1,g_1)$) as $W=\prP^1_x\times_{\prP^1_z}\prP^1_y$ (resp.~$W_1=\prP^1_{u}\times_{\prP^1_z}\prP^1_{v})$: $W$ naturally maps surjectively to $W_1$. From \eql{algGeom}{algGeomb} the irreducible factors of $f(x)-g(y)$ (resp.~$f_1(u)-g_1(v)$) correspond one-one with the connected components of $W$ (resp.~$W_1$).  The 1st sentence of Lem.~\ref{pullbackcomps}  says, in the Zariski topology, the image of a connected space is connected. Result \eql{DavRes}{DavResc}  is  geometric. The rest of  Lem.~\ref{pullbackcomps} is a preliminary to it from  \cite[Prop.~2]{Fr73a}.  

\begin{lem} \label{pullbackcomps} Each irreducible factor of $f_1(u)-g_1(v)$ is the image of one or more irreducible factors of $f(x)-g(y)$.  Further, if $f(x)-g(y)$ does factor, then you can choose $(f_1,g_1)$ so the following holds.  \begin{edesc}  \label{schcond} 
    \item  \label{schconda}  The irreducible factors of $f(x)-g(y)$ correspond one-one with the irreducible factors of $f_1(u)- g_1(u)$; and
    \item  \label{schcondb}  the Galois closure covers of $f_1$ and $g_1$ are the same. \end{edesc} 
\end{lem} 

The end of \S\ref{secI.0} notes many papers quote \cite[Prop.~2]{Fr73a}. Also, as prior to Lem.~\ref{indecompLem}, the original proof works far more generally than those quoters realize. For rational functions, however, \eqref{*7}  won't hold without that $n$-cycle;  you can't even say $\deg(f_1)=\deg(g_1)$. Classifying variables separated factorizations was Schinzel's Problem, not Davenport's. Their mathematical common ground appears to have been their interest in variables separated equations.

They had not considered the equivalence of their problems for the case $f $ is an indecomposable polynomial. They aren't equivalent without the indecomposable assumption. All attempts to write equations for Davenport pairs, especially \cite{CoCa99} (see \S\ref{writeEquats}), used Schinzel's factorization condition. 

Below we  denote the letters of $T_f$ (resp.~$T_g$) by $x_i$ (resp.~$y_i$), $i=1,\dots,n$. Also, $G(x_i)$ is the stabilizer in $G$ of $x_i$. Rem.~\ref{dropIndec} doesn't even assume $n=m$. 

\begin{rem}[Davenport without $f$ indecomposable] \label{dropIndec} \cite[Lem.~3]{Fr73a}, used in  \S\ref{SchinzG}, does not assume $f$ is indecomposable, or even that $f$ and $g$ are polynomials. Suppose $(f,g)$ is a (nontrivial) Davenport pair, so $$T_f(\sigma)> 0 \Leftrightarrow T_g(\sigma)>0, \text{ for each } \sigma\in G \text{ the Galois closure group}.$$ Then, $f(x)-g(y)$ is reducible, or else,  $G(x_1)$ is transitive on $\row y n$. But, then,  conjugates of $G(x_1)\cap G(y_1)=H$ under $G(x_1)$ would cover $G(x_1)$. This contradicts that conjugates of a proper subgroup of $G$  can't cover $G$. \end{rem} 

\subsection{Difference Sets and a Classical Pairing} \label{secV.2} People who like cyclotomy (both Gauss and Davenport did) see difference sets in many situations. The kind that arises in this problem is special (cyclic),  though it is an archetype.

Normalize the naming of $\row x n$ and $\row y n$ in $T_f$ and $T_g$  so that $\sigma_\infty$ (\S\ref{secIII.5}) cycles the $x_i\,$s (and the $y_j\,$s) according to their subscripts. We now combine double transitivity and the action of $\sigma_\infty$ on both sides of \eqref{*7}. From this we see how  the definition of difference set arises. The proof of Prop.~\ref{DS3} includes a shorter proof of  \cite[Lem.~4]{Fr73a}, and a completely different approach to  \cite[Lem.~5]{Fr73a}. The latter included the statement I alluded to from Storer (Prop.~\ref{storer}).  

\begin{prop} \label{DS3} In the nonzero differences from $\sD_1=\{1,\alpha_2,\dots,\alpha_k\} \mod n$  each integer, $\{1,\dots, n\nm1\}$, appears exactly $u=k(k-1)/(n-1)$ times.  Further, writing the $y_i\,$s as expressions in the $x_j\,$s gives the attached different set (up to translation) as $\sD_1$ multiplied by -1.\end{prop}

\begin{proof}   Acting by  $\sigma_\infty$ on $\sD_1$ -- translating subscripts -- produces $\sD_i$,  $i=1,\dots,n$. The permutation action of $G_f$ gives a representation equivalent to $T_f$. The number of times an integer $u \mod n$ appears as a (nonzero) difference from $\sD_1$ is the same as the number of times the pair $\{1, u+1\}$ appears in the union of the $\sD_i\,$s. That is, you are normalizing its appearance as a difference where the first integer is a 1. Double transitivity of $G_f$ is equivalent to transitivity of $G_f(1)$ on $2,\dots,n$. So, the count of appearances of {1, u+1} in all the $\sD_i\,$s is independent of $u$.

Now consider, as in the last sentence, writing the $y_i\,$s in terms of the $x_j\,$s. To do so form a classical $n\times n$  incidence matrix: $I_{x,y}$: rows consist of 0s and 1s with a 1 (resp. 0) at $(i,j)$ if $y_j$ does (resp.~not) appear in $x_i$ (according to the translate of subscripts on \eqref{*7}). Then, applying $I_{x,y}$ to the transpose of  $[y_1\ \dots\ y_n]$ (so it is a column vector) gives the column vector of the $x_i\,$s. Denote the transpose of $I_{x,y}$ by $^\tr I_{x,y}$. From the difference set definition, notice:
$$^\tr I_{x,y} \times  I_{x,y}  = I_{x,y}  \times  {}^\tr I_{x,y}= (k\nm 1) I_n + u1_{n\times n},$$ with $I_n$ the $n\times n$ identity matrix, and $1_{n\times n}$ the matrix having 1s everywhere.

Apply both sides to the transpose of $ [y_1\,\cdots\, y_n]$  to conclude the matrix $^\tr I_{x,y}$ has rows giving the difference set attached to inverting the relation between the $x\,$s and $y\,$s. Now look at the last column of $I_{x,y} $. A 1 appears at position $j$ if and only if row 1 has a 1 at column $n-j+1$. That is, $\mod n$, column $n$ is -1 times row 1 translated by 1. That concludes the last line of the proposition. \end{proof} 

On numerology alone, we may consider which triples $(n,k,u)$ from Prop.~\ref{DS3}  afford difference sets. These are the only possibilities up to n=31:

\begin{equation}\begin{array}{rl} \label{*8}  & (7,3,1), (11,5,2), (13,4,1), (15,7,3), (16,6,2), (19,9,4), (21,5,1),  \\ &(22,7,2), (23,11,5), (25,9, 3), (27,13,6), (29,8,2), (31,6,1).\end{array} \end{equation} 

I eliminated the cases $n = 22$, 23 and 27 with the Chowla-Ryser Thm., which I discovered in \cite[Thms. 3, 4 and 5]{Ha63}. It says, for $n$ even (resp.~odd),  existence of a difference set implies $k-u$ is a square (resp.~$z^2=(k-u)x^2+(-1)(n-1)/2y^2$ has a nontrivial integer solution). Hall's book suggests Chowla-Ryser is \lq\lq if and only if\rq\rq\ for existence of a difference set. Still, we now know  for sure, if there were such a converse, it would not  produce a difference set in a doubly transitive design because we know  Collineation Conjecture  \ref{collConj} is true.

The next section shows how we guessed which groups  -- and conjugacy classes -- arose as monodromy of Davenport pairs: Problem D${}_2$ in \eqref{davpart}. This appearance of projective linear groups, combined with  Riemann-Hurwitz, shows why we stopped the list of  \eqref{*8} with $n=31$. This was the first inkling of the Genus 0 Problem. 

\subsection{Misera's example (sic)} \label{secV.3}  Take a finite field $\bF_q$: $q=p^t$ for some value of $t$, $p$ a prime. For any integer $v \ge 2$, consider $\bF_{q^{v+1}}$ as a vector space $V$ over $\bF_q$ of dimension $v+1$, so identifying it with $(\bF_q)^{v+1}$. 
The projective linear group, $\PGL_{v+1}(\bF_q)=\GL_{v+1}(\bF_q)/(\bF_q)^*$, acts on the lines minus the origin in $(\bF_q)^{v+1}$: on the points of projective $v$-space, $\prP^v(\bF_q)$. Take $n=(q^{v+1}-1)(q-1)$. 

Conclude:  $\PGL_{v+1}(\bF_q)$ has two (inequivalent) doubly transitive permutation representations, on lines and on hyperplanes. Yet, these representations are equivalent as group representations by an incidence matrix -- as in the proof of Prop.~\ref{DS3} -- that conjugates one representation to the other. 

Finally, here is what Misera told me. Apply Euler's Theorem to produce a cyclic generator, $\gamma_q$, of the nonzero elements of $\bF_{q^{v+1}}$. Let $\gamma_q$ act by multiplication on $\bF_{q^{v+1}}$. It induces  (as does $(\gamma_q)^{q-1}$)  an $n$-cycle  in $\PGL_{v+1}(\bF_q)$ acting on $\prP^v(\bF_q)$. 

Misera's example allowed me to produce examples fulfilling  Thm.~\ref{DS2}. At the end of my first year at IAS I took the following step: 

\begin{thm}[DS$_4$] \label{DS4} \cite[p.~134]{Fr73a} writes difference sets for $$\begin{array}{rl} &n=7=1\np 2\np2^2, 11, 13=1\np3\np 3^2, 15=1\np 2\np 2^2\np 2^3, \\ &21=1\np 4\np 4^2 \text{ and } 31=1\np 5\np 5^2.\end{array}$$ My notes to Feit in 1969 give Davenport pairs $(f,g)$ (\S\ref{secII.2}), branch cycles (\S\ref{secVI.4}) and appropriate number fields over which they are defined   for each case. \end{thm}

In rereading, I see \cite[(1.25)]{Fr73a} left out $n=15$ in its list of difference sets. I'll do that case now for use below.

Take an irreducible degree 4 polynomial over $\bZ/2$ (say, $x^4\np x\np 1$). Then, multiply the nonzero elements (nonzero linear combinations of $1, x, x^2, x^3$ corresponding to 1, 2, 3, 4) by $x$ and use the relation $x^4\np x\np 1=0$, to label them $1, 2, \dots, 15$. Example: $x^4 = x+1$ corresponds to 5.

Choose a hyperplane: Say, the linear combinations of $1, x$ and $x^2$. Then, a difference set, $\sD_{15}=\{1, 2, 3, 5, 6, 9, 11\} \mod 15$ is a list of elements on this hyperplane. 

\begin{defn} \label{multiplier} A {\sl multiplier\/} of difference set $\sD \mod n$ is $c \in  (\bZ/n)^*$ with $c\sD$ a translate of $\sD \mod n$. Denote by $M_{\sD}$ the group of multipliers of $D$. \end{defn}

\begin{exmp} 2 is a multiplier of $\sD_{15}$, generating $M_{\sD_{15}}$, an order four subgroup of the invertible integers $\mod 15$. A translate of the one \cite[\S 2.2.5]{CoCa99} took is $\{1, 2, 3, 8, 10, 13, 14\}$. After multiplication by -1, this is a translation of $\sD_{15}$. \end{exmp}

Here, as for $n= 7$, the non-multipliers of the difference set consist of the coset of multipliers time -1, compatible with the contribution of Storer from the opening of \S\ref{secVI}. In that section we refer to $\gamma_q$ as $\sigma_\infty$. We do that here to allow directly referring to the following observation. Use the notation of \S\ref{App.1}, with $q = p^t$.  
A choice of $\sigma_\infty$, up to conjugacy, defines the inertia generator from \S\ref{secIII.4} attached to a polynomial $f$ that has   geometric monodromy  between $\PGL_n(\bF_q)$ and P$\Gamma$L$_n(\bF_q)$. Further, $\sigma_\infty$, up to conjugacy, defines the attached difference set up to translation given in \eqref{*7}.

\begin{lem}[Multiplier]  \label{multiplierlem} The subgroup of $(\bZ/n)^*$ that corresponds to powers of $\sigma_\infty$ conjugate to $\sigma_\infty$  (in {\rm P$\Gamma$L$_n(\bF_q)$}) equals $M_{\sD}$. \end{lem} 

\subsection{Group theory immediately after Graduate School} \label{secV.4} I knew J.~Ax from my two years at IAS. I went  with him to SUNY at Stony Brook (leaving soon after getting tenure), instead of to U. of Chicago which first offered me tenure. Ax suggested I should explain what I was after to W.~Feit. His rationale: While my difference set conditions were complicated, group theory could handle  intricate matters by comparison to what one could do with algebraic geometry. From Ax's suggestion, I  learned to partition a problem into its group theory, number theory and Riemann surface theory pieces, so that I could handle each separately.

\subsubsection{The Collineation Conjecture}  Here is what I expected. The case $n = 11$ is special. It corresponds to a difference set with a doubly transitive group of automorphisms that doesn't fit into the points/hyperplane pairing on a projective space over a finite field. Still, my reading suggested that I now knew all possibilities for these doubly transitive designs -- as described in \S \ref{secVI.3} -- through Riemann's Existence Theorem. Consider the following condition on a group $G$:  

\begin{triv} \label{rephypoth} It has two inequivalent doubly transitive permutations representations, that are equivalent as group representations (of degree $n$). \end{triv} 

Here was the group theory guess. 

\begin{guess}[Collineation Conjecture] \label{collConj} Assume \eqref{rephypoth} and that $G$  also contains  an $n$-cycle.  Then, $G$ either has degree 11, or it lies between $\PGL_{v+1}(\bF_q)$ and P$\Gamma$L$_{v+1}(\bF_q)$, $n=(q^{v+1}-1)(q-1)$, for some $v$ and $q$. \end{guess}

Given Conjecture \ref{collConj}, I described from it the only possible -- finite set of -- Davenport pair degrees $n$ (as in the rest of this report) over some number field. I could give branch cycle descriptions for all Davenport pairs, thus solving problem D$_2$ \eql{davpart}{davpartb}. Indeed, it gave the full nature of these pairs, without writing equations (as in \S\ref{App.5}) the toughest issue to explain to algebraists. 

\subsubsection{My interactions with Feit 1968-69} \label{FeitInteractions} These were complicated -- in those days all through regular mail. Even without the Collineation Conjecture, it was also possible to bound  degrees of Davenport pairs and use Riemann-Hurwitz  to cut down the total number of branch cycles. This  came from knowing that each branch cycle moved at least half the points. I suggested this to Feit in my description of its consequences, and he proved it  (\cite[Thm.~3]{Fe70},  or \cite[Prop.~1]{Fr73a}). 

Yet, it was Conjecture \ref{collConj} that made a case for the Genus 0 Problem.  Feit suggested that if I accepted the simple group classification, then extant literature might prove the Collineation Conjecture. That allowed me to finish it  (published in \cite[\S9]{Fr99}), and several other pieces of pure group theory. \S\ref{secVII.3} models how a (non-group theory) researcher might approach this.

Yet, the biggest surprise didn't come from group theory. It was possible  (\S\ref{secVI.2}) to finish Davenport's Problem over $\bQ$, D$_1$ \eql{davpart}{davparta}, without the Collineation Conjecture -- or anything related to the classification of simple groups. This used a device whose general  applicability opened up directions that went far beyond discussions of separated variables. The next section explains this, and relates my only specific mathematical interaction with UM beyond graduate school (see \S\ref{secVII.4}). 

\section{The B(ranch)C(ycle)L(emma) and Solving Davenport's Problem} \label{secVI} I was immensely assured -- at the time (see \S\ref{secVI.2}) -- by Storer's Statement \ref{storer}. Yet, the 2nd sentence of Prop.~\ref{DS3} -- which I first overlooked, but used later -- already gives its main thrust.  By assumption $T_f$ and $T_g$ are distinct permutation representations. If, however, -1 was a multiplier, then they would not be. 

\begin{prop}[Storer's Statement]  \label{storer} \cite[ p.~132]{Fr73a} says this: "According to T. Storer the fact that -1 is not a multiplier is an old chestnut in the theory of difference sets. He has provided us with a simple proof of this fact, upon which we base the proof of Lemma 5." \end{prop}

Now I explain the BCL and how it  finished Davenport's Problem over $\bQ$. 

\subsection{Branch cycles and the BCL} \label{secVI.1} As in \S\ref{App.1},  denote the automorphisms of the algebraic numbers $\bar \bQ$ fixed on  a field  $K\subset \bar \bQ$ by $G_K$. 

\subsubsection{Branch points}  \label{branchPoints} 
Algebraic relations have coefficients. If the coefficients are in $\bar \bQ$, then Hilbert's Nullstellensatz says points with $\bar\bQ$ coordinates satisfying these relations determine all points satisfying the algebraic relations. 

\S\ref{secII.1} reminds of the distinction between affine sets (defined by equations in a finite set of variables) and projective sets (defined by homogeneous equations in a finite set of variables). You can view a point ($x_0,\dots,x_n$) satisfying homogenous equations as a point on an affine space, but the projective points are equivalence classes $\{a(x_0,\dots,x_n)\}$, $a \ne 0$. We require that one of the $x_i\,$s is nonzero.

In practice, here is the significance of a point lying on an algebraic set, versus, say, lying on a general complex analytic set.  Take any algebraic set, $V$, over $\bar \bQ$ and act on an algebraic point $v\in V$ by $\gamma\in G_\bQ$.  Then the image ${}^{\gamma}v$  will lie on the set defined by $\gamma$ acting on coefficients of the equations for $V$. 

Consider a degree $n$ ($>0$) rational function $f$ in $x$ (or any cover, \S\ref{whatCover}) as a map to  $\prP^1_z$. Then, points of $\prP^1_z$ with fewer than $n$  points of $\prP^1_x$ above them are {\sl branch points\/}, $\row z r$, of $f$.  To be explicit with polynomial covers, we'll take $z_r$ to be $\infty$. 
If $\gamma\in G_\bQ$ fixes the coefficients of  $f$, then $\gamma$  permutes $\row z r$: $\gamma \mapsto \tau_\gamma\in S_r$. 

\subsubsection{Branch cycles, the tie to groups} \label{secVI.4} Recall $\sigma_\infty$  in \S\ref{secIII.5}, a generator of inertia over $\infty$.  Whatever the branch points, $\row z r$, in \S\ref{branchPoints},  for a compact Riemann surface cover  $f : X \to \prP^1_z$, each produces a representative, $\row \sigma r$, of conjugacy classes $\bfC={\row \C r}$ in the geometric monodromy $G_f\le S_n$. This is by the same process, a finger walking (again,  clockwise) around  $z_i$,  along a closed path $P_i$. Then, $\sigma_i$ permutes the points over the base point by following that path. 

Further, the disjoint cycles of $\sigma_i$ correspond to the points of $X$ lying over $z_i$, and the disjoint cycle length is the ramification index of that point over $z_i$. 

App.~\ref{classgens} explains {\sl classical generators\/} \cite{CGen} of the fundamental group of 
$$ \prP^1_z\setminus  \{z_1,\dots,z_r\}=U_\bz: \text{ denoted }\row P r.$$ It indicates we need two further visually verifiable constraints on $\row P r$  to assure they generate the fundamental group of $\pi_1(U_\bz)$ with only one relation (up to uniform conjugation of the paths):  $P_1\cdots P_r$ is homotopic to the trivial path. An explicit one-one correspondence -- albeit, dependent on the choice  of the classical generators unless the covers have abelian monodromy -- goes between branch cycles (\S\ref{prodBCs}) and algebraic covers of the sphere branchcd over  $\{z_1,\dots,z_r\}$. 

A self-contained treatment, filling in everything from material in \cite{Ah79} is in \cite[Chap.~4]{Fr09}, with a survey in http://math.uci.edu/deflist-cov/$\tilde{\  }$mfried/Nielsen-Classes.html. Before we do an exposition on the use of branch cycles we first introduce the Branch Cycle Lemma. This is essentially a separate formula. Solving Davenport's Problem represents its first use. 

The index, $\ind(\sigma)$, of a permutation $\sigma \in S_n$ is just $n$ minus the number of disjoint cycles in the permutation. Example: an $n$-cycle in $S_n$ has index $n\nm 1$, and an involution has index equal to the number of disjoint 2-cycles in it. The Riemann-Hurwitz formula says the {\sl genus}, $\geng_X$ of $X$ satisfies

\begin{equation} \label{*10}     2(n + \geng_X- 1) = \sum_{i=1}^r \ind(\sigma_i). \end{equation}

\subsubsection{Branch Cycle Lemma} \label{BCLtreat} Continue the notation above. Assume $f: X\to \prP^1_z$ is a cover defined over $K$. Denote the order of elements in $\C_i$ by $e_i$, the least common multiple of the $e_i\,$s by $N = N_\bfC$ and  the elements of $\C_i$  put to the power $c$ by $\C_i^c$. 

As in \cite[exp.~(5.7)]{Fr77}, $\gamma\in G_K$ also  acts through the arithmetic monodromy $^aG_f$ (\S\ref{secII.3}) and so through the normalizer, $N_{S_n}(G)$, of $G$ in $S_n$. Write this action with $\omega_\gamma$ acting on the right  of Puiseux expansions of function field  elements $\alpha$, centered at the $z_i\,$s.  That is, $\alpha$ evaluated in a neighborhood of a point $\bp$ over $z_i$ expands as a power series in $(z-z_i)^{\frac 1 k}$, with $k$ the ramification index of $\bp$ over $z_i$. Denote the subgroup of $N_{S_n}(G)$ that permutes the conjugacy classes of $\bfC$, with  multiplicity, by $N_{S_n}(G, \bfC)$. 
The B(ranch)C(ycle)L(emma) compares $\omega_\gamma$  and $\tau_\gamma$ (\S\ref{branchPoints}) with the cyclotomic character   $$\gamma: e^{2\pi i/N} \mapsto e^{c_\gamma2\pi i/N}.$$ 
\begin{triv} \label{*9} If  $j=(i)\tau_\gamma$, then $\omega_\gamma\C_j\omega_\gamma^{-1} = \C_i^{-c_\gamma}$ \cite[p.~62--64]{Fr77}. \end{triv}
Suppose  putting  $\bfC$ to all powers $c\in (\bZ/N_{\bfC})^*$ (resp.~all $c$ fixed on $K\cap \bQ(\zeta_{N_\bfC})$)   leaves $\bfC$ invariant. Then, we say $\bfC$ is a {\sl rational union\/} (resp.~$K$-{\sl rational union}). Denote the  extension of $\C_i$ to $^a G_f$ by $^a\C_i $.  

\begin{rem}[Remembering the BCL]  \label{remBCL} Here is a quick mnemonic for the identifications in \eqref{*9}. Apply both sides to $(z-z_i)^{\frac 1 k}$ for the correct power of $\zeta_k$ on the right side. \cite[p.~39]{Vo96} has written $-c_\gamma$ for our $c_\gamma$. We would love to apply the formula directly to the $\sigma_i\,$s. Yet, as \cite{Fr77} explains, you can't expect to consistently label Puiseux expansions  of function field elements at different points $z_i$ and $z_j$. This is compatible with the topological nature of classical generators (Prob.~\ref{CGtop}). So, the formula only relates conjugacy classes, except, when you work over the real numbers as in the explicit application to real covers in \cite[\S2.4]{DeFr90}. 
\end{rem} 

\begin{res}[Example use of the BCL] \label{exBCL} Assume $f$ has definition field $K$. 
\begin{edesc} \label{useBCLf} \item \label{useBCLfa}  If each $\omega\in {}^a G_f/G$ is in $N_{S_n}(G,\bfC)$, then $\bfC$ is a $K$-rational union. 
\item \label{useBCLfb}  If  $z_i \in K$, then $^a\C_i $ is a $K$-rational  class in $^aG_f$. \end{edesc} \end{res} 

A field extension $L/K(z)$ is {\sl regular\/} if the only constants in $L$ consist of $K$. The condition $^aG_f=G_f$ says the Galois closure of the function field extension $K(X)/K(z)$ is a regular extension of $K(z)$: we have a {\sl regular realization\/} of $G$.  Then, \eql{useBCLf}{useBCLfa} says only by using conjugacy classes where $\bfC$ is a rational (resp.~$K$-rational) union can we find a regular realization of $G$ over $\bQ$ (resp.~over $K$). 

Schur's Conjecture  (see Lem.~\ref{ChebChar} for a more elementary use of the BCL) and Serre's Open Image Theorem (see \S\ref{limitDih}) are especially sensitive to using \eqref{useBCLf} to distinguish between $^aG_f\le N_{S_n}(G)$  -- always true -- and the conclusion of \eql{useBCLf}{useBCLfa}. 

More general, and with much more application than the regular realization of groups are $(G,G^*)$-realizations (with $G^*\le N_{S_n}(G)$) larger than $G$. That is, find covers over $\bQ$ where the geometric/arithmetic monodromy pair is $(G,G^*)$ as in \S\ref{RETConv} on the RET converse C$_2$. 

$(A_n, S_n)$-realizations from polynomials in $\bQ[x]$ disproved three conjectures in the literature  \cite{Fr95a}. It will come in handy for others, too.   \cite{Fr95a}  left unsolved if there are odd {\sl square\/} degree polynomials  in  $\bQ$ giving an $(A_n,S_n)$-realization. \cite{Mu98b} showed such polynomials do not exist, a practical addition to the BCL.

\subsection{Fields supporting Davenport pairs} \label{secVI.2} 
Suppose $f\in K[x]$, $n = \deg(f)$.  Then total ramification over $\infty$ (a  $K$ point) implies any geometric component of the Galois closure has definition field $\hat K_f$ (\S \ref{secI.3}) a subfield of $K(e^{2\pi i/n})$ . 

\subsubsection{Apply the BCL to Davenport pairs}  \label{applyBCL}  Apply $\gamma\in G_\bQ$ to the coefficients of $f$ and $g$, and  
denote  solutions for $x$ in  $^\gamma f(x)-z=0$ (resp.~$^\gamma g(y)-z=0$)) by  $^\gamma x_i$ (resp. ~$^\gamma y_i$). For each $c\in (\bZ/n)^*$, choose $\gamma \in G_K$ whose restriction to $\bQ(e^{2\pi i/n})$ is $c$. This gives an action of $(\bZ/n)^*$ on equation \eqref{*7}, producing a relation
$$^\gamma x_1={}^\gamma y_c+ {}^\gamma y_{c\alpha_2}+\dots + {}^\gamma y_{c\alpha_k}.$$ Expanding  these solutions at $\infty$ in $z^{-\frac 1 n}$ allows tracing this action. Consider the corresponding difference set (from Prop.~\ref{DS3}): $\sD_f =\{1, \alpha_2,\dots, \alpha_k\}$. Denote the fixed field of the multiplier $M_f$ (Def.~\ref{multiplier}) of $\sD_f$ in $\bQ(e^{2\pi i/n})$ by $\bQ_{M_f}$. 

\begin{prop} \label{DS5} Suppose $(f, g)$ is a Davenport pair -- with $f$ indecomposable -- over some number field $K$: the hypotheses of D${}_2$ (or, Thm.~\ref{DS2}, but over $K$). Then: $K$ contains $\bQ_{M_f}$. More generally the following conclusions hold.
\begin{edesc} \label{davpairs} 
\item  \label{davpairs1}  Since -1 is not a multiplier (Prop.~\ref{storer}),  the reals do not contain $\bQ_{M_f}$. So, for any Davenport pair, $K$ is not $\bQ$, thereby solving  \eql{equats}{equats3} of \S\ref{secII.2} with  the hypothesis that $f$ is indecomposable. 
\item  \label{davpairs2}  For each degree in Thm.~\ref{DS4}, there are Davenport pairs over $K$ if and 
only if $K$ contains $\bQ_{M_f}$. For just the degrees $n=7, 13, 15$, there are infinitely many distinct Davenport pairs, mod M\"obius equivalence (\S\ref{secI.1}).
\item  \label{davpairs3}  For the degrees in \eql{davpairs}{davpairs2}, there are Davenport pairs $(f,g)$ with branch points defined over fields disjoint from $\bQ(e^{2\pi i/n})$. For those, consider $\gamma\in  G_\bQ$ mapping $e^{2\pi i/n}$ to $e^{-2\pi i/n}$, but acting trivially on branch points. Then, $f(x)={}^\gamma g(x)$ (action on the coefficients by $\gamma$). \end{edesc} \end{prop}

\subsubsection{Start of  Prop.~\ref{DS5}}    
Multiplier Lem.~\ref{multiplierlem} shows   \eql{davpairs}{davpairs1} is about conjugacy classes, not  merely  cycle types. The multiplier $M_f$ measures -- special case of \eql{davpairs}{davpairs1}  -- how far the class of $\sigma_\infty$ is from rational (Result \ref{exBCL}). With $(f, g)$ a Davenport pair, \eql{davpairs}{davpairs1} follows from concluding in Prop.~\ref{DS3} that -1 times the difference set $\sD_f$ gives the difference set  $\sD_g$. Since $g$ and $f$ give  inequivalent covers, this says the difference set for multiplication by -1 cannot be a translate of the original difference set. I didn't, however, make that observation in \cite{Fr73a}.

By contrast with nonexistence in  \eql{davpairs}{davpairs1}, \eql{davpairs}{davpairs2}  is an existence result.  It uses that the BCL precisely gives definition fields of total families of cover. Explaining this, and those total families takes up the remainder of \S\ref{secVI} and all of \S\ref{App.4}.  

I went after this general context because, while Schur's Conjecture was easy compared to Davenport's problem, there were other problems, much tougher, that acceded to this method. Although I think "attempting to write equations out" is not a road to success, many do want equations. So \S\ref{App.5} revisits this topic.

Lewis knew Al Whiteman, who was at IAS my first year there. I had seen him talk on difference sets, his speciality.  He responded to my questions  by suggesting I talk to his student Storer, who had just been hired by Michigan.

I stayed at UM part of the summer of '68 to write up \cite{Fr73a}. The combinatorial trick \cite[ (1.19)]{Fr73a} is Storer's. He  often told his opinions of me. Especially:  There must be something wrong with me for knowing so much mathematics. His thought: It must be because I spent all of my time slaving in the library. (For the record: I learned mostly by being attentive at talks; secondly from seriously refereeing hard papers. That's relevant to my comments on  group theory in \S\ref{secVII.3}.)

\!\!\subsection{Branch cycles produce Davenport pairs} \label{secVI.3} We use use Davenport's problem to teach Riemann's  approach to algebraic functions beyond abelian functions.  

\subsubsection{Questions aimed at Statement \eql{davpairs}{davpairs2} of Prop.~\ref{DS5}}  \label{questionsDPs} Use notation from \S\ref{secVI.1}. 

\begin{edesc} \label{genus0mon} 
\item  \label{genus0mon1}  What data allows finding Davenport pairs $(f,g)$ (over some number field;  $f$ indecomposable ) of  each degree 7, 11, 13, 15, 21 and 31?
\item  \label{genus0mon2}  Given an affirmative to \eql{genus0mon}{genus0mon1}, how might you describe all such Davenport pairs and their definition fields for each such degree?
 \item \label{genus0mon3}  What has this to do with simple groups, and how  might you persuade others the value of this approach to finding  Davenport pairs? 
\item  \label{genus0mon4}  Assuming success in the above, what general conclusions might you dare about monodromy groups of polynomials or rational functions? \end{edesc}

We start with  $n=7$, to how it works, then refer to the case $n=13$ to compare others who have considered the production of equations. 

The group $\PGL_3(\bZ/2)$ (\S\ref{secV.3}) acts on the 7 points and 7 lines of 2-dimensional projective space over $\bZ/2$. An involution (order 2 element) fixes all 3 points on a line; every other nonidentity element fixes no fewer points. That means the minimal possible index of the $\sigma_i\,$ is 2, and $\sigma_r$ has index six. Since the top space for a polynomial cover $f$ is $\prP^1_w$, that means  $\geng_{\prP^1_w}=0$. 

\S\ref{secVI.5} shows why there are Davenport pairs with their geometric monodromy group equal to $\PGL_3(\bZ/2)$,  answering question \eql{genus0mon}{genus0mon1} of \S\ref{secVI.3}, for degree 7. The method works for all degrees in that question.

First consider the possibility that $r=4$.  What could be the minimal possible indices for branch cycles of a polynomial $f$  with monodromy group $\PGL_3(\bZ/2)$, where $\sigma_4$ is a 7-cycle?  Then, the minimal possible sum of the four indices of corresponding $\sigma_i\,$ in  \eqref{*10}   is $3\cdot 2+6=12$. In our case the right side is 12, and the genus is 0. So, no other choices with $r=4$ would produce genus 0.

Further, if such a polynomial exists representing $f$ in a Davenport pair, we now know that these $\sigma_i\,$s, $i=1, 2, 3$, all lie in this hyperplane fixing conjugacy class. One difference set here is $\{1, 2, 4\}$. \S\ref{prodBCs} shows  why there is a Davenport pair  $(f,g)$ with $T_f$ for $f$ acting on $\{1,2, \dots,7\}$, with these properties: An inertia generator over $z=\infty$, acts as $\sigma_\infty=(1\,2 \dots 7)$, while it acts as translates of $\{1,2,4\}$ for $T_g$. 

\subsubsection{Cover producing branch cycles} \label{prodBCs} 
What we need is a converse -- cover producing conditions -- from such $\sigma_i\,$s. There is one: R(iemann)'s E(xistence) T(heorem).  Given such $\sigma_i$, $ i=1,\dots , r$, in a group $G$,  we are asking when there is a cover $f : X \to \prP^1_z$ branched at any given points, $\row z r$, with its geometric monodromy group $G$, and having the attached conjugacy classes $\bfC=\{\row C r\}$ of  $\row \sigma r$. 

The answer: Such covers correspond to  $\sigma_i'$, conjugate (in $G$) to $\sigma_i$, $i=1,\dots, r$, for which these expression \eqref{pathcond} interpreting conditions hold:

\begin{edesc} \label{cycleconds} 
\item \label{cycleconds1} {\sl Generation}:  $\lrang{\sigma_i' | i=1,\dots, r} = G \le  S_n$; and 
\item \label{cycleconds2} {\sl  Product-one}: $\sigma_1' \cdots  \sigma_r'=1$.  \end{edesc} 

From \eql{cycleconds}{cycleconds2}, any $r\nm 1$ of the $\sigma_i'\,$s in  \eql{cycleconds}{cycleconds1} generate $G$.
Those who use the {\sl monodromy method\/} call such $\sigma_i'\,$s satisfying \eql{cycleconds}{cycleconds1} and  \eql{cycleconds}{cycleconds2}  {\sl branch cycles}. We call the collection of all such, in the respective conjugacy classes $\bfC$, the {\sl Nielsen class\/} $\ni(G, \bfC)$ of the cover. Further, covers corresponding to two such choices of $r$-tuples satisfying \eqref{cycleconds}  will be isomorphic as covers (of $\prP^1_z$) if and only if some element in $S_n$ conjugates the one $r$-tuple to the other.

As in  \S\ref{BCLtreat} 
consider, $N_{S_n}(G,\bfC)$, the subgroup of $S_n$ that normalizes $G$, and permutes the classes in  $\bfC$ (preserving their multiplicity).  
Two covers of $\prP^1_z$  are {\sl absolute\/} equivalent (isomorphic by a map commuting with the maps to $\prP^1_z$)  when their corresponding $r$-tuples are conjugate by  $N_{S_n}(G,\bfC)$. We use two other equivalences than absolute later. 
\S\ref{App.3} explains why branch cycles give algebraic covers. (In the Davenport cases -- genus 0 with $\sigma_r$ an $n$-cycle -- each a polynomial map.) 

The genus $\geng_X$  in \eqref{*10} depends only on the images of $\row \sigma r$ in $S_n$, corresponding to the  representation $T_f$. For that, distinguishing conjugacy classes from cycle-type is irrelevant. Still, Multiplier Lem.~\ref{multiplierlem} exposes that distinguishing conjugacy classes of  $n$-cycles is significant in projective linear groups. Using Storer's Statement \ref{storer} ( as in Prop.~\ref{DS5}), there is more than one such class. 

For $n=7$ there are two, represented by $\sigma_\infty$ and $\sigma_\infty^{-1}$. For $n=13$, $\{1,2,4,10\}$ (translation equivalent to $\{0,1,3,9\}$) is a difference set \cite[p.~60]{Fr05a}, with 3 generating the multipliers. So, $\sigma_\infty^a$, with $a$  running over powers of  $3 \!\!\mod 13$ are conjugate to $\sigma_\infty$.  So there are 4 (translation) inequivalent difference sets mod 13. In \S\ref{secVI.5} this tells us why the covers we produce --  Davenport pairs -- fall in four families, conjugate over the degree 4 extension of $\bQ$ in $\bQ(\zeta_{13})$.  

\subsection{Covers from a Nielsen class} \label{secVI.5}  \S\ref{cont7} continues with $n=7$ and the  classes from \S\ref{secVI.4}. Then, \S\ref{App.3} shows how the Nielsen class computation produces the data for covers. \S\ref{App.4} turns this into properties of Davenport pair {\sl families\/}. 

\!\subsubsection{Branch cycles for $n=7$} \label{cont7} The group $G$ in  \eql{cycleconds}{cycleconds1} of \S\ref{secVI.4} must be $\PGL3(\bZ/2)$ (and not smaller) to assure we get the pair of doubly transitive representations. 

We can write by hand all involutions that could appear as $\sigma_1$, $\sigma_2$ or $\sigma_3$.  In \eqref{*7}, start with  the hyperplane containing the fixed points corresponding to 1, 2 and 4. Then, involutions fixing the points on this hyperplane are one of $(3\,5)(6\,7)$, $(3\, 6)(5\, 7)$ or $(3\,7)(5\, 6)$. Conjugate by (powers of) $\sigma_\infty$ to get all others. 

Now find all  involution 3-tuples $(\sigma_1, \sigma_2, \sigma_3)$ with product this specific 7-cycle $$\sigma_\infty^{-1}=(7\, 6\, 5\, 4\, 3\, 2 \,1) \text{ (done in detail in \cite[p.~349]{Fr95a})}.$$  Therefore, the covers with fixed branch points $( z_1, z_2, z_3, \infty)$, and fixed conjugacy classes attached to these in a given order) correspond to this {\sl absolute Nielsen class}:
$$\ni(\PGL_3(\bZ/2), \bfC)^\ab = \ni(\PGL_3(\bZ/2), \bfC)/\PGL_3(\bZ/2). $$
By listing the 4th entry as $\sigma_\infty^{-1}$, we  fix an absolute Nielsen class element up to conjugation by $\sigma_\infty$. There are precisely 7.  Suppose given $(\sigma_1, \sigma_2, \sigma_3, \sigma_\infty)=\psigma$, and a set of classical generators relative to 3 distinct finite branch points $z_1,z_2,z_3$ (as in \S\ref{classgens}). Then, this produces $f(x)\in \bC[x]$ uniquely up to affine change of $x$. 

Apply the permutation representation $T^{\text{hyp}}$ of $\PGL_3(\bZ/2)$ from acting on the lines of $\prP^2(\bZ/2)$ to $\psigma$ in the Nielsen class. To compute this,  write the hyperplanes as unordered collections of integers given by the translations of the difference set $\{1,2,4\}$.   If the result is $\psigma'=(\sigma_1', \sigma_2', \sigma_3', \sigma_\infty')$, then this is the branch cycle description for $g$: the other half of the Davenport pair for $f$. 

The monodromy method can often be precise about the collection of covers in a given Nielsen class without writing them  explicitly. Here is an example of that. Denote the field $\bQ_{M_f}$ in Prop.~\ref{DS5} by $\bQ_n$. Example: $\bQ((-7)^{\frac1 2})=\bQ_7$. 

\begin{prop}[DS$_6$]  \label{DS6} There are infinitely many (M\"obius inequivalent -- 
\S\ref{secI.1}) degree $7$ Davenport pairs over any extension $K$ of $\bQ_7$. They correspond to the $K$ values of a uniformizer, $t_7$, of a genus zero $j$-line cover $\sH^{\abs,\red}_7$ defined over $\bQ$. A similar result,  with  $\bQ_n$ and a parameter $t_n$, holds for $n=13$ and 15. \end{prop} 

\S\ref{App.5} shows braid computations for $n=7$ that dispell any mystery about $\bQ_n$ that also give these properties of $\sH^{\abs,\red}_n$. (They also hold for $n=13,15$.)   
\begin{edesc} \label{redcomps7} \item  \label{redcomps7a} As a carrier of Davenport pairs, $\sH^{\abs,\red}_n$ has just one component defined over $\bQ_n$; and 
\item  \label{redcomps7b} as a $j$-line cover, $\sH^{\abs,\red}_n$ has definition field $\bQ$ rather than $\bQ_n$.
\end{edesc} 

M\"obius equivalence is  also called {\sl reduced equivalence\/} of covers. This equates two covers $\phi_i:X_i\to \prP^1_z$ if for some $\alpha\in \PGL_2(\bC)$,  $\alpha\circ \phi_1$ is absolute equivalent to $\phi_2$. Nielsen classes are a surrogate for data that canonically  produces a family of covers. By considering reduced (absolute) equivalence, we aim for a  normal form -- here of polynomials -- from which we can generate any family of covers. 

What \eql{redcomps7}{redcomps7b} says is that -- like any reduced Hurwitz space with $r=4$ -- the parameter space is a curve, and a natural $j$-line cover.  \S\ref{spacesCovers} shows how to list  irreducible reduced Hurwitz space components for any $r$. When $r=4$, so these are curves, it shows how to calculate the genuses of  their (compactified) components.

You might ask, "Where are these Davenport pairs?" \S\ref{App.5} discusses their specifics, coming  from alternate treatments -- based on this one -- that produced the pairs.  

\subsubsection{Branch cycles versus algebraic covers} \label{App.3}  \S\ref{cont7} produced a polynomial $f$ (cover) from a set of branch cycles and classical generators. Fixing the classical generators (and branch points) gives a one-one correspondence between $r$-branched covers of $\prP^1_z$ and branch cycles. Here is the major unsolved problem in using RET. 
\begin{prob}[Classical generation] \label{CGtop} Both sides of this correspondence are algebraic, but classical generators are not. Prove such a correspondence without using such a topological gadget.   
\end{prob} 

http://math.uci.edu/deflist-cov/Alg-Equations.html has examples of Prob.~\ref{CGtop}.   \cite[p.~27]{Mum76} lists an imprecise equivalent to classical generators to  relate Teichmuller and Torelli space.  Applications in \cite{Vo96} seem to be only about the Inverse Galois Problem, but really its motivation was from applications  we discuss here. 

It is not immediate that having a cover $f: X \to \prP^1_z$ means that $X$ is algebraic (projective: \S\ref{secII.1}). Still, that follows  given a single further function that {\sl  separates\/} -- has different values on -- the fiber over {\sl some\/} point of $U_\bz$. The R(iemann)-R(och) Theorem guarantees such a function. Though non-trivial, no one argues over RR.  

When $X$ has genus 0, shouldn't it be easy to produce such a function (lets call it $w$)?  Here is an historical track to finding $w$. You take the differential $df$ of $f$.  From general principles it has degree $2\geng_X - 2 = -2$. Similarly, for the function $w: X \to \prP^1_w$ (once we have it): It's differential $dw$ has degree -2. An especially good $w$ would be one that separates all points (is an isomorphism of $X$ to $\prP^1_w$).  The support of its polar divisor is concentrated over $w=\infty$. Since $X$ is simply connected, any meromorphic differential with this property, being locally integrable, is globally integrable to a function.  

\begin{prob} \label{seppoints} When $\geng_X=0$, what types of data allow automatic creation of such a function $w$ giving the isomorphism $w: X \to \prP^1_w$? \end{prob}  

\section{Hurwitz monodromy and braids} \label{App.4} 
\S\ref{App.3}  points to the essential object -- {\sl classical generators\/} on the $r$-punctured sphere $U_{\bz'}$. These assign a cover of $\prP^1_z$ to each element in an absolute Nielsen class. 

\subsection{Grabbing a cover by its branch points}  \label{HurMonUr} Denote the space of $r$ distinct, but unordered,  points on $\prP^1_z$ by $U_r$. Start with one cover $f: X\to \prP^1_z$ branched over $\bz'$. Then,  deform the punctures $\bz'$, keeping them distinct, to another set of $r$ points $\bz''$.   That is, give a  path (continuous and piecewise differentiable) $\sL$: $t \in [0,1] \mapsto \bz'(t)$, in  $U_r$,   with $\bz'(0)=\bz'$ and $\bz'(1)=\bz''$. 

Now consider the case $\bz'=\bz''$: equality of sets of branch points.  Then, $\sL$ may permute the order of the points in $\bz'$. Along $\sL$ we also can deform the initial classical generators $\sP'$. At the end  we have a new set of classical generators $\sP''$. 

A base point distinct from the branch points  is necessary to talk about classical generators. Therefore, freely following $\sL$ may force us to deform the base point $z_0'$, too: $t\in [0,1] \mapsto z_0(t)'$, with $z_0'=z_0(0)'$ and $z_0''=z_0(1)'$.  

You can always {\sl wiggle\/} $\sP''$ fixing its isotopy class and assuring neither $z_0''$ or $z_0'$ are on any of its paths. Then, you can further deform $z_0'' $ to $z_0'$, leaving all points on $\sP''$ fixed, just to get the original base point. Mapping the elements of  $\sP'$ in order to those of $\sP''$ induces an automorphism of $\pi_1(U_{\bz'},z_0')$. Since there  is no canonical way to deform $z_0''$ back to $z_0'$, mod out by the conjugation action of $\pi_1(U_{\bz'},z_0')$ on itself to make this automorphism unambiguous. 

Following  the branch point path produces an automatic analytic continuation of the cover $f$: http://math.uci.edu/\~{}mfried/deflist-cov/Hurwitz-Spaces.html, \S V.

Running over all such paths $\sL$ induces the {\sl Hurwitz monodromy group\/}, $H_r$. It acts as automorphisms on $\pi_1(U_{\bz'},z_0')$ modulo this inner action.  Two elements of $H_r$ generate it. We call these $q_1$ and $\sh$. For our purposes we have only to know their action (\cite[\S4]{Fr77} or \cite[Def.~9.3]{Vo96}) on a Nielsen class representative: $\bg=(g_1,g_2,g_3,\dots,g_r)\in \ni(G,\bfC)^\abs$. 

\begin{edesc}
\item $q_1: \bg \mapsto (g_1g_2g_1^{-1},g_1,g_3,\dots,g_r)$ the {\sl 1st\/} (coordinate) {\sl twist\/}, and
\item  $sh: \bg\mapsto  (g_2,g_3,\dots,g_r, g_1)$, the {\sl left\/} shift. \end{edesc} 
They both preserve generation, product-one and the conjugacy class collection conditions of \eqref{cycleconds}, 
Conjugating $q_1$ by $\sh$, gives $q_2$, the twist moved to the right. Repeating gives $q_3,\dots, q_{r-1}$. Three relations generate all relations for $H_r$:
\begin{edesc} \label{hurrelations} \item  \label{hurrelations1} Sphere: $q_1q_2\cdots q_{r-1}q_{r-1}\cdots q_1$; 
\item  \label{hurrelations2} Commuting: $q_iq_{j}=q_{j}q_i$, for $|i-j|\ge 2$ (read subscripts  mod $r\nm 1$); and 
\item  \label{hurrelations3} (Braid) Twisting: $q_iq_{i\np1}q_i=q_{i\np1}q_iq_{i\np1}$. \end{edesc}
The group $H_r$ inherits  \eql{hurrelations}{hurrelations2} and \eql{hurrelations}{hurrelations3} from the Artin braid group. 

\subsection{Spaces of covers}  \label{spacesCovers} A permutation representation of any fundamental group produces a(n unramified) cover. In particular, the  $\pi_1(U_r,\bz')$ permutation action on $\ni(G,\bfC)^\abs$ (\S\ref{prodBCs}) produces a cover:  $\sH=\sH(G,\bfC)^\abs   \to U_r$. 

\subsubsection{The points of the space} Each (complex) point $\bp \in \sH$   represents an equivalence class of sphere covers. The equivalence --  the simplest possible (called absolute) -- of $\phi: X \to  \prP^1_z$ and $\phi': X' \to \prP^1_z$ is where there is a continuous map from $X$ to $X'$ commuting with the projections to $\prP^1_z$. 

\begin{defn} \label{centralizer} A permutation representation $G\le S_n$ satisfies the {\sl centralizer\/} condition if no nontrivial element of $S_n$ commutes with $G$. It satisfies the {\sl normalizer\/} condition if the normalizer of $G(1)$ in $G$ is just $G(1)$.  \end{defn} 

From \cite[Lem.~2.1]{Fr77} the Def.~\ref{centralizer} conditions are equivalent. If a cover $\phi: X\to Y$ corresponds to the permutation representation, this is equivalent to there being no (nontrivial) automorphisms that commute with $\phi$. For example, the following gives a practical application of knowing the geometric monodromy group. 

\begin{lem} \label{commuteCyc} Suppose $G\le S_n$ is primitive (as in \eql{primdt}{primdta}), it contains an $n$-cycle $\sigma_\infty$, and $G(1)$ is nontrivial. Then, a cover $\phi$ with monodromy $G$ has no automorphisms. \end{lem} 

\begin{proof} From the above, if $\phi$ has an automorphism, then some $\tau\in S_n$ centralizes $G$. Compute easily: $\tau\in S_n$ centralizing $\sigma_\infty$ is a power of $\sigma_\infty$ (as in \cite[p.~47]{Fr70}). So, $\tau\in G$, but  $\tau\not\in G(1)$. As $G$ is primitive, $\lrang{G(1),\tau}=G$:  $\tau$ is transitive on $\{1,\dots,n\}$. So, it is an $n$-cycle itself that centralizes $G(1)$ and $G(1)$ is trivial.  \end{proof} 

\subsubsection{Using fine moduli} \label{fineModuli}  For each projective variety, including $\sH$, each point has a field generated by its coordinates. When, as in Prop.~\ref{absFamCovers}, points represent solutions to a problem, that may allow precisely finding over what fields such solutions occur. This holds, as in Thm.~\ref{vectorBundle}, applied to existence of Davenport pairs. 

 \begin{prop} \label{absFamCovers}  Assume  $K\subset \bC$. Then, a $K$ point of $\sH$ corresponds to an equivalence class of covers with the whole set defined over $K$. Assume any of the equivalent conditions of Def.~\ref{centralizer}. Then, there is a unique total family $$\Phi: \sT\to \sH\times \prP^1_z$$ of covers over $\sH$  \cite[p.~62]{Fr77}. Also, a  $K$ point $\bp \in \sH$ gives a well-defined $K$ cover in the class of $\bp$: $\Phi_\bp: \sT_\bp \to  \bp\times \prP^1_z$;  interpret this as a $K$ cover of $\prP^1_z$. \end{prop}

This abstract result says that we can recover any given family of absolute covers in a given Nielsen class, assuming the conditions of Def.~\ref{centralizer}. That is, these guarantee {\sl  fine moduli\/} for covers in the corresponding Nielsen class.  The word \lq\lq unique\rq\rq\ means that for any other such representing family $\Phi': \sT'\to \sH\times \prP^1_z$, over $\sH$, there is a unique analytic map from $\sT$ to $\sT'$ that commutes with  $\Phi$ and $\Phi'$. Such a family being algebraic -- giving meaning to the definition field statements -- implies there is an $m$ (not unique), so that $\sT$ embeds in $\sH\times \prP^m$ with $\Phi$ compatible with the natural projection $\sH\times \prP^m\to \sH$. (Furthermore, $\sH$ is quasi-projective.) 

\subsubsection{Finding definition fields}  \label{famdeffield} 
We indicate an essential step: How we  find  the definition field of the family $\Phi$ in Prop.~\ref{absFamCovers} from information on the Nielsen class. 

Recall the integer $N_\bfC$ from \S \ref{secVI.1}.  Prop.~\ref{DS5} introduces a multiplier group, and  \cite[\S5]{Fr77} generalizes it -- based on the BCL -- to define a cyclotomic field (generalizing $\bQ_{M_f}$ in \S\ref{applyBCL}) related to any absolute Nielsen class. Recall the elements, $N_{S_n}(G,\bfC)$,  of $S_n$ that normalize $G$ and permute the classes of $\bfC$ ( \S\ref{BCLtreat}). 

Simultaneously conjugating all entries of $\bg\in \ni(G,\bfC)$ by $N_{S_n}(G,\bfC)$ (\S\ref{BCLtreat}) gives $h\bg h^{-1}\in \ni(G,\bfC)$. \cite[p.~60]{Fr77} generalizes the multiplier group: 
\begin{equation} \label{multgps} \begin{array}{rl}
M_{\bfC}=\{c\in (\bZ/N_\bfC)^* &\mid \exists \beta \in S_r, h\in N_{S_n} (G,\bfC),  \\ 
&h^{-1}\C_i^c h=\C_{(i)\beta}, i=1,\dots,r\}.\end{array} \end{equation} 

\begin{defn} Denote the fixed field of $M_\bfC$ in $\bQ(e^{2\pi i/N_\bfC})$ by $\bQ_{M_\bfC}$. \end{defn} 

\begin{prop} \label{absHurwitzSpaceDef} If the Def.~\ref{centralizer} conditions  hold, the total family of Prop.~\ref{absFamCovers}  over $\sH$, with its map to $U_r$, has precise definition field $\bQ_{M_\bfC}$. Also, the definition field of each connected component of the family contains $\bQ_{M_\bfC}$.

Even if the conditions of Def.~\ref{centralizer} don't hold,  the definition field statement holds by  regarding $\sH$ as the moduli of covers in the Nielsen class. Orbits of $H_r$ on $\ni(G,\bfC)^\abs$ correspond one-one with  connected components of $\sH$. \end{prop} 

The 1st paragraph of Prop.~\ref{absHurwitzSpaceDef} suffices for Davenport pairs. The proposition is a corollary of \cite[Prop.~5.1]{Fr77}. App.~\ref{abs-innTie} reviews this -- including explaining the 2nd paragraph -- and ties it to \cite[Main Thm.]{FrV91}. The Hurwitz space interpretation  shows Prop.~\ref{absHurwitzSpaceDef}  is the essential ingredient to the latter. 

Let $\sH'$ be a (complex analytically) connected component of $\sH$.  If there is only one component, then it has definition field $\bQ_{M_\bfC}$. Now assume there is more than one. Regarding $\sH'$ as a space of covers, some number field $K$ is a  minimal definition field for that structure. Since $\sH$, As an unramified cover of a manifold, $\sH$ is a manifold. So, an argument so simple, I give it here, says that no $\bp\in \sH'$ can have coordinates in a field smaller than $K$ \cite[\S5]{Fr77}. 

For simplicity assume $\bp$ has coordinates in $\bQ$, and $[K:\bQ]> 1$. Choose $\gamma\in G_\bQ$ nontrivial on $K$: $^\gamma\sH'$ is a component of the moduli space for  a new space of covers of $\prP^1_z$; either another Nielsen class or a different component of $\sH$.  You may compatibly apply $\gamma$ to any subspace $\sH^*$ of $\sH'$, extending it to the corresponding spaces of covers over $\sH^*$. Now apply it to the point $\bp$. Since $\bp$ has coordinates in $\bQ$, $\gamma$ extended to a representing cover will be in the same Nielsen class, contrary to our assumption about $\gamma$. So, $^\gamma\sH'$ is a further, distinct, component of $\sH$, which also contains $\bp$. That gives two components of $\sH$ through $\bp$, contrary to $\sH$ being a manifold.  

\begin{rem} \label{normalvssing} The argument above that $\bp$ can have coordinates in no field smaller than $K$ requires only that $\sH$ is a normal variety. \end{rem} 

\subsubsection{Spaces of polynomials} \label{polySpaces} Consider a family of covers, with the notation below Prop.~\ref{absFamCovers}. 
Since the  fibers of the map $\Phi$ are curves, it may happen that we could choose $m=2$. This would  be representing the fibers $\Phi_\bp: \sT_\bp \to  \bp\times \prP^1_z$ as the zero set in projective 2-space with coordinates $(x_0,x_1,x_2)$ of a homogenous polynomial, $f(x_0,x_1,x_2)$,  and   the $z$ variable identified to $x_1/x_0$.  For families of genus 0 curves, we might even hope for $m=1$. 

Problems about polynomial covers (and others) often call for restricting to closed paths in $U_r$ that keep a branch point, say   $z_r=\infty$, fixed. Appropriate to Davenport pairs is  the following situation. 

Suppose $\phi: X \to \prP^1_z$ is a cover over $K$. Assume there is a {\sl unique\/} totally ramified place $x_\infty$;  we assume it is over $z_r$. Then, $z_r$ has definition field $K$. By applying a linear fractional transformation we may assume $z_r=\infty$. Further, in the expansion of the most negative term of $\phi$ around $x_\infty$, by changing $\phi$ to $a\phi$ we may assume that term has coefficient 1. 

If, in addition, we assume $X$ has genus 0,  then some isomorphism of $X$ with $\prP^1_w$ over $K$ sends $x_\infty$ to $w=\infty$.  That $K$ rational point $x_\infty$ is essential for this. With $\deg(\phi)=n$, rename $\phi$  as a monic polynomial in $w$: $P : \prP^1_w \to \prP^1_z$ over $K$. Still,  the isomorphism isn't yet unique.  

There is still a polynomial {\sl collection\/}, all {\sl affine} equivalent to $\phi$ and subject to choices we've already made: \begin{equation} \label{affinePolClass} \{P(e^{2\pi i j/n}w+b')+b\}=\tilde P_\phi, \text{  $j$ an integer, $b',b$ any constants}.\end{equation} 
Given $P$ over $K$, setting the penultimate coefficient to 0 determines $b'$ (still in $K$). 

Now we get to subtle normalizations when applied to Davenport pairs. 
Suppose $K\le \bR$. Then, if we name the zeros of $P(w)=z$ as $\row w n$, given as expansions in $1/z^{\frac 1 n}$, we can also normalize the connection between $w_1$ and $w$, by associating that expansion with a {\sl tangential base\/} point (as, say, in \cite[opening of \S15]{Del89}). That is, restrict values of $z$ to a sector 
$$ \{re^{i\theta}\mid r<\epsilon, -\pi< \theta<+\pi\} $$ 
\begin{triv} \label{tangbasept} and choose $j$ so that by renaming  $\zeta_n^jw$ to be $w$, it has  its values  lying in a sector around the positive real axis near $\infty$.\end{triv}  Yet, none of the Davenport pairs has definition field  $K\le \bR$.

Here is another normalization that doesn't work for Davenport pairs. 
\begin{triv} \label{const0} We can choose $b$ so the constant term of $P$ is 0.\end{triv}  But this would violate the condition of conjugacy between Davenport pairs  $f$ and $g$ in \eql{davpairs}{davpairs3}. 
So, the topic of polynomial normalization continues in  \S\ref{writeEquats}. 

Consider a family $\Phi: \sT\to \sF\times \prP^1_z$ of $r$-branch point covers. Assume each fiber $\Psi_\bp: \sT_\bp \to  \bp\times \prP^1_z$ has genus 0, with exactly one totally ramified place over $z=\infty$. 
\begin{defn} \label{polyfam} Call $\Phi$ a {\sl family of  polynomial covers\/} if for some polynomial  $P(\bp,w)$ in $w$ with coefficients in the coordinates $\bp\in \sF$,  each fiber of  $$P: \sF\times  \prP^1_w \to  \sF\times  \prP^1_z \text{  by }(\bp,w)\mapsto P(\bp,w)$$ represents the corresponding fiber of $\Phi$. \end{defn} 

\subsubsection{Branch cycles for $j$-line covers} \label{bcUr} 
Consider $ U^r$, the set of {\sl ordered\/} (unlike $U_r$ in \S\ref{HurMonUr}) distinct points on $\prP^1_z$. Two groups act on $U_r$: $\PGL_2(\bC)$ acting the same on each slot; and $S_r$ permuting the coordinates.  
For general $r$ the {\sl configuration space\/} $J_r$ for {\sl reduced\/} absolute equivalence is the quotient of $U^r$ by these  commuting actions.  That is, $$\PGL_2(\bC)\backslash U^r/S_r=\PGL_2(\bC)\slash U_r\eqdef J_r.$$ The parameter space for this equivalence, $$\sH(G,\bfC)^{\abs,\red}=\sH(G,\bfC)^{\abs}/\PGL_2(\bC),$$ is  the normal variety  given by extending the action of $\PGL_2(\bC)$ on $U_r$ to  $\sH(G,\bfC)^\abs$  (as in \S\ref{cont7}).  The result  has a  natural map to $J_r$. 

The classical $j$-line minus the point at $\infty$  is $J_4$ ($r=4$). The cases of Davenport families where $r=4$ are included. They have reduced parameter spaces  $\sH(G,\bfC)^{\abs,\red}$ whose components are  each upper half-plane quotients by a finite index subgroup of $\PSL_2(\bZ)$. Each has a natural normal (since $r=4$, nonsingular) compactification, $\bar \sH(G,\bfC)^{\red}$, as a cover of the $j$-line (references below). 

Designate the whole $j$-line by $\prP^1_j$, with the variable $j$ normalized to have $j=0$ and $j=1$ as the two possible finite branch points of upper half-plane quotients. 
We can compute explicitly the components, their ramification (so their genuses), and geometric monodromy  as $\prP^1_j$ covers.  For that we use \eqref{gammas}  for its branch cycles. Define $\sQ''_4$ to be the (normal) subgroup of $H_4$ generated by $\sh^2$ and $q_1q_3^{-1}$. 

\begin{defn} The reduced (absolute) Nielsen  class of $(G,\bfC)$ is  $$\ni(G,\bfC)^{\abs}/\sQ''_4=\ni(G,\bfC)^{\abs,\red}.$$ For completeness, there is a definition when $r\ge 5$, but then $\sQ''_r$ is trivial, and reduced classes are the same as Nielsen classes. \end{defn}

The action of $H_4$ on reduced Nielsen classes factors through the {\sl mapping class group}: $\bar M_4\eqdef H_4/\sQ''\equiv
\PSL_2(\bZ)$ \cite[Prop.~4.4]{BFr02}. \cite[\S2.7]{BFr02}  makes this identification by expressing certain generators from the images of words in the $q_i\,$s: 
\begin{equation} \label{gammas} \begin{array}{rl} & \lrang{\gamma_0,\gamma_1,\gamma_\infty},  
\gamma_0=
q_1q_2,\gamma_1=\sh= q_1q_2q_3=q_1q_2q_1\!\! \mod \sQ'',\gamma_\infty = q_2, \\ &\text{satisfying the product-one relation: }\gamma_0\gamma_1\gamma_\infty=1.
\end{array}\end{equation}  
Note:  \eqref{hurrelations} appears dramatically in these identifications.     
For example,  see that $\gamma_0$ (resp.~$\gamma_1$) has order 3 (resp.~2) by successively applying   \eql{hurrelations}{hurrelations2} and \eql{hurrelations}{hurrelations3} \!\!$\!\!\mod \sQ''$: 
\begin{equation} \begin{array}{rl} &q_1q_2q_1q_2q_1q_2=q_1q_2q_1q_1q_2q_1=q_1q_2q_3q_3q_2q_1=1;  \\
\text{ (resp. } & q_1q_2q_3q_1q_2q_3=q_1q_2q_1q_1q_2q_1=\dots=1).\end{array}\end{equation} 

\subsection{Applying Riemann-Hurwitz} \label{j-lineRH} Let $O$ be an orbit of $\bar M_4$ on $\ni(G,\bfC)^{\abs,\rd}$. Then, $O$ corresponds to a reduced Hurwitz space component $\sH_O$.  There is a unique non-singular completion, $\bar \sH_O$, that is a $j$-line cover. 
Now we interpret R-H \eqref{*10}: $(\gamma_0,\gamma_1,\gamma_\infty)$ acting on $O$ 
$\Leftrightarrow$  branch cycles for this cover \cite[Prop.~4.4]{BFr02}. 

\begin{edesc} \label{cuspct} \item  \label{cuspcta}  Ramified points over 0 $\Leftrightarrow$ orbits of
$\gamma_0$.
 \item  \label{cuspctb}  Ramified points over 1 $\Leftrightarrow$ orbits of $\gamma_1$.
\item  \label{cuspctc}  Use one representative $\bg\in \ni(G,\bfC)^{\inn,\rd}$ for each $\text{Cu}_4=\lrang{q_2,\sQ''}$ orbit.  Then,  $\ind(\gamma_\infty)$ is the sum 
$|(\bg)\text{Cu}_4/\sQ''|-1$ over those orbits. 
\end{edesc}

The points of $\bar \sH_O$ lying over $j=\infty$ are the {\sl cusps\/} of $\sH_O$ and these correspond to the $\text{Cu}_4$ orbits on $O$ \cite[Prop.~2.3]{BFr02}. The meaning of an absolute reduced family of covers in a given Nielsen class $\ni=\ni(G,\bfC)^{\abs,\red}$, with parameter space $\sF$ is analogous to the inner reduced family case of \cite[\S 4.3]{BFr02}. It is a sequence of morphisms of normal spaces $\Phi: \sT\to \sB \mapright{\Gamma}  \sF$, with these properties: 

\begin{edesc} \label{redfam} \item for each $\bp\in \sF$, $\sB_\bp$ is isomorphic to $\prP^1_z$ (over $\bC$); and 
\item  \label{redfam2}  the fiber $\Phi_\bp: \sT_\bp\to \sB_\bp$ is a cover in  the Nielsen class.  \end{edesc} Then, \eql{redfam}{redfam2} gives a natural morphism  $\Psi: \sF\to J_r$ by $\bp\mapsto \Psi(\bp)$,  the $\PGL_2(\bC)$ class of the $\Phi_\bp$ branch locus. We call $(\Phi,\Gamma)$ a family in the reduced Nielsen class. 

The goal is to compare this with the natural map $\Psi_{G,\bfC}: \sH(G,\bfC)^{\abs,\red}\to J_r$ in the following style. Suppose there is a family satisfying \eqref{redfam}, $$\Phi_{G,\bfC}: \sT_{G,\bfC} \to  \sB_{G,\bfC}  \mapright{\Gamma_{G,\bfC}} \sH(G,\bfC)^{\abs,\red} \text{with  $\sH(G,\bfC)^{\abs,\red}$ replacing $\sF$.}$$  (Say, if reduced fine moduli holds, as below.) Then, we can compare the pull back -- fiber product  -- of  this family over $\Psi$ with the family over $\sF$.

Assume $\sH$ is a component of $\sH(G,\bfC)^{\abs}$ corresponding to an $H_4$ orbit $O$ on $\ni(G,\bfC)$, and $\sH^\red$ is its corresponding reduced space. 
Here is the two-parted fine-moduli result -- analog of Prop.~\ref{absFamCovers} for reduced Hurwitz spaces -- for $r=4$  \cite[Prop.~4.7]{BFr02}. For the map $\Psi$, denote the locus over $J_4\setminus \{0,1\}$ by   $\sF'$, with  $(\sH^{\red})'$ the the pullback of $\sH^\red$ over $\sF'$. 
\begin{edesc} \label{finem} \item  \label{finem1} {\sl b(irational)-fine}: $(\sH^{\red})'$ parametrizes a {\sl unique\/} family (up to equivalence) if and only if  restricting $\sQ''$ (\S\ref{bcUr}) to $O$ has length 4 orbits. 
 \item \label{finem2} {\sl e(lliptic)-fine}: Same conclusion with $\sH^{\red}$ replacing $(\sH^{\red})'$, if, in addition to \eql{finem}{finem1} ,  $\gamma_0'$ and $\gamma_1'$ have no fixed points.  \end{edesc} 

\S\ref{App.5} computes  the data in \eqref{cuspct}  for the families of Davenport polynomials when degree $n=7$ based on \eqref{deg7BCs}. For each Nielsen class, there is   just one component.  There are two Nielsen classes corresponding to the two conjugacy classes of 7-cycles in $\PGL_3(\bZ/2)$.  We find that $\bar \sH(G,\bfC)^{\red}$ has genus 0. Using the fine moduli statements of \eqref{finem}, we then know over which fields there are Davenport pairs of  degree 7 (as in \eql{genus0mon}{genus0mon1}). Note:  \eql{finem}{finem1} holds, but \eql{finem}{finem2}  does not. 

\subsection{Three genus 0 families of Davenport Pairs} \label{App.5} 
Applied to polynomial covers with monodromy given in the $\PGL$ groups over finite fields, \S\ref{secVI.4} shows that only for $n=7$, 13 and 15,  could we have $r=4$ for Davenport pairs. (In all other cases $r=3$.) To illustrate what happened in these three cases we do just $n=7$. 

\subsubsection{Davenport pairs of degree 7} Let $\sD$ denote the difference set  $\{1,2,4\} \mod 7$ for  $n=7$ of \S\ref{secVI.5}. There are two conjugacy classes of 7-cycles, ${}_1\C_{\infty}$ and ${}_2\C_{\infty}$ in $\PGL_3(\bZ/2)$. That gives two sets of conjugacy classes ${}_i\bfC$, $i=1,2$,  determined by 3 involutions and a 7-cycle. Each  defines a Nielsen class. The computation for each is the same since an outer automorphism takes ${}_1\bfC$ to ${}_2\bfC$. 

For reduced classes mod out by $\sQ''$. Here is how the b-fine moduli property \eql{finem}{finem1} follows.  Given $\psigma\in  \ni(\PGL_3(\bZ/2), \bfC)$, a unique element of $\sQ''$ changes it to have the 7-cycle in the 4th position. Take it as  $\sigma_\infty^{-1}=(1\,2\,\dots\,7)^{-1}$, compatible with \S\ref{questionsDPs}.  Our permutations  act on the right of integers. We use $T_1$ (resp.~$T_2$) for the representation of $\PGL_3(\bZ/2)$ on points (resp.~lines). 

Expression \cite[(4.14)]{Fr95a} lists the reduced absolute Nielsen classes and \eqref{deg7BCs}  lists  their first three entries, the three finite branch cycles $(\sigma_1,\sigma_2,\sigma_3)$
for a polynomial $h$.  There are exactly 7, denoted $Y_1,\dots,Y_7$,  
up to conjugation by $S_7$: 
\begin{equation} \label{deg7BCs} \begin{array}{ll} Y_1: ((3\,5)(6\,7),((4\,5)(6\,2),(3\,6)(1\,2));&  
\!\!\!\!\!\!Y_2:((3\, 5)(6\, 7),(3\, 6)(1\, 2),(3\, 1)(4\, 5)); \\
Y_3: ((3\, 5)(6\, 7),(1\, 6)(2\, 3), (4\, 5)(6\, 2)); &    
\!\!\!\!\!\!Y_4:((3\, 5)(6\, 7),(1\, 3)(4\, 5),(2\, 3)(1\, 6)); \\
Y_5:((3\, 7)(5\, 6),(1\, 3)(4\, 5),(2\, 3)(4\, 7));& 
\!\!\!\!\!\!Y_6:((3\, 7)(5\, 6),(2\, 3)(4\, 7),(1\, 2)(7\, 5)); \\
Y_7: ((3\, 7)(5\, 6),((1\, 2)(7\, 5),(1\, 3)(4\, 5)). &\end{array} \end{equation} 
  
To simplify the notation relabel $Y_i$ as $i'$ and have the $q_i\,$s 
act on $1',\dots,7'$. (This isn't the action through
the representations $T_1$ and $T_2$.)   Denote the action of the $\gamma\,$s in \eqref{gammas} on $1',\dots,7'$ by $\gamma'\,$s. We get (see \cite[\S5]{Fr05a}) for $n=13$): 
$$q_1= (3'\, 5'\, 1')(4'\, 7'\, 6'\, 2'), \text{ and } 
q_2= (1'\, 3'\, 4'\, 2')(5'\, 7'\, 6').$$ Our action of the $q_i\,s$ is on the right. Therefore, $$\gamma_0'=(1'\,4'\,6')(3'\,7'\,5'), \gamma_1'=(1'\,7')(2'\,4')(3'\,6') \text{ and }\gamma_\infty=(1'\,2'\,4'\,3')(5'\,6'\,7')^{-1}.$$ Now we give the main results about $\sH(\PGL_3(\bZ/2),{}_j\bfC)^{\abs,\red}\eqdef \sH_{{}_j\bfC}$, $j=1,2$. As previously,  these  are upper half plane quotients, with their compactifications, $\bar \sH_{{}_j\bfC}$,  $j$-line covers. So,  it is appropriate to ask if they are modular curves.  

\begin{thm} \label{vectorBundle} The curves $ \bar \sH_{{}_j\bfC}$, $j=1,2$, have genus 0. The geometric (or arithmetic) monodromy group of each over $\prP^1_j$  is $S_7$. As reduced Hurwitz spaces they have b-fine, but not fine moduli. These are not modular curves. 

As moduli of Davenport pairs, $ \sH_{{}_1\bfC}$ is conjugate over $\bQ(\sqrt{-7})$ to $ \sH_{{}_2\bfC}$.   Each field containing $\bQ(\sqrt{-7})$ has infinitely many reduced inequivalent Davenport pairs.  Also, these reduced Hurwitz spaces support an explicit family of polynomial covers. 

Finally, $\sH(\PGL_3(\bZ/2),{}_j\bfC)^{\abs,\red}$ also identifies as an inner Hurwitz space.  So,  the two spaces for $j=1$ and $2$ are the same, and isomorphic to $\prP^1_{t}$ ($t=t_7$ in the statement of Prop.~\ref{DS6}) over $\bQ$. \end{thm} 

\begin{proof} Compute the genus $g_{{}_1\bfC}$ of $ \bar \sH_{{}_1\bfC}$ by applying R-H to its branch cycles, $\gamma_0,\gamma_1,\gamma_\infty$ as a $j$-line cover: 
$$ 2(7+g_{{}_1\bfC}-1)=\ind(\gamma_0')+\ind(\gamma_1') + \ind(\gamma_\infty')=4 + 3 +(2+3)=12.$$ So, $g_{{}_1\bfC}=0$. That the monodromy group is $S_7$ is also quick: It is a degree 7 group containing a 3-cycle, $\gamma_\infty^4$, and a 4-cycle, $\gamma_\infty^3$. 

We have already noted above that \eql{finem}{finem1} -- b-fine moduli -- holds. The condition for fine moduli is that neither 
$\gamma_0'$ nor $\gamma_1'$ have fixed points. In our case, however, both do, so fine moduli doesn't hold.  If  $\sH(\PGL_3(\bZ/2),{}_j\bfC)^{\abs,\red}$ were a modular curve, its monodromy group would be a quotient of $\PSL_2(\bZ/N)$ for some integer $N$. Indeed,  $N=12$ would work, according to Wohlfahrt's Theorem \cite{Wo64}. Just the order of $\PSL_2(\bZ/12)$  shows it is not divisible by 7, so this is impossible. 

The normalizing group of $ \PGL_3(\bZ/2)$ in its action on the points of projective space is just $ \PGL_3(\bZ/2)$. Apply \eqref{FrV}. Then, $$\sH(\PGL_3(\bZ/2),\bfC_j)^\inn\to \sH(\PGL_3(\bZ/2)\bfC_j)^\abs$$  has degree the order of that normalizer modulo $ \PGL_3(\bZ/2)$. So, the degree is 1, identifying $\sH(\PGL_3(\bZ/2),\bfC_j)^\inn$ and $ \sH(\PGL_3(\bZ/2)\bfC_j)^\abs$. The former, however, is the space of Galois closures of the covers in the latter, according to \eql{FrV}{FrVa}. 

As noted in \S\ref{polySpaces}, we handle the normalizations to produce a family of polynomials in \S\ref{writeEquats}. Apply Thm.~\ref{DS2} to identify the Galois closures of the covers parametrized by $\sH(\PGL_3(\bZ/2)\bfC_j)^\abs$ for $j=1,2$. That is, $\sH(\PGL_3(\bZ/2),\bfC_j)^\inn$, $j=1,2$, are exactly the same Hurwitz spaces, which now identify with the absolute versions of those spaces.  

Prop.~\ref{absHurwitzSpaceDef} gives the precise definition field of the families of Davenport polynomials as $\bQ(\sqrt{-7})$, but it gives the definition field of the inner Hurwitz space as $\bQ$. Therefore, as a cover of $\prP^1_j$, the inner space has definition field $\bQ$.  

Further, we can identify rational points on this genus 0 space. For example,  $\gamma_\infty'$ has a 3-cycle and a 4-cycle. This indicates points of ramification index 3 and 4 over $j=\infty$ by applying the general idea of \S\ref{secVI.4} to these $j$-line covers as given in \S\ref{bcUr}). Any element  $\alpha \in G_\bQ$ keeps $\infty$ fixed. So, it must permute the points of the fiber over $\infty$ moving them to points having the same ramification indices over $\infty$. The uniqueness of such ramification indices  means both points have definition field $\bQ$. A genus 0 curve over a (characteristic 0) field $K$ with a $K$ point is well-known to have definition field $K$. This concludes our proof. \end{proof} 

\subsubsection{Identification of a space of bundles} The inner Hurwitz space of Thm.~\ref{vectorBundle} (through Thm.~\ref{DS2})  turns  Davenport pairs into bundles for a degree $n$ representation of their  geometric monodromy groups.  This interpretation supports Conj.~\ref{genus0Conj}. 

Any degree $n$ (complex analytic) cover $\phi: X\to Z$ (of nonsingular varieties) defines a rank $n$ bundle, as its corresponding {\sl direct image sheaf}. Briefly: Over a (simply-connected) coordinate patch $U$ on $Z$, form the local structure sheaf $\sO_U$, and similarly form the structure sheaf $\sO_{\phi^{-1}(U)}$ over $U$. Then, from flatness (\S\ref{roleFlatness}), $\sO_{\phi^{-1}(U)}$  is a free, rank $n$, module over $\sO_U$. That means, the structure sheaf $\sO_\phi$ is a locally free, rank $n$ bundle over $\sO_Z$. 

Apply this to a Davenport pair $(f,g)$, so there are two such rank $n$ bundles $\sO_f$ and $\sO_g$ over $\sO_{\prP^1_z}$. Actually, these spaces identify as the quotient of the regular representation of the Galois group of the covers that gives the permutation representation of the degree $n$ covers. These representation spaces identify in the case of Davenport pairs from the transition matrix of Thm.~\ref{DS2}. 

\subsubsection{$n=7, 13$ and $15$} 
With slight variation from their having more than two conjugacy class collections $\bfC$, Thm.~\ref{vectorBundle} applies also to $n=13$ and $15$.  
\cite[Thm. 8.1 and 8.2]{Fr99} shows $n=13$ works similarly, and as easily. Here the Hurwitz space is a degree 13 -- again the same as $n$ -- cover of $J_4$. The significant difference is that the multiplier of the difference set $\sD=\{1, 2, 4, 10\}$ has order 3. So, the definition field $K$ for these spaces is the degree 4 extension of $\bQ$ inside $\bQ(e^{2\pi i/13})$. Thus, there are two {\sl pairs\/} of conjugate Davenport pairs in this case [Fr05a, ¤3.4]. 

Consider the collection $\sC_{\PGL_\infty,r}$ of  reduced Hurwitz spaces of $r$-branch point covers with projective linear monodromy groups. We do not assume the covers in the Nielsen classes have genus 0.  

\begin{guess} \label{genus0Conj} Do  only finitely many of the spaces in $\sC_{\PGL_\infty,4}$ ($r=4$)  have genus 0? \end{guess} 

Finally, notice that there are a great many other Nielsen classes on which there is only one possible difference in the final conclusions that occurred for Davenport pairs.  Assume, in addition to the conditions for $\sC_{\PGL_\infty,4}$, that 
\begin{triv} \label{DavExtra} exactly one class of  $\bfC$ is an $n$-cycle, in the notation previously. \end{triv} Denote the elements of $\sC_{\PGL_\infty,4}$ satisfying \eqref{DavExtra} by $\sC_{\PGL_\infty,4,\C_\infty}$. Then, you can apply the BCL and find that covers won't be defined over $\bQ$. Just as in the Davenport cases, you can compute the genus of absolute components of elements in $\sC_{\PGL_\infty,4,\C_\infty}$. Yet, it is likely the components won't have genus 0.  Further, there may be more than one component. One point of \S\ref{missmultiplier} is to tell you something about our knowledge of such computations. 

\begin{rem}[Infinitude of $\sC_{\PGL_\infty,4,\C_\infty}$] \label{infPGL} We make use of this exercise in \S\ref{3etale}. Go through the production of the Nielsen classes of genus 0 covers in \S\ref{cont7}, but drop the condition of genus 0. Show there are infinitely many possible Nielsen classes. 
\end{rem} 

\section{The significance of Davenport's Problem} \label{secVII} We use what came from Davenport's Problem, and others solved by the monodromy method, to reconsider truly  general problems that arose around them. Of necessity I review the work of many others, by efficiently  using the previous sections. \S\ref{secVII.1} gives conclusions on the genus 0 problem, while \S\ref{writeEquats} considers the biggest bug-a-boo from RET, that it's not done with algebraic equations.  

Then, \S\ref{secVII.2} looks at  the relation between {\sl Chow motives\/} and {\sl Galois stratification\/} using Monodromy Precision \S\ref{MPres}.  \S\ref{secVII.3}  motivates why going beyond the simple group classification  will require new techniques.  For this we return to the comment from \cite{So01} on the groups that occur \lq in nature\rq\ being close to simple groups. Finally, \S\ref{secVII.4} considers a different overview of RET, though still based on what came from Davenport's problem. 

\subsection{The Genus 0 Problem} \label{secVII.1} Solving Davenport's problem produced some lucky lessons.  Most propitious was my interaction with John Thompson, walking to lunch early in Fall 1986 after I arrived at the U.~of Florida. 

\subsubsection{Evidence for the genus 0 problem} I gave Thompson my conviction of the specialness of genus 0 monodromy groups. My support came much from \cite{Fr80}. 

\begin{edesc}  \label{genus0} \item  \label{genus01} The product-one condition (\eql{cycleconds}{cycleconds2} of \S\ref{prodBCs})  together with genus 0 limited -- but didn't annihilate -- the groups arising in Davenport's Problem, and the Hilbert-Siegel problem (as in \cite{Fr74a}). 
 
\item  \label{genus02}  As geometric monodromy,  cyclic,  dihedral,  $S_n$ and $A_n$,  and closely related, groups all appeared often  when the problems had no further constraints on conjugacy classes.
 \end{edesc} 

{\sl Comments on\/} \eql{genus0}{genus01}: My main question to John was whether he thought that genus 0, product-one and primitivity would be sufficient to limit exceptional arisings of monodromy groups, and what exactly exceptional would be.

\subsubsection{Comments on  \eql{genus0}{genus02}--Cyclic composition} \label{limitDih} I document the surprising complication of groups close to dihedral.  Tchebychev polynomials have dihedral geometric monodromy and their Galois closures are defined over the maximal real field in $\bQ(e^{2\pi i/n}$). Capturing how exceptional this was proved Schur's conjecture (\S\ref{secI.0}). It and Serre's OIT still are the main producers of exceptional covers (\S\ref{MPres}). 

The OIT also gives dramatic distinctions between arithmetic and geometric monodromy. It is convenient to quote \cite{Se68}, though Serre's program wasn't quite complete there. \cite[\S6.2, esp.~Prop.~6.6]{Fr05b} explains all of the following.   This was especially dramatic because in Serre's $\GL_2$ case the degree $p^2$ covers, with $p$ prime, have tiny (resp.~large) geometric (resp.~arithmetic) monodromy $(\bZ/p)^2\xs \bZ/2$ (resp.~an extension of the geometric group by $\GL_2(\bZ/p)/2$).  

Further, the degree $p^2$ covers are given by rational functions. These reveal one profound distinction between compositions of rational functions and polynomials.  In Lem.~\ref{decompPoly} we saw that $f\in K[x]$ that decomposes over $\bar \bQ$ already decomposes in $K[x]$. Myriad examples, however, from the OIT give rational functions of degree $p^2$ indecomposable over $K$ (even over $\bQ$), but decomposable over $\bar \bQ$. 

\begin{triv} \label{decompExc} Excluding finitely many degrees of rational functions, but allowing any number field $K$,  the OIT produces all such examples.  \end{triv}  

The groups that appear in \eqref{decompExc} are not related to those in  \eql{genus0}{genus01}. There are many primitive exceptional genus 0 groups. It is a finite number. Yet, consider what went into showing the finiteness part of \eqref{decompExc} in  \cite[Chap.~3]{GMS03}.  For a particular problem,  apropos \S\ref{secVII.3}, even those who know the classification well will drown trying to navigate the documentation without finding some geometry and/or function theory  like that we used in handling Davenport's problem. 

\subsubsection{Comments on  \eql{genus0}{genus02}--Alternating composition} \label{limitAlt} Almost any graduate book in algebra has regular realizations (over $\bQ$, \S\ref{BCLtreat}) of dihedral groups. Though,  as \S\ref{limitDih} shows, in trying to realize them with genus 0 covers over $\bQ$ you  might have dihedral geometric monodromy, but much larger arithmetic monodromy. 

Similar occurs with (near) alternating groups as following Res.~\ref{exBCL} for $(A_n,S_n)$-realizations. The alternative, is $(A_n,A_n)$-realizations (regular $A_n$ realizations). Hilbert's first application of his Irreducibility Theorem was to finding regular $A_n$ realizations \cite{Hi1892}. For example, as \cite{Mes90}, and \cite[Chap.~9]{Se92} show there is an abundance of such retional function $f$, even extending to Spin$_n$ regular realizations. 

One take on the Irreducibility Theorem is that it must be obvious. Yes, there are easy proofs, say \cite[Thm.~12.7]{FrJ86}$_1$, of its first incarnation. 

\begin{prop}[HIT] Suppose  $m(z,w)\in \bQ[z,w]$ is irreducible. Then, for infinitely many $z_0\in \bZ$, $m(z_0,w)$ is irreducible as a polynomial in one variable. \end{prop} 

The first Hilbert-Siegel Problem puts a constraint on $m(z,w)$.  It has the form  $m(z,w)=f(w)-z$, $f\in \bQ[w]$. Yet, the monodromy method enters because the conclusion is independent of the degree of $f$. Denote by $\sV_f$ (resp. $\sR_f$) the values assumed by $f$ on $\bZ$ (resp.~ the $z_0\in \bZ$ such that $f(w)-z_0$ factors over $\bQ$). 

\begin{prop}[1st Hilbert-Siegel \cite{Fr74a} Prob.] \label{1stHS} Suppose $f\in \bQ[w]$ is indecomposable and   $\sR_f\setminus \sV_f$ is infinite. Then, all but finitely many of elements in $\sR_f\setminus \sV_f$ fall in the values of  $g\in \bQ(x)$ where $f(x)-g(y)$ factors as one of two types. 
\begin{edesc} \label{HS} \item \label{HSa}  Either $g\in \bQ[x]$; or
\item \label{HSb}  with $\deg(f)=n$, $\deg(g)=2n$ and a branch cycle $\sigma_\infty$ for $g$ over $\infty$ has the shape $(n)(n)$. \end{edesc} 
\end{prop}

The arithmetic reduction came through Siegel's famous description of curves with $\infty$-ly many quasi-integral points \cite{Si29}. You  can change all formulations referring to $z_0\in \bZ$ to be about quasi-integral points (only finitely many primes allowed as divisors of denominators). We previously handled \eql{HS}{HSa} under the solution of Schinzel's problem. We conclude this subsection with the upshot of the story for the new cases, \eql{HS}{HSb}. In the style of Thm.~\ref{DS2}, \cite[Cor.~2]{Fr74a} gives the exact branch cycle conditions. These come as Nielsen class conditions for covers $f: \prP^1_x\to \prP^1_z$ and $g: \prP^1_y\to \prP^1_z$ having the same Galois closure groups $G$, and respective permutation representations $T_f$ and $T_g$. Use the notation of \S\ref{secIV.1}. 

\begin{edesc} \item $T_f$ is doubly transitive;  $T_g$ is primitive,  but not doubly transitive. 
\item $T_g$ restricted to the stabilizer, $G(1,T_f)$, in $T_f$ is intransitive. 
\item The absolute Nielsen classes for both permutation representations have genus 0 (as in Riemann-Hurwitz, a la \eqref{*10}). 
\item  The class for ramification over $\infty$ in the cover for $T_f$ (resp.~$T_g$)  has cycle type $(n)$ (resp.~$(n)(n)$). \end{edesc}

  \begin{prop} \label{DeFrMC}  \cite[Prop.~1.3]{DeFr99}: The only possible degree for $f$ satisfying \eql{HS}{HSb} is 5. All possible $f\,$s derive from one Nielsen class (below) with $r=4$. Among the $f\,$s over $\bQ$ in this Nielsen class,  infinitely many have $\sR_f\setminus \sV_f$ is infinite. \end{prop} 
 
 The Nielsen class comes from the standard representation, $T_f$, of $G=S_5$. The conjugacy classes are  $\bfC=\bfC_{52^22_d}$: 5-cycle, 2-cycles repeated twice, and the class, $\C_{2_d}$, of (2)(2) type.  Denote this Nielsen class as $\ni(S_5,\bfC)$. 
 
The representation $T_g$ is from the action on the 10 unordered pairs of integers from $\{1,2,3,4,5\}$. Then, the $g$ cover Nielsen class comes from  applying $T_g$ to $\bfC$, giving $T_g(\bfC)=\bfC_{5_d,2_t^2,2_q}$: respective classes of type $(5)(5)$, $(2)(2)(2)$ repeated twice and $(2)(2)(2)(2)$. Denote the Nielsen class as $\ni(T_g(S_5),T_g(\bfC))$. 
 
\S\ref{bcUr} discusses the space $U^r$ of ordered branch points on $\prP^1_z$. You can order some attached to certain conjugacy classes, and not others, to consider spaces between $U^r$ and $U_r$. Order the two branch points attached to the 2-cycle conjugacy classes. Denote the Hurwitz space by $\sH$, and the pullback with that ordering as $\sH^*$.  The corresponding spaces for $T_g$,  $\sH_g$ and $\sH_g^*$, actually identify with $\sH$ and $\sH^*$.  
 \begin{edesc}  \label{deg10} 
 \item  \label{deg10a} All three branch point covers have branch cycle descriptions from coalescing those in the Nielsen class $\ni(S_5,\bfC)$. 
 \item \label{deg10b} Both of the spaces $\sH_g$ and $\sH_g^*$ have a dense set of $\bQ$ points. 
 \item \label{deg10c} For a dense set of  $\bp\in \sH_g(\bQ)$, the total space $\sT_g\to \sH_g\times \prP^1_z$ has fibers  $\sT_{g,\bp}$ that are conics in $\prP^2$ without any $\bQ$ points. 
 \item \label{deg10d} For all points $\bp\in \sH^*_g(\bQ)$, the pullback fibers of \eql{deg10}{deg10c}  represent degree 10 rational functions over $\bQ$. 
 \end{edesc} 

Two out of three of the delicate diophantine issues are handled on purely Nielsen class terms, without explicit coordinates. The 1st: The rational function $g_\bp$ corresponding to $\bp\in \sH_g^*$ comes by taking one of the two branch points, $z_{1,\bp},z_{2,\bp}$, in the cover for $\bp$ corresponding to the classes $T_g(\C_2)$. The three 2-cycles above, say, $z_{1,\bp}$ correspond to three points (as in \S\ref{secVI.4}) -- of ramification index 2 over $z_{1,\bp}$. Those three points sum to  an odd degree divisor on the genus 0 cover $\phi_\bp:X_\bp \to \prP^1_z$. An odd degree divisor on a genus 0 curve is well-known to produce an isomosphism of it with $\prP^1_y$ over its field of definition. 

The 2nd diophantine issue meets the requirement, for applying Siegel's Theorem, that the two points over $z=\infty$ are {\sl real\/} conjugate (defined over $\bQ(\sqrt{5})$ \cite[Cor.~2.2]{DeFr99}). Many examples in \cite[\S4]{DeFr99} illustrate the well-developed theory of real points on covers in \cite[\S2]{DeFr90}; what we called Siegel-N\'eron problems. 

Finally, the issue not addressed until \cite{DeFr99},  was to show among the $g_\bp$ were some with $\sR_f\setminus \sV_f$ infinite.  \cite[\S4.2]{De99} has an exposition concentrating on this arithmetic point phrased thus: Find when a Siegel family has a dense set of fibers whose value sets intersect a fractional ideal infinitely often. 

This is the only place I know where explicit coordinates accomplished something not done without them. The issue is whether it is possible to answer such a question based only on calculating with Nielsen classes defining the Siegel family. We include using the BCL,  braid group action, {\sl lifting invariants\/} (as in \cite[\S5.4]{BFr02}). 

\cite[\S4.4]{De99}   shows we often can expect affirmative results, like \cite[Th.~3.14]{DeFr90b} and the many examples of  \cite[\S3.6-\S3.7, and \S 4]{DeFr90b}, when covers in the family have genus 0. An ingredient for this is  \cite{LSc80} (over $\bQ$; over a general number field in \cite[p.~211]{Sc82}), comparing specializations at $\bQ$ fibers with what happens at the generic point. As in \S\ref{secVII.4}, I knew of this from my UM education. 

\subsubsection{Thompson's response and the program} \label{Thompsonresp}  Immediately John confessed to being \lq\lq seized\rq\rq\ by the problem. His response was that we shouldn't limit it to polynomial covers. Rather, include  indecomposable rational functions (genus 0 covers). In place, however, of considering constraints and guessing what precisely the exceptional permutation representations might be, he suggested showing that all composition factors of the geometric monodromy groups would be cyclic or alternating. Then, the exceptions would come from just finitely many simple groups -- outside $A_n\,$s and $\bZ/p\,$s -- appearing among these composition factors. 

All statements related to exceptional covers (\S\ref{secI.0}; like the interpretation of dihedral groups as the essence of Serre's OIT in \S\ref{limitDih}), suggested aiming at  actual monodromy groups rather than composition factors. Still, what John proposed generated data to guide finding which actual monodromy groups (and corresponding permutation representations) were not exceptional. Especially since we were certain to get some close to, but not quite, alternating group surprises. 

He proposed we work on the problem together. My heart was in algebraic equations. I suggested  Bob Guralnick as far more appropriate. Here was the upshot. 

Peter M\"uller produced a definitive classification of the polynomial monodromy, including -- a la what happened in Davenport's Problem -- a list of the polynomial monodromy that arose over $\bQ$ \cite{Mu95}. Davenport's Problem had captured the harder \lq\lq exceptional cases\rq\rq\ of that classification. M\"uller says \cite{Fe73} was what he first saw of the details of Davenport pairs, and he corrected an error in that. Thm.~\ref{DS4} and \ref{secV.4}, especially \S\ref{FeitInteractions}, give traces in the literature of how Feit handled his interactions with me, with the comment in \cite{Fr73a} relevant here. 

The more optimistic conjecture I made for polynomials turned out true even for indecomposable rational functions. That is, it was possible to consider the precise permutation representations that arose in series of groups related to alternating and dihedral groups. This addition to Guralnick-Thompson was Guralnick's work (and formulation) with many co-authors and independent papers by others.

Guralnick visited Florida while I was there, and he and Thompson generated series of \lq\lq genus 0 groups.\rq\rq\ They based this on running through the classification of primitive groups using \cite{AOS85} (\S \ref{secVII.3} and \S \ref{App.2}). \cite{AOS85} constructs a template of five patterns of primitive groups. Into four of those you insert almost simple groups. Affine groups comprise the fifth (\S\ref{App.1a}). 

Leaving aside affine groups -- on some problems they cause grave difficulties --  this then naturally divided the task into running through the simple groups inserted into these templates. This was a special expertise of Guralnick (see \S \ref{secVII.3}). So, the Genus 0 Problem ran through two filters: \cite{AOS85}; and the distinct series of finite simple groups, together with affine. This {\sl lexigraphic\/}  procedure accounts for the number and length of contributions to the genus 0  resolution (for covers over $\bC$). 

 \cite{CKS76} sufficed for the group theory in Davenport's problem and the solution of the 1st Hilbert Siegel Problem \ref{DeFrMC}.  \cite{GT90} is the first paper proving that there are infinitely many simple groups that were not composition factors of genus zero groups. \cite{GFM99} classified all genus 0 rank 1 Lie group actions, and it gave all the branch cycles for the exceptional genus 0 groups in this case. 
 
I could look at early Guralnick-Thompson results on exceptional genus 0 groups from this list, and just from the BCL (\S\ref{secVI.2})  see that a small number provided rational functions outside Serre's OIT that gave {\sl Schur\/} ({\sl exceptional\/} as referred to in \S\ref{secI.0}) covers over $\bQ$: one-one maps on $\bZ/p \cup \infty$, for infinitely many $p$. We didn't know such existed previously. (We apologize for the two uses of exceptional -- covers, versus groups -- but it is historial.) It was unlikely that the whole genus 0 problem would have been solved without having been so precise. 

\cite[Exp. 6.3]{Fr99} has Guralnick's conjecture for what would be the exceptional genus 0 monodromy (over $\bar \bQ$) and now it is a theorem. In these lists you see several related to $A_n$. For this discussion, especially,  notice the permutation representation of the cover acts on distinct, unordered pairs of integers. 

Yet, in the Hilbert-Siegel problems,  a Siegel Thm.~constraint over $\infty$ leaves  but  finitely many: Just the degree 10 rep.~in \eqref{deg10}. \cite{GMS03} shows  the Schur problems about exceptional covers motivating  the whole topic (as in \S\ref{limitDih} and \S\ref{coGS}). 

By distinguishing covers with genus {\sl slightly\/} larger than 0, distinctions between genuses 0, 1 and higher came clear. The final formulation includes a genus $\geng$ version, with the cases with $\geng > 1$ differing only in the list of finitely many exceptional pairs:  (groups, primitive permutation representations). 

Yet, the  precision for the exceptional groups we saw for polynomials wasn't possible on all the exceptional \lq\lq genus $\geng$ groups\rq\rq\ (not even $\geng=0$). Especially, when it came to eliminating most of the  \lq\lq exceptional simple Lie-type groups.\rq\rq\ I searched for a way to document that, and found likely its relation to the story of finding reasonable presentations for $G_2$ (over $\bC$) in \cite[pp.~924--25]{Ag08}. Problems related to Davenport's Problem, that arose early in these developments, remain the unequaled archetype for being precise. 

Guralnick also led the study classifying  genus 0 groups, and their representations, that could occur -- his name --  \lq\lq generically.\rq\rq\  An algebraic geometer would mean the generic curve of a given genus has a cover of $\prP^1_z$. Guralnick's meaning, however, is that the curves realizing such covers occur densely in the moduli  of genus $\geng$ curves. Here the classification of the exceptional groups is  precise. 

\cite{GM98} includes showing, for $\geng>0$,  unless a cover alternating or symmetric or symmetric monodromy with a limited set of permutation representations, it cannot occur densely.  \cite{GSh07} settles the generic curve problem in characteristic 0. 

\subsection{Writing equations}  \label{writeEquats}  \S\ref{depSch} explains attempts to produce coordinates for Davenport pairs.  Generalizing Ritt's Thm. -- \cite{Ri22}, on the ways in which a rational function can have multiple decompositions -- is related to Davenport's and Schinzel's problems. \S\ref{genRitt} reminds how that generalization brought more attention to using \lq\lq explicit\rq\rq\  equations than any other topic. 

\subsubsection{Branch cycles versus equations preliminary} \label{bcPrelim} 
Here is an example contrasting using branch cycles on Schinzel's problem with the explicit equation approach. Assume $g(y)=a f(y)+b$ for some $a,b\in \bC$.  Lem.~\ref{ChebChar} uses branch cycles  to show that $f(x)- g(y)$ factors into degree 1 or 2 factors over $\bC$ if $f$ is affine equivalent to a (degree $n$) Chebychev polynomial, and $ax+b$ permutes its finite branch points. If $ax+b$ doesn't permute the branch points, then \eqref{schcond} says $f(x)- g(y)$ is irreducible. 

\begin{lem} \label{ChebChar} Use the assumptions above. With  $n$ odd, $f(x)- g(y)$ has one degree 1 factor; all others of degree 2. With $n$ even, the result is the same if $L: x\mapsto ax+b$ fixes each branch point of $f$; all factors have degree 2 if $L$ nontrivially permutes the branch points. With $f\in \bQ[x]$ and $a,b\in \bQ$, for all $n$, each degree 2 factor has definition field generated by the symmetric functions in $\{e^{2\pi ij/n}, e^{-2\pi ij/n}\}$ (or, functions in $\cos(2\pi j/n)$) for some integer $j$. For a given value of $\gcd(j,n)$, the collection of factors corresponding to $j$ with that value are conjugate over $\bQ$. \end{lem} 

\begin{proof} First take $n$ odd. \cite[p.~47]{Fr70} has this  Chebychev characterization: $f$ has two finite branch points and a branch cycle description $(\sigma_1,\sigma_2,\sigma_\infty)$ with $\sigma_i$, $i=1,2$, in the unique involution class $\C_2$ in the dihedral group $D_n$. The condition on $ax+b$ says the cover for $g(y)$ has the same branch cycle description at the same branch points.  So, $f$ and $g$ give equivalent covers of $\prP^1_z$. Irreducible factors of $f(x)- g(y)$ correspond to orbits of $D_n(1)=\bZ/2$, which correspond to orbits of multiplication by -1  on  $\{0,1,2,\dots,n\nm 1\} \mod n$: 1 length 1 orbit, the rest length  2. 

For $n$ even, there are two classes of involutions in $D_n$: $\C_2$ (resp.~$\C_2^*$) with shape the product of $\frac{n}2$ (resp.~$\frac{n\nm 2}2$) disjoint  2-cycles. If $L$ leaves $f\,$'s  branch points fixed, then, again, the covers are equivalent, and the result is the same. If $L$ permutes the branch points, then the covers can't be equivalent (they have different branch cycles), but their Galois closures have the same branch cycles. 

The permutation representations for the covers of $f$ and $g$  correspond to the respective cosets of the two conjugacy classes of copies of $\bZ/2$ in $\bZ/n\xs \{\pm 1\}$. One is generated by $\alpha_1=(0,-1)$; the other by $\alpha_2=(1,-1)$. As above, irreducible factors of $f(x)- g(y)$ correspond to orbits of $\alpha_1$ on cosets of the group $\lrang{\alpha_2}$. 

Apply the BCL (\S\ref{BCLtreat}) to any $\bQ$ cover with the branch cycles above. The only non-trivial power of $\sigma_\infty$ conjugate to $\sigma_\infty$ in $D_n$ is  $\sigma_\infty^{-1}$. So, the cover given by $f$ must have Galois closure $\bZ/n\xs (\bZ/n)^*$. Thus,  $|{}^a G_f/G_f|$ (as in  \S\ref{secII.3}) has degree $|(\bZ/n)^*|/2$, and  $\bQ(\zeta_n)$ contains the definition field of the Galois closure.  That characterizes constants as the subextension of $\bQ(\zeta_n)$ of index 2. Those constants  come from the coefficients of the factorizations above.  We are done. 
\end{proof} 

For the two cases in Lem.~\ref{ChebChar} where $a=\pm 1$, $b=0$, \cite[Prop.~2.2]{AZ03} lists \cite{DLSc61} and  \cite{Tv68} as explicitly writing equations for these formulas. The two step solvable group $D_n$ has easy explicit equations. It is the first grad course regular realization of centerless groups as Galois groups. Yet, even for dihedral groups, there are Nielsen classes that arise in applications where explicit equations are -- understatement -- a deeper story, as in Ex.~\ref{modcurves}.  

\S\ref{depSch}  considers the story of writing equations for branch cycles  for a Davenport case,  where $G$ is almost simple, but not an alternating group. 

\begin{exmp}[Modular curves] \label{modcurves} The group theory of another Nielsen class is almost identical to Lem.~\ref{ChebChar}. Again, $G=D_n$, and $\C_2$ is the unique involution class when $n$ is odd ($n$ even is similar). The Nielsen class $\ni(G,\bfC_{2^4})$ -- repeating $\C_2$ four times -- contains branch cycles for genus zero covers. For some  $f:\prP^1_x\to \prP^1_z$ representing one of these covers, normalize (as always)  the 2-fold fiber product $\prP^1_x\times_{\prP^1_z} \prP^1_x$.  There is a (degree 1 over $\prP^1_x$) diagonal component. The other $\frac{n\nm 1}2$ components over $\bar \bQ$ have degree 2. For odd $n> 1$, each elliptic curve appears as a component. The reduced Hurwitz space is the modular curve $X_0(n)$ minus its cusps. That observation, \cite[\S2]{Fr78}, seeded \cite[\S5.1-5.2]{DeFr94} and \cite{FrV92} that developed into the {\sl Modular Tower\/} generalization of modular curves \cite{Fr95b}. \end{exmp} 

\subsubsection{Dependence on Schinzel's problem} \label{depSch} \cite[Def.~3]{CoCa99} applies  polynomial normalizations of \S\ref{polySpaces} to a pair $(f,g)$. This isn't, however, a Davenport pair. We might -- considering the relation from \eql{DavRes}{DavResc}   in Thm.~\ref{DS2} -- subtly call it a {\sl Schinzel pair}. The authors, though number theorists, work over $\bC$. You can do inner affine adjustments of $f$ and $g$ separately. To, however, retain the Davenport property, you must  apply outer composition of $z \mapsto az +b$ simultaneously to both. 

With subscripts indicating the homogenous term degrees: 
\begin{triv} $f(x)-g(y)$ factors as $A(x,y)B(x,y)$, with $$A=A_k(x,y)+A_{k\nm1}(x,y)+\dots\text{ and }B=B_{n\nm k}+B_{n\nm k\nm 1}+\dots.$$  \end{triv}  

For either Davenport's or Schinzel's problem, you could assure the conditions they (or \S\ref{polySpaces}) list for one polynomial, say $f$, without loss. Example: To know about equality of values of $f$ and $g$ over residue class fields, by choosing $a$ any nonzero constant, you can assure $f$ is monic. Over $\bar K$ (but not necessarily over $K$), you can make an affine change to $y$, to assure $g$ is also monic \cite[\S3]{CoCa99}. 

So, to assume Davenport pairs are simultaneously monic, requires consequence Prop.~\ref{DS5}, \eql{davpairs}{davpairs3} relating the normalized polynomials as conjugate over a large locus of the parameter space.  Do that, however, and the assumption of  \cite[Def.~3]{CoCa99} that $f$ has 0 constant term would, incorrectly, also have $g$ with 0 constant term. 

Instead, we need -- as in Prop.~\ref{DS6} or Thm.~\ref{vectorBundle} -- to consider the constant term $c_n$ of a generic $f$ as a function in $t_n$ with coefficients in $\bQ_n$. Denote by $\bar c_n$ its complex conjugate, so the constant term in  $f(x)-g(y)$ is $c_n-\bar c_n$.  Simultaneously adding the same  $b\in \bQ(t_n)$ to $f$ and $g$ leaves $(f,g)$ a Davenport pair.  

I now summarize \cite[\S3]{CoCa99} to highlight how they pop up a parameter identifiable with $t_n$ for the degrees $n=7, 13$ and $15$ in Thm.~\ref{DS4} and Prop.~\ref{DS6}. I assume they took  inspiration from Birch's degree 7 example \cite[p.~593]{Fr80}. 

Their calculations start from the existence of a difference set $\sD_n=\{1, \alpha_2,\dots, \alpha_k\} \mod n$ from Prop.~\ref{DS3}. Especially that the highest homogenous terms for the factors for the values of $n$ listed can be taken with no loss as
$$A_k=(x-\zeta_n)\prod_{i=2}^k (x-\zeta_n^{\alpha_i}y)\text{ and } B_{n-k}=\prod_{j\in \{0,1,\dots,n\nm 1\} \setminus \sD_n}(x-\zeta_n^jy).$$
From this point we work in the principal ideal domain $\bC(y)[x]$: the ring in $x$ over the field $\bC(y)$. Example: Since the $A_k\,$s have no common factors in $x$, 
\begin{triv} \label{euclid1}  there are $A',B'\in \bC(y)[x]$ so that $A_kB'+B_{n\nm k}A'=1$. \end{triv} 
\noindent Write   $f(x)=x^n+c_2x^{n\nm2} + \dots + c_{n\nm1} x + c_n$ and $g(y)=x^n+d_2y^{n\nm2} + \dots + d_{n\nm1} x + d_n$. \begin{equation}   \label{euclid2} \text{Then, }c_{\ell}x^{n-\ell} - d_{\ell}y^{n-\ell}=\sum_{0\le u\le \ell}A_{k-u} B_{n\nm k \nm \ell\np u}. \end{equation} 
 
Plug $\ell=1$ into \eqref{euclid2}. From \eqref{euclid1}, $A_{k\nm1}\equiv B_{n\nm k\nm1} \equiv 0$.  For $\ell=2$, multiply  \eqref{euclid2} by \eqref{euclid1} to deduce 
$$\begin{array}{rl} A_{k\nm 2} \equiv (c_2x^{n\nm 2} - d_2y^{n\nm2})A'  &\mod A_k \text{ and } \\ 
B_{n\nm k\nm 2} \equiv (c_2x^{n\nm 2} - d_2y^{n\nm2})B' &\mod B_{n\nm k}.\end{array}$$ 
Put $y=1$, in the 1st of these. The coefficient of $x^{k\nm 1}$ on the left is 0, so it is on the right,  giving  $d_2$ as a function of $c_2$. The same happens for the 2nd of these, 
\begin{triv} \label{twoResults} giving a second expression for $d_2$ in $c_2$, that must be the same.\end{triv}  
Proceed inductively in $\ell$, remembering this is on examples for $n=7, 13, 15$. Using PARI  you find you can express all the $c_i\,$s, $i\ge 3$,  and all the $d_j\,$s, $j\ge 2$ as functions of $c_2$. This empirical induction isn't in detail; more illustrated -- as we have done -- by the case $n=7$. Yet,  even there it is unclear where they use $c_n=0$; once they have enough coefficients to determine $A$, they quit. 

The upshot:  at the end of  \cite[\S3]{CoCa99}, $c_2$ is a replacement for $t_n$. Yet,  certainly not the canonical kind of replacement called for in Prob.~\ref{CouvPoss}.  Indeed, without explanation,  \cite[\S5]{CoCa99} has dropped several of the original normalizations (even including that $f$ and $g$ are monic?). It seems they found it is better to take $f$ as conjugate to $g$ because of the natural symmetry. This is done by taking the replacement for $t_n$ a constant time $g_2$, and dropping $c_n=0$.  Finally, they illustrate Prop.~\ref{DS5}, with particular choices produced by machine as above and dependent on the theory from our previous sections that went into it.  

\begin{prob} \label{CouvPoss} Could some refined version of the procedure of \cite{CoCa99}  eliminate using the simple group classification in Davenport's/Schinzel's problem, just as the BCL avoided it in the restriction to the version over $\bQ$? 

Since $t_n$ is an automorphic function on the upper half plane, can we find  a $q$-expansion with coefficients based on the representation theory of the groups $\PGL_n$? \end{prob} 

For the 1st statement  in Prob.~\ref{CouvPoss}, my opinion is that this is unlikely. For the second, my reaction is to ask: How could it not be so?  

\subsubsection{Ritt I} \label{genRitt}  Denote the greatest common divisor (resp.~least common multiple) of $(m,n$) by $\gcd(m,n)$ (resp.~$\lcm(m,n)$). Suppose $f(x)\in \bC[x]$ has a maximal decomposition in the form \begin{equation} \label{maxchain} f_v\circ f_{v\nm1}\circ \cdots \circ f_1.\end{equation} Ritt described all maximal decompositions of $f$ by starting from $f$ using (decomposition) substitutions for some $1\le i\le v\nm 1$,  $f_{i\np1} \mapsto  f_{i\np1}^*$ and $f_{i} \mapsto  f_{i}^*$,  whenever $$f_{i+1}\circ f_i=f_{i\np1}^*\circ f_i^*.$$  Ritt's 1st Thm.~says all maximal decompositions of $f$ come from chains of substitution in these two cases: 
\begin{edesc} \label{rittsub} \item \label{rittsuba} M\"obius insert: For some $\mu\in \PGL_2(\bC)$: $f_{i\np1}^*=f_{i\np1}\circ \mu$, $f_i^*= \mu^{-1}\circ f_i$. 
\item \label{rittsubb}  Ritt substitution: $(\deg(f_{i+1}),\deg( f_i))=1$ and $\deg(f_{i\np1})=\deg(f_i^*)$. 
\end{edesc} 

\cite{FrM69} generalizes Ritt's 1st Thm. to any field extension with a totally ramified discrete valuation whose ramification index is prime to the characteristic. This situation includes the case \cite[\S2.3]{Pa09} calls a generalized polynomial cover. 

Ritt's Thm.~2 in \cite{Ri22}, describing exactly when you can have \eql{rittsub}{rittsubb} is harder. These Ritt substitutions suffice in \eql{rittsub}{rittsubb} with $n,m$ distinct primes. 
\begin{edesc} \item Chebychev -- \eqref{Chebychev}:  $f_{i+1}=T_n\mapsto T_m$ and $f_i=T_m\mapsto T_n$ with $(n,m)=1$ 
\item Cyclic: $f_{i\np1}=x^n \mapsto x^mh^n(x)$ and $f_i=x^mh(x^n)\mapsto x^n$, $h$ nonconstant  \end{edesc} 
\cite[p.~2]{Pa09} has a typo -- equivalent to  $f_i\mapsto x^m$ -- where I have the cyclic case. 
We turn to how  
\cite[Cor.~p.~47]{Fr73b} classifies  variables separated equations $f(x)-g(y)=0$ over $\bQ$  that have infinitely many quasi-integral points,  so generalizing Ritt's Thm.~2. As in  \eqref{*}: $\deg(f)=m,\deg(g)=n$. Siegel's Thm.~(\S\ref{limitAlt}) gave branch cycle conditions, exactly as in Prop.~\ref{1stHS} on the  factors of $f(x)-g(y)$ as $\prP^1_z$ covers; starting with each defining a genus 0 curve. 

Suppose \eqref{*} is irreducible. Then,  apply the so-called {\sl Abhyankar's lemma}.  It was used often by, say, Hilbert, Hurwitz, Minkowski, Siegel, \dots, but a super-use, and its naming, came from Grothendieck's application \cite{Gr59} (see \eql{grothEx}{grothExa}). The form of the lemma in our case says: If, over a branch point $z_i$ of  $f$ (resp.~$g$), $x_{j,i}$ (resp.~$y_{k,i}$) ramifies to order $m_{j,i}$  (resp.~$n_{k,i}$), then corresponding to this pair on (the projective, normalization of) \eqref{*}, 

\begin{equation}  \label{abhyLem} m_{j,i}\cdot n_{k,i}/\gcd(m_{j,i},n_{k,i}) \text{ points ramify of order }\lcm(m_{j,i},n_{k,i})\text{ over }z_i.\end{equation}  
In contrast to applying Riemann-Hurwitz to $j$-line covers (\S\ref{j-lineRH}), we now easily compute the genus of \eqref{*} from \eqref{abhyLem} and this data:   
\begin{triv}  \label{bcconds} the cycle-type of the branch cycles for $f$ and $g$, especially noting those attached to a common branch point for $f$ and $g$. \end{triv}

\noindent \cite[Thm.~3]{Fr73b} then produces the equations \eqref{*} satisfying these conditions: 
\begin{edesc}  \label{ritt} \item  \label{rittb} $\gcd(m,n)=1$ or 2; (nonsingular completion of) \eqref{*} has genus 0;   
\item  \label{ritta} and $f(x)-g(y)$ is irreducible. 
\end{edesc}  It is immediate from \eqref{schcond} that if $\gcd(m,n)=1$, then \eql{ritt}{ritta} holds. \cite{Tv64} indicates the history of that case.  If \eql{ritt}{ritta} doesn't hold with $\gcd(m,n)=2$, then both $f$ and $g$ are composite up to inner equivalence with the same degree 2 polynomial. 

\cite[Cor.~p.~47]{Fr73b}  needed to separate the possibility \eqref{*} is reducible from the basic genus calculation.  That used  \cite[Prop.~2]{Fr73a} as stated in \eqref{schcond}.   

The case \eql{ritt}{rittb} is  Ritt's Theorem in disguise. About that, after the proof of Ritt's Thm. in \cite[pp.~15--39]{Sc82} says: \lq\lq More general but less precise results are found in \cite{Fr73b}.\rq\rq\ For Ritt's Theorem the only difference is that I've left out explicit equations for affine equivalence. 

 \cite[p.~50]{Fr73b} reproduced \cite{DLSc61} and \cite{Le64} as special cases showing, at times, that checking  \eqref{bcconds} is easy. In the former case: 
$$f(x)=f_n(x)=\frac{x^{n\np1}-1}{x-1}-1 \text{ and  }g_(y)=f_m(y), n\ne m.$$ 
A long history of diophantine equations motivates  this additive expression. 

Here are two more modern sets of polynomials which took on the same issues: when does  \eqref{*}  have infinitely many integral, {\sl or rational\/}, solutions. 
\begin{edesc} \label{checkSeries} \item \label{checkSeriesa} With  $f_{m,d}=\prod_{i=1}^{m\nm1} (x+id)$, $1<m, d\in \bQ$ positive, $$f=f_{m_1,d_1}, g= f_{m_2,d_2}, (\text{ if }  m_1=m_2, \text{ then } d_1\ne d_2).$$
\item \label{checkSeriesb} With $f\in \bQ[x]$ and $g(y)=cf(y)$, $c\ne 0,1$ (as in Prop.~\ref{Gusic}). \end{edesc} 

We will contrast the approaches of  \cite{BeShTi99} and \cite{AZ03}  in the problems proposed respectively by \eql{checkSeries}{checkSeriesa}  and \eql{checkSeries}{checkSeriesb}. In both, the main job was finding genus 0 (or 1) curves defined by factors of variables separated expressions. The addition to \cite{Fr73b}: This used the quasi-integral solutions condition limiting some possible branch cycles. In practice, they used the same rigamarole up to a concluding identification problem that brought up new issues. 

The definitions \eql{checkSeries}{checkSeriesa}  occur in \cite[Thm.~1.1]{BeShTi99} which chose to  consider results on   equality of multiplicative expressions.  \cite[Thm.~2.2]{BeShTi99}  is  similar in  replacing $g= f_{m_2,d_2}$ by a constant times this $g$, but set $d_1=d_2=1$. 

Indeed, $f_{m,d}$ is  just a scaling of the variable for $f_{m,1}$. If $m$ is odd, from Descartes' rule of signs the finite ramified points -- zeros  of $\frac{df}{dx}$ -- fall neatly between the zeros of $f_{m,1}$.  Plug  them in to see that these local maxima of $f$ evaluated at $f$ decrease in value. So, the corresponding finite branch points are distinct, and the finite branch cycles $\row \sigma {m\nm1}$ are all 2-cycles. According to \eql{cycleconds}{cycleconds1} -- generation -- the monodromy group of the cover is $S_m$, the only group generated by 2-cycles. 

For $m$ even, the involution $x \mapsto -x+(m\nm 1)$ maps the zeros into themselves. This symmetry means $f$ is a composite of some $f_1$ with $(x-\frac{m\nm1}2)^2$. Written explicitly, almost the same argument as above shows the finite branch cycles of $f_1$ are also 2-cycles. So, its monodromy group is $S_{m/2}$ and the monodromy group of $f$ is the wreath product (see Rem.~\ref{wreath}) of $S_{m/2}$ and $\bZ/2$.  In this simple case, irreducibility is  easy to check: $S_m$ is doubly transitive and has one degree $m$ permutation representation. \cite{BeShTi99} didn't indicate these monodromy groups. 

The genus calculation shows its swift growth based on those 2-cycles, with a small set of low degree $(m,n)$ pairs where the genus might be 0 or 1, depending on possible overlapping branch points. The proof of \cite[Thm.~2.2]{BeShTi99} for example, includes specific checks for that. Also, the proof of  \cite[Thm.~1.1]{BeShTi99} runs into genus 1 curves in classical forms. To finish the arithmetic result uses  \cite{Ma77} (Mazur's explicit Thm.~on elliptic curve torsion points over $\bQ$), to conclude that one of the obvious rational points on the equation is not torsion. \S\ref{modgenRitt} gives a general context for the problems considered by \cite{BeShTi99}.  

In \cite[Thm.~2]{AZ03} the authors follow the actual description of genus 0 cases in \cite[p.~42 Cor.]{Fr73b}. They recognize the cyclic and Chebychev cases 1st,basically using the quote following Lem.~\ref{ChebChar}. Then, they show how to reduce to where $f$ is indecomposable. \cite[Lem.~3.1]{AZ03} reproves the special case of \eqref{schcond} where $f$ is indecomposable, to consider possibilities that $f(x)-g(y)$ is reducible, and therefore the Galois closure covers of $f$ and $g$ are the same. 

\cite[p.~274--276]{AZ03} proves a version of Prop.~\ref{Gusic} -- we use its notation -- to show there are no reducible cases beyond where $f$ is Chebychev or cyclic. The gist of its application, is that the equating of the Galois closures of the covers for $f$ and $g$ comes with an automorphism $c_{\text{\rm AZ}}$ of $G_f$, not in $N_{S_n}(G_f)$ (\S\ref{BCLtreat}),  leaving the conjugacy class of $\sigma_\infty$ invariant. 

They could have completed the impossibility of a reducible case using Prop.~\ref{DS5}, under their $f$ indecomposable assumption. They didn't,  so I explain how it works here. The $c_{\text{\rm AZ}}$, up to conjugation by $G_f$, takes $\sigma_\infty$ to $\sigma_\infty^{-1}$,  the -1 being a non-multiplier of the design attached to the pair of doubly transitive representations in  \eql{davpairs}{davpairs1}. That is, it would change the class of $\sigma_\infty$ to a new conjugacy class. More directly, there is no such automorphism as $c_{\text{\rm AZ}}$ extending $G_{\mu \circ \hat f}$ in Prop.~\ref{Gusic} in this case. The linear transformation of \eqref{*7} relating zeros of $f(x)-z$ to those of $g(y)-z$ won't extend to an automorphism of the Galois closure function field. 

Above \cite[p.~270-274]{AZ03} they say they want to avoid the classification. Yet,  \cite{Fr73a} doesn't use the classification -- there was none, then. They use what we  reviewed prior to Prop.~\ref{DS5}. (The  classification use in \eql{davpairs}{davpairs2} and \eql{davpairs}{davpairs3}  gives the precise Davenport pairs $(f,g)$ with $f$ indecomposable occuring over some number field.) We see  $c_{\text{\rm AZ}}$ again in \S\ref{SchinzG} to consider  \eql{checkSeries}{checkSeriesb} when $f$ is decomposable. 

In the irreducible case of \eql{checkSeries}{checkSeriesb}, \cite[Prop.~2.6]{AZ03}   quotes \cite[Prop.~1]{Fr73b} on the formula for the variable separated -- fiber product -- curve genus  from Abhyankar's lemma. As usual the demanding cases have $f$ and $g$ with overlapping finite branch points. Especially interesting is a list of explicit polynomials  $P_1,\dots, P_6$ \cite[Def.~2.1]{AZ03} where the last 3 are particular  $f\,$s in  \eql{checkSeries}{checkSeriesb} \cite[Thm.~2]{AZ03}. 

\cite{AZ01} consider $f(x)-g(y)=0$ \eqref{*}, when $(\deg(f),\deg(g)=1$, where we have already remarked (after \eqref{ritt}) irreducibility is automatic.  When the genus is now 1, they give many interesting examples, some not over $\bQ$ and involving the Mazur-Merel result (\cite{Ma77} and  \cite{Me96}). I mention it here, to note that we haven't considered what would limit any curve from being a component of a variables separated equation. For example,  Ex.~\ref{modcurves} says every genus 1 curve occurs in many different ways as a component of a variables separated equation. 

\begin{rem}[Wreath product exercise] \label{wreath} \cite[\S2]{Fr70} introduces wreath products to write branch cycles for the composite, $f_1\circ f_2$, of rational functions from branch cycles for $f_1$ and $f_2$.  Assume  $h^*=h((x-b)^2)$ where the finite branch cycles -- relative to some classical generators, (\S\ref{classgens}) -- of the degree $n$ polynomial $h$ are 2-cycles, and $h(0)$ is not a branch point of $h$.  
Then, we can choose $\psigma=((1\,2),(1\,3),\dots, (1\,n),(1\,2\,\,\dots\,n)^{-1})$ as branch cycles for $h$.  Now use $$\{\{1',1''\},\dots,\{n',n''\}\}$$ for the letters on which branch cycles for $h^*$ act. Branch cycles for $h^*$ will give branch cycles for $h$, in a natural way, by mapping both $i'$ and $i''$ to $i$, $i=1,\dots,n$.  Here are branch cycles for $h^*$:
$$\begin{array}{rl} \psigma^*=&((1'\,2')(1''\,2''),(1'\,3')(1''\,3''),\dots,(1'\,n')(1''\,n''), \\
&\ (1'\,1''),(1'\,2'\,\dots\,n'\,1''\,2''\,\dots\,n'')^{-1}).\end{array} $$ 

This special case of \cite[Lem.~15]{Fr70} shows why I know  the monodromy group of  $h^*$ is $S_{n}$ semidirect product with $(\bZ/2)^{n}$,  the wreath product named above. \end{rem} 

\subsubsection{Wreath products and Ritt II} \label{modgenRitt}   In Rem.~\ref{wreath}  the monodromy, $H$, of (the cover from) a composite $f_1\circ f_2$  of rational functions is the entire wreath product. Let $H_i$ be the monodromy of $f_i$, $i=1,2$. If  the conditions of \cite[Lem.~15]{Fr70} don't hold, then $H$ may be a {\sl proper\/} subgroup of $H_2^{\deg(f_1)}\xs H_1$ satisfying these conditions: 
\begin{triv} \label{wreathconds} $H$ maps surjectively onto  $H_1$, and its intersection with $H_2^{\deg(f_1)}$  maps surjectively onto each fiber. \end{triv} 

These wreath product  ideas, especially using branch cycles, apply for any composite covers of $\prP^1_z$. For example, in a composite of covers $X_2\to X_1\to \prP^1_j$ where $X_2\to X_1$ has degree 2 (but neither necessarily of genus 0), then 
the intersection of $H$ in \eqref{wreathconds}  might only be the subgroup of $H_2^{\deg(f_1)}$ whose entries sum to $0 \mod 2$. 

That's the case in  the Main result of \cite{BiFr86} generalizing the cyclic covers of genus $\geng$ curves result of  \cite{DelMu67}. It is a connectedness of moduli  result, like that giving the computations of Thm.~\ref{vectorBundle},  from transitivity of the braid group on certain Nielsen classes. 

Both the monodromy group above, and of Rem.~\ref{wreath} are {\sl Weyl groups}. Vasil Kanev was  inspired to extend, say, \cite{BiFr86} to consider all Weyl groups: subgroups of wreath products of $S_n$ (in its standard representation) and $\bZ/2$ satisfying \eqref{wreathconds}.  Not just  to classify, but rather, to provide a limited context for useful connected Hurwitz space results. Many corollaries follow from deciphering orbits of braid groups on Nielsen classes, such as  \cite{Ka89}, \cite{FVe08} and \cite{FVe09}. 

These results  model generalizing  \cite{BeShTi99} sufficiently, so their formulation is akin to the Hilbert-Siegel problems of \S\ref{limitAlt}. That is, using similar Nielsen classes we ask if conclusions might depend only on natural related data. This would extend the problems of \cite[\S4]{DeFr99} and also put Mazur's Thm.~ in a new context.  

Wreath products  are a tool for describing monodromy groups (over $\bC$) of composites of rational functions. The situation of \eqref{wreathconds} requires deciding from two primitive genus 0 groups, what subgroups of the full wreath product could possibly occur. We easily concoct the full product from  \cite[\S2]{Fr70}. Yet, divining subgroups of the full product that occur  takes us beyond the genus 0 problem (\S\ref{Thompsonresp}). 

\S\ref{DavM} reminds of \cite{Mu98}  on extending Davenport to polynomial composites and \S\ref{SchinzG} notes  \cite{Fr87} on the $(m,n)$-problem (related to Schinzel). Both require subgroups of genus 0 wreath products. This subsection concludes by distinguishing using equations from using branch cycles to calculate composition factors. 

Given $f\in \bC(x)$, or branch cycles, $\psigma_f=(\row \sigma r)$, for $f$, how efficient is it to  find degrees of the indecomposable constituents of $f$?
All rational functions in a given Nielsen class, $\ni(G,\bfC)^\abs$ (with a representation $T_n: G \to S_n$) have the same composition factor degrees (dividing $n$).  Any subset $$W=\{i_1,\dots, i_d\mid 1< d < n, d|n\}$$ of distinct elements from $\{1,\dots, n\}$ has a $G$ orbit. Denote the collection of such $G$-orbits by $\sI_{G,\bfC}$.
I will now assume that computing the action of any given $\sigma\in G$ on such a $W$ requires just one immediate operation. When $f\in \bC[x]$, with no loss, we can assume branch cycles with, as in \S\ref{secIII.3}, $\sigma_r=\sigma_\infty=(1\,2\,\dots\,n)$. 

\begin{lem} \label{PolyComp} An $I\in \sI_{G,\bfC}$ represents a composition factor, up to affine equivalence, if and only if for any two subsets $W,W'\in I$, either $W=W'$ or $W\cap W'=\emptyset$. 

Suppose $f\in \bC[x]$, with $\sigma_\infty$ as above. Then, $I\in \sI_{G,\bfC}$ represents a composition factor, if and only if for some $d|n$, $I$ contains $W_d=\{0,d,\dots,n/d\}$. To compute the composition factors from $\psigma_f$ in this case, requires only checking for each $1< d|n\le \sqrt n$ if each $\row \sigma {r\nm1}$ permutes the collection $W_d^{\sigma_\infty^j}\eqdef W_d+j$, $j=0,\dots,\frac n d\nm1$. Listing decomposition factors therefore requires no more than $\sum_{d} (r\nm1)\cdot \frac n d$ operations, clearly bounded by a polynomial in $n$. 
\end{lem} 

\begin{proof} The first sentence characterizes a permutation representation through which $T_f$ factors, corresponding to a cover through which $f:\prP^1_x\to \prP^1_z$ factors. Now consider the special case where $\sigma_r=\sigma_\infty$ as above. 

Suppose $W$ representing  $I\in \sI_{G,\bfC}$ gives a system of imprimitivity of size $n/d$ as above. Translate $W$ (apply a power of $\sigma_\infty$) to assume it contains 0, and that $h$ is the 1st positive integer in it. Then, $W-h$ contains 0 so equals $W$, and has the next largest integer $h$. Continue to conclude that $W$ contains all integer multiples of $h$, and so must be $W_d$. This concludes the lemma. \end{proof} 

\begin{prob}  \label{compFactors}  Use the notation above. Given branch cycles $\psigma_f$ for $f\in \bC(x)$, can you find a polynomial in $\deg(f)=n$ bounding the production of all degrees of composition factors of $f$ akin to the Lem.~\ref{PolyComp} polynomial case? Is there a polynomial time algorithm in $n$, the size of the coefficients of $f$, and the minimal distance between branch points, for computing $\psigma_f$? \end{prob}  

\cite{FrW82} has a programmable algorithm for computing branch cycles, but it doesn't answer Prob.~\ref{compFactors} precisely. A positive answer to Prob.~\ref{compFactors} would give a polynomial time algorithm in deciding the composition factors of a polynomial, or rational function, if Lem.~\ref{PolyComp} has a rational function version. 

An  intuitive theme appears -- sometimes in Schinzel's papers -- that among all rational functions $f \in  \bC(x)$,   whose numerator and denomenator  have altogether no more than $\ell$ nonzero terms, only special $f$ will have nontrivial composition factors. \cite[Main Thm]{ClZ10} has this result. 

\begin{thm} Suppose$f(x) = g(h(x))$ is a composition of two rational functions of degree exceeding 1, but  $h$ is not a composite of some  $\alpha\in \PGL_2(\bC)$  and anything of the shape  $(ax^n + bx^{-n}),  a,b\in \bC$. Then
$\deg(g)\le   2016 á 5^\ell$. \end{thm}  The issue I raise here is that $\ell$ appears in the exponent, not in a polynomial expression. This is common for many rational function type results. I find it unintuitive that decomposability is a complicated subject, but apparently it is.  

\subsubsection{Laurent polynomials and Ritt III} \label{laurent} A Laurent polynomial is a polynomial in $z$ and $1/z$, so it is a rational function with poles, at most, at 0 and $\infty$. With an affine change, we may assume a Laurent polynomial has its (possible) finite pole anywhere you wish. \cite{Pa10a} considers when it is possible that (nonconstant) entire functions -- analytic everywhere on the complex $u$-plane -- uniformize a component $X^0$ of a separated variable equation \eqref{*}. That is, 
\begin{triv} \label{functSepEquat} $f(h^*_f(u))=g(h^*_g(u))$, with $(f,g)$ a polynomial pair, and $(h^*_f,h^*_g)$ entire.\end{triv} 
\noindent You should equivalence $(f,g)$ and $(f(\alpha_f(u)),g(\alpha_g(u)))$ with $\alpha_f$ and $\alpha_g$ affine transformations. \cite{Pa10a} quotes \cite{Pi1887} for the following. I give its proof. I'm curious where in mathematics history it belongs;  Riemann had to know and use it. Denote  the nonsingular projective curve defined by $X^0$ by $X$.  

\begin{lem} Given  \eqref{functSepEquat}, $X$ has genus 0 or 1. \end{lem} 

\begin{proof} Denote the universal covering space of $X$ by $\tilde X$. The entire  $u\mapsto (h^*_f,h^*_g)$ lifts to an entire $u\mapsto h^*_X(u)\in X$ by Riemann's removable singularity theorem \cite[p.~103]{Con78}.  Then, analytic continuation gives an entire function $u\mapsto \tilde h_X\in \tilde X$. Now apply Riemann's mapping theorem \cite[Thm.~9-6]{Sp57}. Unless $X$ has genus 0 or 1, $\tilde X$ is analytically isomorphic to a disk. So, an entire (nonconstant) function has range in a disk:   impossible from Liouville's Thm.~\cite[Thm.~3.4]{Con78}. \end{proof} 

According to Picard's Little theorem \cite[p.~297]{Con78}, an entire function has range missing at most one value in $\bC$. 
An example of where an entire function $h$ would appear is if we have 
\begin{triv} \label{redLaurent} $h^*_f=h_f\circ h$ and $h^*_g=h_g\circ h$ with $(h_f(u),h_g(u))$ either Laurent or ordinary polynomials,  and \eqref{functSepEquat} holds by substitution: $(h^*_f,h^*_g) \mapsto (h_f,h_g)$. \end{triv}  \noindent 

\cite[\S2]{Pa10a}  quotes \cite{BNg06} for the converse: If  \eqref{functSepEquat}, then \eqref{redLaurent} for some entire $h$ and $(h_f,h_g)$. The serious new case is where $(h_f,h_g)$  are Laurent polynomials. The cover $u\to f\circ h_f=z$ has two points over $z=\infty$, so the most telling case for solutions to  \eqref{functSepEquat} reverts to describing the factors of $f(x)-g(y)$ that are genus 0 curves with two points over $z=\infty$. 

We conclude with the \cite{Pa09} generalization of Ritt's Thm., and its use of the explicit result in \cite{BT00}. Note: These papers always work over the complexes. Given a pair of  covers  $f: X\to Z$ and $g: Y\to Z$, their phrase \lq\lq the pair $(f,g)$ is irreducible\rq\rq\  means the fiber product $X\times_Z Y$ is irreducible (compatible with  \cite[Prop.~2.1]{Pa09}). As  in \S\ref{secII.3}, this means the combined Galois closure group $G_{f,g}$ is transitive on the pairs  $(i,j)$, $1\le i\le m,1\le j\le n$, corresponding to the tensor product of $T_f$ and $T_g$. 

\S\ref{fibproduniv} notes the Galois see-saw argument of  \cite[Prop.~2]{Fr73a}, phrased in \eqref{schcond}, is very general. It shows, with no loss, we may replace $f$ and $g$ by covers through which $f$ and $g$ factor, but with the Galois closures of the new $f$ and $g$ the same. Further,  there is a one-one correspondence between the components of the new and the old fiber products.  The use of the fundamental group of $\prP^1_z$ in \cite[Thm.~2.3]{Pa09} is unnecessary and limiting even for covers of $\prP^1_z$. 

The proof of Lem.~\ref{indecompLem} -- \cite[Thm.~2.4]{Pa09}, but using fiber product -- has nothing to do with genus 0 curves. So, in the result you can replace all the $\prP^1\,$s by general normal varieties and finite morphisms. 

\begin{lem}  \label{indecompLem} Assume $(f,g)$ is irreducible, and suppose $\phi_W: W\to Z$ is a cover of nonsingular curves that factors through both $f$ and $g$. If both $W\to X$ and  $W\to Y$ are indecomposable, then $f$ and $g$ are also both indecomposable.  \end{lem} 

\begin{proof} Use  the universal property of fiber product \eqref{UnivPropCurves} (or its generalization \eqref{UnivPropGen}). The irreducibility assumption says $\phi_W$ factors surjectively through $X\times_Z Y$. Since the factorization through $f$ and $g$ are indecomposable, $W$ actually equals $X\times_Z Y$. 

From the construction of the Galois closure (\S\ref{secII.3}), the group of the Galois closure of the projection  $W=X\times_Z Y \to X$ is a subgroup $G_{W/X}$ of the Galois closure group, $G_g$ of $g$, by its action on the same letters. Indecomposability of $W\to X$ is equivalent to this action of $G_{W/X}$ being primitive (\S\ref{secIV}). Therefore, the (possibly) larger group $G_f$ acts primitively on the same letters: $f$ is indecomposable. The same argument gives $g$ indecomposable. \end{proof} 

Suppose we start with two maximal decompositions of $f\in \bC(x)$ (as in \eqref{maxchain}): 
 \begin{equation} \label{maxnonritt} f_v\circ f_{v\nm1}\circ \cdots \circ f_1=g_u\circ g_{v\nm1}\circ \cdots \circ g_1.\end{equation} If you drop the degree conditions in \eql{rittsub}{rittsubb}, the substitution of \eql{rittsub}{rittsuba} is  included in \eql{rittsub}{rittsubb}. We'll refer to that as a {\sl weak Ritt substitution}. Use the symbol $\sim^w$ to indicate one decomposition is obtained from another through weak Ritt substitutions. Let $f=f_v\circ f_{v\nm1}\circ \cdots \circ f_2$ and $g=g_u\circ g_{v\nm1}\circ \cdots \circ g_2$. From $f\circ f_1=g\circ g_1$, Lem.~\ref{indecompLem} implies either $f(x)-g(y)$ is reducible or one of $u$ or $v$ exceeds 1.

The main idea in \cite{Pa09} in generalizing Ritt's Theorem is to consider the collection, $\sR_k$, of rational functions $f$,  for which $f:\prP^1_w \to \prP^1_z$ has at least one place $z_0$ over which it has at most $k$ points. Then, $\sR_k$ is closed with respect to decomposition in that  $f_1\circ f_2\in \sR_k$ implies $f_i\in \sR_k$, $i=1,2$. The latter property is a stand-in for the more general idea of what me might call a {\sl closed Ritt class}. For any element $f$ in any closed Ritt class $\sR$, we can map $f$ to its collection $\sD_f$ of maximal decompositions.  Consider the set  $\sD_\sR =\{\sD_f\mid f\in \sR\}$. 
By replacing an explicit ordered list of composition factors by the composition we get a map back 
$$ \sD_\sR =\{\sD_f\mid f\in \sR\} \to \sR \text{ by } \sD_f\mapsto f.$$ then, modding out by the action of  $\sim^w$ induces a {\sl Ritt map}: 
$R_\sR: \sD_\sR/\sim^w \to \sR$.

For example, Ritt's Theorem is that $R_{\sR_1}$ is one-one.  One conclusion of  \cite[\S3]{Pa09} is that $R_{\sR_2}$ is also one-one. Pakovich notes that this is closely connected to the {\sl Poincar\'e center-focus\/} problem, but that is another topic. 

\subsection{Attaching a zeta function to a diophantine problem} \label{secVII.2} \S\ref{60sProblems} reviews the problems that motivated subsequent developments. Like Davenport/Schinzel problems, their nitty-gritty particulars contrast to the general techniques they motivated in \S\ref{galStrat} and \S\ref{chowMotives}. We see Davenport motivations for considering zeta functions in Prob.~\ref{excCharext}. We simplify notation by assuming diophantine statements are over $\bZ$; adjustment to the ring of integers of a number field is easy. 

\subsubsection{Problems from the '60s} \label{60sProblems}  Let $\afA_d$ denote the space of coefficients of hypersurfaces of degree $d$ in $\prP^d$ (projective $d$-space). For $\by\in \afA_d$ denote the corresponding hypersurface in $\prP^d$ by $h_{d,\by}(\bx)$. We regard it as the fiber of a subspace $\sH_d\subset \afA_d\times \prP^d$ after projection on the first coordinate of $\afA_d\times \prP^d$. 

Recall {\sl Chevalley's Theorem\/} \cite[p.~6]{BoSh66}: A hypersurface over $\bQ$ in $\prP^{d}$ of degree $d$  has a $\bZ/p$ point for every prime $p$. The problem is diophantine, but not {\sl existential}. It has the shape \begin{equation} \label{2quantifiers} D_{\text{Ch}}: \forall \by\in \afA_d, \exists \bx \in \prP^d [ (\by,\bx)\in \sH_d ].\end{equation} You interpret the problem  at each prime as $D_{\text{Ch},p}$ by restricting the coordinates of $(\by,\bx)$ to lie in $\bZ/p$. The conclusion is that $D_{\text{Ch},p}$ is true for all primes $p$: Each degree $d$ hypersurface over $\bZ/p$ has a $\bZ/p$ point. 

Take $\bZ_p$ to be the $p$-adic integers. 
{\sl Artin's Conjecture\/} was similar: For degree $d$, $h(\bx)$ is a hypersurface in $\prP^{d^2}$. Interpret  $D_{\text{Ar},p}$ to mean that each degree $d$ hypersurface over $\bZ_p$ has  a $\bZ_p$ point. 

The Ax-Kochen solution \cite{AxKo66}, however, was a shock:  
$D_{\text{Ar},p}$ is true for all {\sl but finitely many\/} primes $p$. An alternative statement of its conclusion: Artin's Conjecture is true over all nontrivial ultra-products of all $p$-adic completions of $\bQ$. This used a result of Lang for comparison. So, the method applied to few problems, and it left a mystery on the exceptional primes. Yet it made a splash.  

The  Ax-Kochen method produced a new set of fields  by considering the algebraic numbers inside nontrivial ultra-products of all residue class fields of $\bZ$. {\sl Almost\/} (but not) all such  fields would have the P(seudo)A(lgebraically)C(losed) property: All absolutely irreducible $\bQ$ varieties over such a field would have a rational point. Applied to Chevalley's problem they suggested to Ax \cite{Ax68} the following. 

\begin{guess} \label{AxGuess1} Each degree $d$ hypersurface over $\bQ$ in $\prP^d$  should have a rational point in any PAC field $F\le \bar \bQ$. This is equivalent to each such hypersurface containing an absolutely irreducible $\bQ$ subvariety. \end{guess} 

Finally, as a special case of Igusa-like conjectures, for a single prime $p$, and fixed $\by\in \afA_d(\bZ)$ there was the problem of counting the solutions $c_{m,p}$ on $h_{d,\by}(\bx)$  in $\bZ/p^m$. The qualitative question was this. 

\begin{prob} \label{Igusa} Show the Poincar\'e series $\sum_{m=0}^\infty c_{m,p}t^m$ is in $\bQ(t)$. \end{prob} \cite[p.~47, Prob. \#9]{BoSh66} is a special case with $d=2$, of Prob.~\ref{Igusa}, I first heard about it very near the time of Ax-Kochen.  

\subsubsection{Uniform in $p$ quantifier elimination} \label{galStrat} Ax-Kochen, clearly modeled on Tarski's elimination of quantifiers, left a general problem.  Is there such an elimination of quantifiers for problems $P,$generalizing \eqref{2quantifiers}, over finite fields.  \cite{Ax68} posed this. 
(We understood this would give versions by replacing all finite fields by all $p$-adic completions, as noted in \S\ref{chowMotives}.) 

That is, suppose $\row Q m$ are quantifiers (often taken to alternate between $\exists$ and $\forall$) on blocks of variables $\row {\by} m$, with possible unquantified parameters $\bz$. Could you form a series of statements in one less (block of) quantifier(s), that for almost all primes $p$ would be equivalent to the previous statement, until you were down to an unquantified statement. For a statement  $D_{P, \bz, Q_1\by_1,\dots, Q_m\by_m}$ of the type above, denote by $D_{P, \bz, \by_1,\dots, \by_{m\nm1}, Q_m\by_m}$ the statement where you drop the first $m\nm 1$ quantifiers. Here is a statement of the elimination of quantifiers in equation form, where $N_{D_P}$ denotes an explicit finite set of primes dependent on $D_P$. 

\begin{prob} \label{stratproc} Given $D_{P, \bz, Q_1\by_1,\dots, Q_m\by_m}$ can you form $D_{P', \bz, Q_1\by_1,\dots, Q_{m\nm1}\by_{m\nm1}}$  (dependent on $P'$ and $P$) so that for all $p\not\in N_{D_P}$, for each  $(\bz, \by_1,\dots, \by_{m\nm 1}) \mod p$: $$D_{P, \bz, \by_1,\dots, \by_{m\nm1}, Q_m\by_m} \mod p \text{ if and only if }D_{P', \bz, \by_1,\dots, \by_{m\nm1}}\mod p.$$ \end{prob} 

We understand $P$ and $P'$ as above to be algebraic subspaces of the space with appropriate variables. It was seen almost immediately that the conclusion to Prob.~\ref{stratproc} was impossible. Yet, a logic statement asserted that by  G\"odel numbering all possible proofs of all possible statements there would be one in the end that would be either a proof or disproof of the starting finite field problem. 

That may have sufficed for many logicians, for whom particular problems of algebra may not have mattered. So arose surmises there would be no such useful procedure of any sort along the lines of Prob.~\ref{stratproc}. But there was, based on the following principle: With an enhancement, what worked in Davenport's problem -- without  the RET part -- worked in general. 

What allowed elimination of quantifiers was to extend the simple quantified variable statements, and replace them by generalizations of monodromy statements like that of Thm.~\ref{Chebp}. Here are some of the ingredients of the generalization; called a {\sl Galois Stratification}.  Instead of 1-variable $z$, you would have many variables -- in the induction procedure, $\bz, \by_1,\dots, \by_{m\nm1}$; and instead of the trace statement \eqref{*5}, there would be a statement about elements falling in conjugacy classes. 

We couldn't expect with such general problems that there would be an idea like Monodromy Precision (\S\ref{MPres}). For complete generality we must replace one cover of $\prP^1_z$ by a stratification of the space with variables $\bz, \by_1,\dots, \by_{m}$. Attached to each piece of the stratification $A$ there would be an attached Galois cover $\phi:_A: X_A\to A$ of the underlying space, with associated conjugacy classes $\bfC_A$. 

You also need to extend the meaning that the variables would have values in a finite field $\bZ/p$. Suppose $\bz, \by_1,\dots, \by_{m\nm1}$ is within a particular $A'$ of the stratification $\mod p$ for $P'$, and $Q_m$ is $\exists$. Then,  for some  $\by_m$ with values in $\bZ/p$: 

\begin{triv} \label{aChebValue} with $(\bz, \by_1,\dots, \by_{m})$ in a stratification piece $A$ attached to $P$ that projects to $A'$, the  Frobenius attached to that value is in  $\bfC_A$. \end{triv} 
There is a similar statement for $\forall$. Most seriously, no simple trick allowed reverting everything to existential statements, unlike Tarski's situation. Of course, the work comes in producing the stratification, covers and conjugacy classes, with stratification pieces $A'$ that are projections of stratification pieces $A$ of $P$. 

The start and end of the procedure caused some confusion for those with preconceptions. The  start had to also be a Galois Stratification. The trick  -- use trivial (degree 1) covers and the identity conjugacy class -- maybe seemed so trivial as to be inconsequential. When, however, you remove the first block of quantifiers, the replacement Galois Stratification will be as consequential as the difference between Davenport's original problem, and the Thm.~\ref{Chebp} monodromy statement. 

There was one further confounding ingredient. Ax referred to his version \cite{Ax68} of a procedure special case as {\sl one-variable}. That sounds like it included, say, problems like Davenport's. But that was not so.  The Galois Stratification procedure recognized Ax's case as the {\sl zero\/} variable case:  the base was an open subset of $\Spec$ of the ring of integers of a number field. 

The many variable Chebotarev density referred to in the comments after  \eqref{*5} allowed {\sl uniformity with $p$}. At each elimination of a block of quantifiers the procedure carried a possibly increasing exceptional set of primes: $N_{D_P}|N_{D_{P'}}$ in  the equivalence of Prob.~\ref{stratproc}. 

\subsubsection{Introducing zeta functions} \label{zetaCoefficients} \cite[Chap.~25 and 26]{FrJ86}$_1$ and  \cite[Chap.~31 and 32]{FrJ86}$_2$ have  complete details of the most elementary form of the Galois Stratification procedure along with the zeta function production -- our next topic -- based on {\sl Galois Stratification coefficients\/}.  I briefly remind what these things are, along with the value of, and problems with, Chow motives. Then I conclude with problems that tie to Schur's Conjecture and Davenport's Problem.  

\cite{DeLo01} and \cite{Ni10}  also have expositions of Galois Stratification, and they enhance the zeta function coefficients, extending them to  {\sl Chow motive coefficients}.  
A {\sl Zeta function\/}, $Z(t)$, has an attached {\sl Poincar\'e series\/} $\tP(t)$. This is given by the logarithmic derivative: $$t\frac{ d}{dt}\log(Z(t))=\tP(t).$$ Add that $Z(0)=1$, and each determines the other.  The catch: $Z(t)$ rational (as a function of $t$) implies $\tP(t)$ rational, but not always the converse. 

Given diophantine problem $D_{P, \bz, Q_1\by_1,\dots, Q_m\by_m}$ as in Prob.~\ref{stratproc},  consider the cardinality of the set of $\bz^0$ with values in $\bF_{p^k}$ for which when you set $\bz$ to $\bz^0$ the parameter free statement $D_{P, \bz^0, Q_1\by_1,\dots, Q_m\by_m}$ is true over $\bF_{p^k}$. Denote this by $\nu_p(D_{P,\row Q m},k)$. Abusing notation, the most elementary Poincar\'e series  attached to $D_{P, \bz, Q_1\by_1,\dots, Q_m\by_m}$ at the prime $p$ is 
\begin{equation} \label{WeilVectors} \tP_{D_{P,\row Q m}}(t)\eqdef \sum_{k=1}^\infty  \nu_p(D_{P,\row Q m},k)t^k. \end{equation} 

I don't know when Ax introduced such  $\nu_p(D_{P,\row Q m},k)$, but he told me the problem of meaningfully computing them at IAS in Spring '68. 
The Galois stratification procedure concludes with an integer $N_{D_P}^*$ and the following: 
\begin{edesc} \label{strat} \item \label{stratUnif} a quantifier free Galois stratification $P_\bz$ over $\afA_\bz[1/N_{D_P}^*]$, the affine space over $\bZ$ with the $p|N_{D_P}^*$ removed; and  
\item \label{stratInc} for each  $p| N_{D_P}^*$, a stratification of $\afA_{\bz}\mod p$. \end{edesc} We call \eql{strat}{stratUnif} (resp.~\eql{strat}{stratInc}) the {\sl uniform\/} -- in $p$ -- (resp.~{\sl incidental\/}) stratification. Both are important, but  Denef-Loeser deal only with the uniform stratification. 

\begin{thm} \label{nearRatZeta} For each prime $p$, $\tP_{D_{P,\row Q m}}(t)$ is a rational function $\frac{n_p(t)} {d_p(t)}$, with $n_p,d_p\in \bQ[t]$ and computable.  The corresponding $Z_{D_{P,\row Q m}}(t)$ has the form $\exp(m_p^*(t))(\frac{n_p^*(t)} {d_p^*(t)})^{\frac 1 \ell_p}$ with $m_p^*,n_p^*,d_p^*\in \bQ[t]$ and $\ell_p\in \bZ^+$ computable. Further, there are bounds independent of $p$, for all those functions of $t$. \end{thm} 

\begin{proof}[Comments on the proof of Thm.~\ref{nearRatZeta}] These comments are highlights from  \cite[\S26.3]{FrJ86}$_1$ or \cite[\S31.3]{FrJ86}$_2$ (which are essentially identical) titled: Near rationality of the Zeta function of a Galois formula. We point especially to the effect of stratification choices and the use of Dwork's cohomology for the result.  What we say here applies equally to the uniform and incidental stratifications. 

The conclusion of the Galois stratification procedure over the $\bz$-space gives this computation for $\nu_p(D_{P,\row Q m},k)$. It is the sum of the $\bz$ with values in $\bF_{p^k}$ for which the Frobenius falls in the conjugacy classes attached to the piece of the stratification going through $\bz$.  

The expression of that sum in {\sl Dwork cohomology\/} is what makes the effectiveness statement in the Thm. possible, and this is what suggests its direct relation to Denef-Loeser. An ingredient for that is a formula of E.~Artin.  It computes any function on a group $G$ that is constant on conjugacy classes as a $\bQ$ linear combination of characters induced from the identity on cyclic subgroups of $G$. 

A function on $G$ that is 1 on a union of conjugacy classes, 0 off those conjugacy classes, is an example. \cite[p.~432-433]{FrJ86}$_1$ recognizes the L-series attached to that function as a sum of L-series attached to those special induced characters. I learned this from  \cite[p.~222]{CaFr67} and had already used it in \cite[\S2]{Fr74a}. Kiefe -- working with Ax -- learned it, as she used it in \cite{Ki76} -- from me, as a student  during my graduate course in Algebraic Number Theory at Stony Brook in 1971. The core of the course were notes  from Brumer's Fall 1965 course at UM. 

Kiefe \cite{Ki76}, however, applied it to the list-all-G\"odel-numbered-proof procedure  in \S\ref{galStrat}; not to the Galois stratification procedure I  showed her (see my Math Review of her paper, Nov.~1977, p.~1454). Consider the  the identity representation induced from a cyclic subgroup $\lrang{\sigma}$ of $G$. Then this  L-Series is the same as  the zeta function for the quotient of the cover by $\lrang{\sigma}$ \cite[exp. 7-9, p.~433]{FrJ86}$_1$. 

Given a rational function in $t$, its  {\sl total degree\/} is the sum of the numerator and denominator degrees; assuming those two are relatively prime. \cite[Lem.~26.13]{FrJ86}$_1$ refers to combining \cite{Dw66} and \cite{Bm78} to do the affine hypersurface case for  explicit bounds -- dependent only on the degree of the hypersurface -- on the total degree  of the rational functions that give these zeta functions. Then, some devissage gets back to our case, given explicit computations dependent only on the degrees of the functions defining these algebraic sets.  

Finally, \cite[Lem.~26.14]{FrJ86}$_1$ assures the stated polynomials in $t$ have  coefficients in $\bQ$, and it explicitly bounds their degrees. The trick  is to take the logarithmic derivative of the rational function. Then, the Poincar\'e series coefficients are power sums of the zeta-numerator zeros minus those  of the zeta-denominator zeros. Using allowable normalizations, once you've gone up to the coefficients of the total degree, you have determined the appropriate numerator and denominator of $\tP(t)$. 

One observation is left to  uniformly bound in $p$ the degrees of the zeta polynomials, etc. That is, we need a uniformity in the primes whereby you are applying the uniform stratification \eql{strat}{stratUnif}.   It comes from this that the degrees of polynomials describing the affine covers, in applying Dwork-Bombieri, do not change. 
\end{proof} 

\subsubsection{Chow Motive Coefficients}  \label{chowMotives} 
The comments on Thm.~\ref{nearRatZeta} show we can express  the coefficients in the Poincar\'e series  from the trace of Frobenius iterates acting on  the $p$-adic cohomology that underlies Dwork's zeta rationality result. 
Positive: The computation is effective. Negative: The cohomology underlying Dwork's construction varies with $p$. Nothing in 0 characteristic represents it. 

Even, however, with Dwork's cohomology (in his original proof in 1960), you deal with stratifying your original variety. By \lq\lq combining\rq\rq\ the different pieces you conclude the rationality of the zeta function from information on the Frobenius action from the hypersurface case. 

Every variety is birational to a hypersurface in some projective space. Yet, reverting to hypersurfaces requires stratifying the original space in a problem. Also,  \cite{FrJ86}  stratifies the underlying space to assure covers are unramified (no branch locus).  This is to have monodromy precision (\S\ref{MPres}) along each underlying piece of the stratification.  If you adhere to avoiding branch loci, then covers of projective spaces, for example, force refined stratifications. 

Denef and Loeser in \cite{DeLo01} applied Galois stratification (see the arXiv version of \cite[App.]{Hal07}) to eliminate quantifiers in their $p$-adic   problem goals. They phrased these as $p$-adic integrations generalizing Prob.~\eqref{Igusa}. \cite[\S26.4, last subsection]{FrJ86}$_1$ discusses several  $p$-adic problems, but there is no \cite[\S31.4]{FrJ86}$_2$  corresponding? 
The main Denef-Loeser innovation replaces Dwork cohomology of affine hypersurfaces, varying with $p$, with $\ell$-adic (\'etale) cohomology of {\sl projective nonsingular varieties\/} in 0 characteristic. 

That enhanced the uniformity in $p$ in the uniform stratification \eql{strat}{stratUnif}, the part of the stratification they used. 
The effect in  \cite{DeLo01} was to compute Poincar\'e series coefficients -- it worked for similar  reasons on their $p$-adic problems -- through coefficients  in the category of {\sl Chow motives}. 

Roughly:  an element in a Grothendieck group generated by nonsingular projective varieties replaces each piece of the uniform stratification. So,  each Poincar\'e coefficient is a formal \lq\lq sum\rq\rq\ of $\ell$-adic ($\ell\not =p$) vector spaces. This stratification replacement uses resolution of singularities in 0 characteristic.  \cite{Den84}, from a one-prime-at-a-time period, was a forerunner. For that alone the primes of the  incidental stratification were untouchable.  

A {\sl Tate twist} of a cohomology group is a tensoring of the group by some power of the cyclotomic character (\S\ref{BCLtreat}). If a nonsingular projective variety has coefficients in $\bQ$, then $G_\bQ$ acts on its Tate-twisted cohomology. 

The vector spaces come from the \'etale cohomology groups of projective nonsingular varieties. The word   {\sl motivic\/} means that the weighted pieces -- rather than from, say, the $m$th cohomology of a projective nonsingular variety -- might be a summand of this, tensored by a Tate twist. A correspondence -- cohomologically idempotent --  is attached to indicate the source of the projector that detaches a summand from the full weighted cohomology. As you vary primes of the uniform stratification, you compute the Poincar\'e series or zeta function coefficients by applying iterates of the $p$ Frobenius -- followed by the trace -- to the Chow motives. 

The Denef-L\"oser approach adds canonical zetas to the pure Galois stratification procedure. Still, it requires equivalences that relegate covers to the background of the final result. 

\subsubsection{\'Etale cohomology observations}  \label{3etale} 
Let $n$ be the modulus for an arithmetic progression
$A_a = A_{a,n} = \{a + kn \mid 0 \ge  k \in  Z\}$ with $0 \le  a \in  Z$.
Call $A_a$ a {\sl full progression\/} if $a < n$. A full {\sl Frobenius\/} progression $F_a = F_{a,n}$ is the union of the full arithmetic progressions
mod $n$ defined by all residue classes $a\cdot (\bZ/n)^*\! \mod n$. Example: The
full Frobenius progression $F_{2,12}$  is $A_{2,12} \cup  A_{10,12}$.

The following, including Prop.~\ref{frobProg},  is an extension of \cite[\S 8.2.2]{Fr05b}.   We call any $\bQ$-linear combination of series $\tP_{D_{P,\row Q m}}(t)$ (as in \eqref{WeilVectors})  a {\sl Weil vector}. For a particular Weil Vector $\tP_{D_P}$, its 0-{\sl support\/} is  the collection of  $k\in \bZ$ with the coefficient of $t^k$ equal to 0. Denote that $\Sup_{D_P}(0)$.  We say two Weil vectors  have a {\sl Weil relation\/} if their difference has an infinite 0-support. 

\begin{prop} \label{frobProg} For any Weil vector, $\Sup_{D_P}(0)$ differs by a finite (accidental) set from a union of full (possibly empty) Frobenius progressions. Dependent on the equations defining a 
Galois stratification, it is possible to find the
accidental set and union of Frobenius progressions attached to it explicitly.  
\end{prop} 

\begin{proof}  Consider the near rational zeta function, $Z(t)\eqdef\exp(m_p^*(t))(\frac{n_p^*(t)} {d_p^*(t)})^{\frac 1 \ell_p}$, attached to the Weil Vector by Thm.~\ref{nearRatZeta}.   
The polynomial $n_p^*$ has the form
$\prod_{i=1}^{m_1}(1-\alpha_it)$ while $d_p^*$ has the form
$\prod_{j=1}^{m_2}(1-\beta_jt)$.  The $\alpha_i\,$s and $\beta_j\,$s are complex numbers. 

Take the logarithmic derivative of  $Z(t)$. The result  
is a polynomial in $t$ plus  a constant multiple of an expression of form \begin{equation}
\sum_{k=0}^\infty \nu(D_P,k)t^k \eqdef \sum_{k=0}^\infty
\bigl(\sum_{i=1}^{m_1} \alpha_i^k-\sum_{j=1}^{m_2} \beta_j^k\bigr)t^k. \end{equation} 

The statement on Frobenius progressions follows by showing the collection
$$\Sup_{D_P}\eqdef\Bigl\{k\in \bN^+\mid \sum_{i=1}^{m_1} \alpha_i^k-\sum_{i=1}^{m_2}
\beta_j^k=0\Bigr\}$$ is a union of full Frobenius progressions. 

Lem.~\ref{hyperPlaneThm} is in 
\cite[Thm.~2.3.1]{V87} (result due to
\cite{vdP82}). The argument for curves in \cite[Median
Value Curve Statement 3.11]{Fr94} requires a modification for the general case. Take $L$ to be the field generated by all the $\alpha_i\,$s and $\beta_i\,$s.  
Then take $\Gamma$  to be the 
multiplicative subgroup generated by $\alpha_i/\alpha_1$, $i=2,\dots,m_1$, and 
$\beta_{i-m_2+1}/\alpha_1$, $i=m_1,\dots,m_1+m_2$, and -1.  

\begin{lem} \label{hyperPlaneThm} With $L$ a number field and $\Gamma$ a finitely
generated subgroup of $L^*$, all but finitely many
solutions in $\Gamma$  of
\begin{equation} \label{**} u_1+\cdots+u_n=1,\qquad u_i\in 
\Gamma \end{equation}
lie in one of the diagonal hyperplanes $H_I$ defined by the equation
$\sum_{i\in I}x_i=0$ with $I\subset  \{1,\dots,n\}$ and $2\le |I|\le n$.  
\end{lem} 

Apply this with $n=m_1+m_2-1$. So, excluding a finite subset,
elements of $\Sup_{D_P}(0)$ correspond to solutions on one of the hyperplanes
$$H_{I_1\cup I_2} (I_1\subset 
\{2,\dots,m_1\} \text{  and }I_2\subset  \{m_1+1,\dots,m_1+m_2\}).$$ For each such
$H_{I_1\cup I_2}$, denote the corresponding set of $k$ by $S(I_1, I_2)$. We show
$S(I_1,I_2)$, up to a finite set, is a union of full Frobenius progressions. Then,  running
over such 
$(I_1,I_2)$,  we get $\Sup_{D_P}(0)$ is such a union. 

Apply an induction on
$n$. Suppose for some infinite subset of $k\in S(I_1,I_2)$, there is a proper subset
$J$ of
$I_1\cup I_2$ for which 
$w_{i,t}=(\alpha_i/\alpha_1)^k\,$, $i\in I_1\cap J$ and
$w_{i,t}=-(\beta_{i-m_1}/\alpha_1)^k$,
$i\in I_2\cap J$, which   sum to 0.  That gives two proper subsets (for $J$ and
$I_1\cup I_2\setminus J$) summing to 0. Find a union of Frobenius progressions for the
first (using induction on
$n$), then we automatically get one for the second, giving such for $H_{I_1\cup I_2}$. Thus,
in heading for our conclusion, assume no infinite set of $k$ gives a proper subset of the 
$w_{i,k}\,$s summing to 0. Then, according to \cite[loc.~sit.]{V87}:  
\begin{triv} \label{condSum} For this set of $k$, the collection $w_{i,k}$ is constant in
$k$, for each $i$.  
\end{triv}
This says each of the $\alpha_i/\alpha_1$ and $\beta_i/\alpha_1$ are roots of 1.
Conclude this part of the theorem easily.  Under the hypothesis of explicit equations (given Thm.~\ref{nearRatZeta}), we
get  an explicit conclusion if  the argument above can be
made explicit. That is, we need only decide if various subsets of the $w_{i,k}\,$ sum
to 0, or are roots of 1. 
\end{proof}

\begin{rem} Prop.~\ref{frobProg} didn't attend to the cardinality
of the accidental set:  $k\in \Sup_{D_P}(0)$, yet not part of a Full Frobenius progression. 
\cite{Ev03} has the following result.  Let $K$ be a field of characteristic 0, and
$G\le K^*$ a finitely generated subgroup.  Consider linear equations $a_1x_1+\cdots
+a_nx_n=\ba\cdot \bx=1$, all $a_i\,$s nonzero,  with $\bx=(\row x n)\in G^n$.  He says 
$\ba$ and $\ba'$ are $G$-equivalent if there is $\bu\in G^n$ with $\ba=\bu\cdot \ba'$. 
Let $m(\ba,G)$ be the smallest $m$ for which the set of solutions of $\ba\cdot \bx=1$ is
contained in the union of $m$ proper linear subspaces of $K^n$. Clearly, 
$m(\ba,G)$ depends only on the $G$-equivalence class of $\ba$. It is also finite.  Gyory
and Evertse show (1988) that there is $c(n)$ so that, for all but finitely many
$G$-equivalence classes $\ba$, $ m(\ba,G)< c(n)$. \cite{Ev03} improves 
this to $c(n)=2^{n+1}$.\end{rem}

Let $X_{i,q} $, $i=1,2$,  be normal and projective over $\bF_q$ with this property:  
\begin{triv} \label{FqWeil} $|X_{1,q}(\bF_{q^k})|=|X_{2,q}(\bF_{q^k})|$ for $\infty$-ly many $k$. \end{triv} That is,  their Poincar\'e series have a Weil relation.  If we take an affirmative answer to Prob.~\ref{monpreciseext} as a working hypothesis, then our questions below extend to normal, rather than projective varieties. 
Prop.~\ref{frobProg} shows how to decide for such $X_{i,q}\,$s if they do have such a Weil relation. Now assume $X_{i,K} $  is a normal projective variety over a number field $K$, with its reduction mod $\bp$ denoted  $X_{i,K,\bp}$, $i=1,2$.  To  consider the global version of \eqref{FqWeil} assume this property:
\begin{triv} \label{KWeil} The Poincar\'e series for $X_{1,K,\bp}$ and $X_{2,K,\bp}$ have a Weil relation for infinitely many $\bp$. \end{triv}  

\begin{prob} Find a procedure like that of Prop.~\ref{frobProg} to check condition \eqref{KWeil} among the primes of the uniform stratification (in \eql{strat}{stratUnif}). \end{prob} 

A pr-exceptional cover $X\to Z$ (any cover of normal varieties) over a finite field $\bF_q$ is one for which $X(\bF_{q^k})\to Z(\bF_{q^k})$ is surjective for $\infty$-ly many $k$. Similarly for a pr-exceptional cover over a number field $K$ (see \eql{DavEx}{DavExd}). 
A pr-exceptional correspondence between $X_{1,q}$ and $X_{2,q}$ is an algebraic set $Y_q\subset X_{1,q}\times X_{2,q}$ over $\bF_q$ such that for $\infty$-ly many $k$, $Y_q$ is simultaneously -- by projection on the $i$th factor --  a pr-exceptional cover of $X_{i,q}$ over $\bF_{q^k}$, $i=1,2$.  Similarly, there is an analogous idea of a pr-exceptional correspondence $Y$ between $X_{1,K}$ and $X_{2,K}$.

Suppose, as above, $Y_q$ is a pr-exceptional correspondence with exactly {\sl  one\/}   absolutely irreducible component over $\bF_{q^k}$ for $\infty$-ly many $k$ in the support of the Weil relation \eqref{FqWeil}.  Then it is an {\sl exceptional correspondence\/} \cite[\S 3.1.2]{Fr05b}. 

Similarly, over a number field $K$, $Y$ is an exceptional correspondence if there are infinitely many $\bp$ for which reduction $\mod \bp$ is an exceptional correspondence. In the respective cases the conditions \eqref{FqWeil} and \eqref{KWeil} hold. \cite[Prop. 4.3]{Fr05a} notes that if $Y_q$ is an exceptional correspondence, then: 
\begin{triv} \label{excFrob}  the support of the Weil relation has a full Frobenius progression containing $k=1$, but it does not contain all $k$. \end{triv} When $X_2=\prP^m$ for some integer $m$ we refer to the Weil relation as having {\sl median value\/}. The case $m=1$ is significant. 

\begin{prob} \label{excCharext} Consider $X_{1,K},X_{2,K}$ satisfying \eqref{KWeil}, where  \eqref{excFrob} holds (for $X_{1,K,\bp},X_{2,K,\bp}$) for $\infty$-ly many $\bp$. Can you characterize this in Denef-Loeser cohomology components (\S\ref{chowMotives}). Give an example where there is no exceptional correspondence between $X_1$ and $X_2$. 
\end{prob} 

Recall condition \eql{DavRes}{DavResb} for Davenport pairs $X_i\to Z$ over $\bF_q$, $i=1,2$: The number of points of $X_i(\bF_q)$ having a given image $z\in Z(\bF_q)$  is independent of $i=1,2$. 
Such a Davenport pair is an i(sovalent)DP (over $\bF_q$). Then, \eqref{FqWeil} holds. Similarly, we have iDP\,s over a number field, and then  \eqref{KWeil} holds.  

For a cover $X\to Z$, denote its $u$-fold fiber product over $Z$ -- I apologize for  the overloaded notation --  by $X_Z^u$. \cite[Prop.~3.9]{Fr05a} characterizes the   iDP property by noting that there are pr-correspondences between $X_{1,Z}^u$ and $X_{2,Z}^u$, $u=1,\dots,n$, where $n$ is the common degree over $Z$ of the Davenport pair. So, it to is a monodromy precise condition (\S\ref{MPres}). Rem.~ \S\ref{infPGL} gives many dimension one iDPs. \cite[Prop.~8.2]{Fr05a}: For iDPs over $\bF_q$, the support set consists of all $k\ge 1$; and for iDPs over a number field, the support set  includes all but finitely many $\bp$. 

\begin{prob} In analogy with Prob.~\ref{excCharext}, characterize  when there is a $v$ and a system of pr-correspondences between $X_{1,Z}^u$ and $X_{2,Z}^u$, $u=1,\dots,v$ accounting for condition \eqref{FqWeil} over $\sO_K/\bp$ and all its finite extensions for almost all $\bp$. \end{prob} 

Finally, a  major problem would be to take advantage of the Denef-Loeser enhancement of Galois stratification, in the following form. 

\begin{prob} \label{unif-inc} Both quantitatively and qualitatively separate the primes of the incidental and uniform stratification. \end{prob} 

\subsubsection{Modestly motivic} 
Consider the Frobenius on the  \'etale cohomology pieces from Denef-Loeser in \S\ref{chowMotives}. Its  eigenvalues have absolute value determined by the weight of the cohomology and the Tate twist powers, from Deligne's proof of the Weil conjectures \cite{De74}. The Galois stratification procedure  produced the stratification pieces that allowed this application of \'etale cohomology. 

Still, once we have it, \cite{Fr86} aimed to distinguish  \lq\lq good\rq\rq\ and \lq\lq bad\rq\rq\ primes attached to a particular problem $D_P$. That is, to separate conceptually the uniform  from the incidental primes  in statements of, say, Prob.~\ref{unif-inc}.  
For one, the eigenvalues of the Frobenius don't have the same archimedian virtues in Dwork cohomology.  \cite{De80} has techniques for treating the Frobenius on \'etale cohomology for families of varieties, whose relevance \S\ref{App.5} hints at. 

For example, statements attached to our Davenport problems (over number fields) seem to have no bad primes -- either by \S\ref{App.4} theory or \S \ref{writeEquats} equations. This  contrasts with the primes that are exceptional for a given degree $d$ in Artin's Conjecture a la the Ax-Kochen \lq\lq solution\rq\rq\ (\S\ref{60sProblems}). 

Following Deligne's definition in \cite[p.~90]{Del89}, you might aim to attach a motivic object to a problem where it makes sense to consider various \lq\lq realizations:\rq\rq\ over the reals, $\ell$-adics and $p$-adics. So, a motivic cohomology would be cohomologically functorial  on appropriate algebraic varieties with a de Rham, \'etale  and, say, Dwork cohomology realization, when  they make sense. Deligne's treatise was about motivic integration giving  \lq\lq motivic\rq\rq\ interpretation of polylogs. 

\begin{prob} Produce objects as zeta coefficients that specialize to Chow motives at the uniform primes and to Dwork cohomology at the incidental primes. \end{prob} 

\cite{Fr86} inspected, based on flat covers,  how to avoid unnecessarily refining Galois stratifications. It also produced the definition of an L-series on a Galois stratification. That  starts from a Galois stratification on the {\sl base\/} (the space defined by no quantified variables; given by $\bz$ in, say, Prob.~\ref{stratproc}).  Flatness also appears in \cite{Be11} which talks up a relation with Thm.~\ref{nearRatZeta} considerably. I comment. 

The paper starts with a constructible equivalence relation over the base $B$ over a finite field. It considers the zeta function counting the $\bF_{q^k}$ equivalence classes and produces a zeta exactly as in Thm. \ref{nearRatZeta}, essentially by quoting it. 

A restatement: Given a constructible set $C$ in $\afA^{n+m}$ over $ \bF_q$, you form Poincar\'e series  coefficients $N_k = |\{ x\in A^n(\bF_{q^k}) |  p^{-1}(x)\cap C(\bF_{q^k})\}| $
where $p: A^{n+m} \to  A^n$  is the projection. Understatement:  The counting problem is a special case of ours, for it is pure existential,  in a 2-page Intersection-Union process section   \cite[\S2]{Fr76}. 

As in \cite[Def.~3.6]{Be11}, a good and flat stratification: \lq\lq A modicum
of care is needed to find an expression varying suitably \lq continuously\rq\  in flat families.\rq\rq Hilbert Schemes put edges on his stratification;  monodromy precision does  on ours. More applications of these zeta functions would test these stratification conditions. 

\subsection{Applied group theory and challenges occuring \lq in nature\rq} \label{secVII.3} 
The topic of what groups occur \lq in nature\rq\ started in \S \ref{secI.3} with a phrase of Solomon \cite{So01}. \S\ref{spacesP_f} reconsiders that. Some, however, might  prefer something less solemn like \cite{KSi08}  (authors based at UM) in the Scientific American as a substitute. Their article snuck in the topic of \lq what are simple groups?\rq\ through a spirited analog of Rubik's cube. They based this on a Mathieu group, $M_{12}$,  property: like all simple groups (consequence of  the classification) it requires just two generators. 

Most mathematicians, however, know that the technical -- rather than playful -- side of group theory tends to dominate. \S\ref{coGS} gets into  how you, even if you had little group theory training, could deal  with it. 

\subsubsection{Extending both RET and the genus 0 problem} \label{coGS} 
\cite{So01} wanted to document that the simple group classification -- including the so-called quasi-thin part questioned by Serre \cite[p.~79]{Se92} -- is available. That is, you may confidently apply it as we suggest below.  \cite[\S5]{Fr94b} inspects Serre's challenge in this light and concludes \lq\lq More than to complete our confidence in the classification, Gorenstein wanted it accessible to a researcher not dedicated to group theory.\rq\rq 

Experience shows  that most mathematicians who might use the monodromy method -- as in Davenport's Problem -- will require collaboration with a group theorist.
To show how that might work, I later took on one more problem in the Davenport range. That was a version of Schur's problem on polynomial covers, but restricted to finite fields of a fixed characteristic. 

Guralnick and Jan Saxl joined me in the 3rd section: Going through every step of the \cite{AOS85} classification, as in \S \ref{secVII.1} and \S \ref{App.2}. I was not a passive purveyor of Guralnick and Saxl. First, I caught the unusual new Schur covers for the primes 2 and 3 that were slipping by overly-optimistic group assumptions. Second, I showed how using \cite{AOS85} worked (\S \ref{App.2}). 

Expression \eqref{excChar} has the definition of an exceptional cover over a given finite field.  The original proof of Schur's conjecture in \cite{Fr70} easily described all exceptional (Schur) polynomial covers $f$ over a finite field $\bF_q$, when $\deg(f)$ is  prime to the characteristic. When this hypothesis does not hold, the ramification group $I_\infty$ over $\infty$ is no longer generated by a single element, $\sigma_\infty$  (from \S \ref{secIII.3}). 

Yet, a loosening of this statement works. There is a {\sl factorization\/}  $^aG_f(1)\cdot  I_\infty$ of $^aG_f$:  It is a set theoretic product of the stabilizer of a letter in the representation, and $I_\infty$.  Since $p$ divides $|I_\infty|$  any possible exceptional covers are  wildly ramified at a significant place. So, the traditional Riemann's Existence Theorem no longer applies, though we gained from experience with it. 

A composition of two polynomials over a finite field gives a  one-one map if only and if each is one-one. Conclude that a polynomial over a finite field is exceptional if and only if its composition factors over the field are. So, classifying exceptional polynomials over a finite field, reverts to assuming the arithmetic monodromy,  $^aG_f$, is primitive; $G_f$ maybe not. What I understood was that organizing \cite{AOS85} was Guralnick's job. Filling in possible factorizations of primitive groups that could arise was Saxl's -- based on his familiarity with \cite{LPS}.

We easily solved Dixon's 1897 conjecture classifying the exceptional covers of degree $p$ over a finite field of characteristic $p$ \cite[Thm.~8.1]{FrGS93}. Moreso, we extended his conjecture to describe all exceptional polynomials with geometric monodromy of the form $V\xs  C$ with $C$ cyclic and acting irreducibly
on $V = \bZ/p^a$, an especially easy affine group (\S\ref{App.1a}). These are the semi-linear polynomials of Cohen \cite{C90}. \cite[Cor.~11.2]{FrGS93} characterizes which of  these are indecomposable over $\bF_q$, but not over $\bar \bF_q$. This provides infinitely many examples showing the necessity of the hypothesis $p\not | n$ in the Polynomial Primitivity Lem.~\ref{decompPoly}. 

That described all
affine groups known then to be arithmetic monodromy groups of exceptional indecomposable polynomials. But, then an unexpected event caused the biggest stir. To understand, consider the two main results on exceptional  $f$ over finite field $\bF_q$, $q$ a power of $p$, that are not one of the examples above. \S\ref{App.1a} has the definition of $\text{P}\Gamma \text{L}_{p^a}$. 

\begin{edesc} \label{excPoly} \item  \label{excPolya} If $p\ne 2$ or 3, then $f$ has geometric monodromy an affine group  acting on $V=\bZ/p^a$ with $\deg(f)=p^a$ \cite[Thm.~13.6]{FrGS93}. 
\item  \label{excPolyb} If $p=2$ or 3 and $G_f$ is not affine as in \eql{excPoly}{excPolya}, then it is between $\PGL_2(p^a)$ and $\text{P}\Gamma \text{L}_{2^a}$ with $a\ge 3$ odd.  If $p = 2$ ,  $\deg(f) = 2^{a\nm 1}(2^a \nm1)$ and if 
$p = 3$, $\deg(f)= 3^a(3^a \nm1)/2$ (which is odd) \cite[Thm.~14.1]{FrGS93}. 
\end{edesc} 

With the group theory pointing the way in \eql{excPoly}{excPolyb},  Peter M\"uller came up with the first example. Then \cite{CM94} and \cite{LZ96} fulfilled 
the other degrees of these here-to-fore unexpected exceptional covers. 

We now use one-half (see \eql{grothEx}{grothExa}) of Grothendieck's famous RET version \cite{Gr59} that applies to tamely ramified covers in positive characteristic. It assures  that if we avoid primes dividing the orders of the groups that arise in Thm.~\ref{DS4}, or Prop.~\ref{DS6}, then the solution of Davenport's problem is essentially the same as it is in positive characteristic. That is, for such a prime, you can figure exactly the fields $\bF_q$ over which there are Davenport pairs $(f,g)$ with  $f$ and $g$ having exactly the same ranges over $\bF_{q^t}$ for {\sl every\/} integer $t\ge 1$. 

Yet, here, too, there is a surprise. If we allow wild ramification, instead of just those finitely many possible degrees 7, 11,13, 15, 21 and 31, we find a whole new infinite collection of Davenport pairs of degrees prime to the characteristic, arise over essentially every finite field. They aren't esoteric; we understand them precisely as an analog of the original Davenport pairs.  

Let $\lrang{j}_q\eqdef   1+q + q^2 + \dots + q^j$. \cite[Thm.~5.2]{Fr99} says, for each $\bF_q$ and each integer $m\ge 3$, there is a Davenport pair $(f,g)$ of degree $n=\lrang{m\nm 1}_q$ over $\bF_q$ with geometric monodromy group $\PGL_m(\bF_q)$. Also, $f(x)-g(y)$ has exactly two absolutely irreducible factors, one of degree $\lrang{m\nm 2}_q$. The result describes precisely the arithmetic monodromy group in each case. 

\cite{AbProjPol} explicitly gives the polynomials $f$. We take these as corresponding to the representation $T_f$ on points of projective space. After what works unchanged in this case from \cite{Fr73a}, the main problem is to  guarantee that the cover resulting from the representation of $\PGL_m(\bF_q)$ on hyperplanes also has genus 0. 

Since the cover for $f$ wildly ramifies,  R-H \eqref{*10} doesn't apply. We only know that its substitute depends on computing orbits of the higher inertia groups (in this case, from ramification over $z'=0$) as in \cite[Lem.~3.1]{Fr99}. As elsewhere, I didn't explicitly compute $g$ attached to $f$, but \cite{Bl04}  did. 

Thus, we see that the genus 0 problem has a different texture in positive characteristic.  In concentrating on Davenport's problem, there are immensely more covers in positive than 0 characteristic. Yet,  characteristic 0 illuminated the way. Ram Abhyankar's goals included producing all groups as Galois groups over the algebraic closure of positive  characteristic fields -- as with Grothendieck, there was no number theory objective --  from genus 0 covers. 

Though \cite{Ra94} solved the conjecture made in \cite{Abh57}, using Harbater patching -- as epitomized in \cite{H94} -- even to this day it is referred to as a conjecture.  The covers in Prop.~\ref{AbhyConj} violate both \eql{cycleconds}{cycleconds1} and product-one \eql{cycleconds}{cycleconds2}:  the RET constraints have no obvious analog in positive characteristic. 

\begin{prop}[Abhyankar's Conjecture] \label{AbhyConj}  Consider any finite group $G$ generated by its $p$-Sylows (including all simple groups of order divisible by $p$). Then, there is a Galois cover $f_G: X_G\to \prP^1_z$ with group $G$ ramified only over $z=\infty$. \end{prop} 

The critical proof piece in Abhyankar's Conjecture might have you despair of  ever figuring which simple groups of order divisible by $p$ might be \lq\lq characteristic $p$ genus 0 groups\rq\rq\ (as in \S\ref{Thompsonresp}). Yet, from 
\cite{G03} it is known, for any fixed $g$,  that many simple groups are not monodromy groups of 
genus $\le g$ covers of $\prP^1_z$. This defies Abhyankar's empirical Galois group producing attempts. 

Yes, the monodromy method works. Yet, solving Davenport's problem, as in \S\ref{App.5}, gives us  spaces whose points exactly correspond to production  of Davenport pairs. \S\ref{whenceRET} concludes this paper by discussing a result -- inspired by these examples -- that extends Grothendieck's Theorem to wildly ramified covers. 

\subsubsection{Davenport and M\"uller's Conjecture} \label{DavM} This subsection and the next consider  the immense divide between Davenport's problem and Schinzel's, once you drop the indecomposability (read, primitivity) assumption of, say, Prop.~\ref{DS5} that assures their essential equivalence.  

First consider Davenport's Problem (over $\bQ$). Peter M\"uller has gone after finding exceptions from polynomials with exactly two composition factors. His list \cite[p.~25]{Mu98}  considers $f(x) = a(b(x))$, $a,b \in  K[x]$ of degree exceeding 1 and each indecomposable ($K$ a number field). His conclusion:  $g$ has the form $a(b^*(x))$. 

He assumes  $(b,b^*)$ don't form a Davenport pair over $K$:  otherwise, composing any $a$  with both $b$  and $b^*$ gives an obvious  Davenport pair. He lists the finite many resulting  monodromy groups. He  notes \cite[p.~27]{Mu98} a recurrance from Thm.~\ref{DS2} (DS$_2$ ): $T_f $ and $T_g$ are equivalent as group representations. That is, as in \eql{DavRes}{DavResb} (or below \eqref{heir}),  the values of $f$ and $g$ are achieved with the same multiplicity. Finally, he has this conjecture  \cite[Conj.~11.3]{Mu98} (augmented by \cite{Mu06}), using the degree 8 pairs $(f_d,g_d)$  from Ex.~\ref{x^8exmp} up to our usual equivalence. 

\begin{guess}[M\"uller's Conjecture] \label{MullConj} Let $f, g \in \bQ[x]$ be a Davenport pair over $\bQ$. Then, they are either linearly equivalent over $\bQ$, or $f= h(f_d)$ and $g= h(g_d)$ for some polynomial $h \in \bQ[x]$ and $(f_d(x),g_d(x))$ as given above. \end{guess}  

I start to consider that there may be vastly different conclusions to  the Davenport and Schinzel hypotheses when  $f$ is decomposable. Consider a Galois cover over a number field $K$ with group ${}^aG$ having two faithful (no kernel) permutation representations $T_f$ and $T_g$. Assume these are inequivalent as permutation representations. (The $f$ and $g$ subscripts identify with our previous topics; we don't assume polynomials yet.) We summarize a hierarchy of conditions. Again, ${}^aG(T_f,1)$ is the stabilizer in ${}^aG$ of a particular letter on which $T_f$ acts. 

\begin{edesc} \label{heir} \item \label{heira}  $T_f$ and $T_g$ are are equivalent as group representations. 
\item \label{heirb} For each $\sigma \in {}^aG$, $\tr(T_f(\sigma))> 0 \Leftrightarrow \tr(T_g(\sigma))> 0$. 

\item \label{heirc} ${}^aG(T_f,1)$ is intransitive on the letters of the representation $T_g$. 
\end{edesc} 

We have a one group, two faithful representations, hypothesis. \cite[Lem.~3]{Fr73a} says \eql{heir}{heirb} implies \eql{heir}{heirc}: You need not assume the same degree. It also says   \eql{heir}{heirc}  -- restating Schinzel's hypothesis in  \eql{equats}{equats4}, that $f(x)-g(y)$ is reducible -- group theoretically.   If  $f$ is indecomposable, condition \eql{heir}{heira} -- equivalent to $T_f(\sigma)=T_g(\sigma)$ for each $\sigma\in {}^aG$ -- comes from Thm.~\ref{DS2}, \eql{DavRes}{DavResa}.  

Without assuming $f$ is indecomposable, \eql{heir}{heira} implies a S(trong) D(avenport) hypothesis from the converse statement of  \eql{DavRes}{DavResa}: For almost all primes $\bp$, not only are the ranges of $f$ and $g$ the same over $\sO_K/\bp$, but each element in the range is assumed with the same multiplicity. Condition  \eql{heir}{heirb}  is equivalent to the ranges are the same, but drops the \lq\lq with the same multiplicity\rq\rq\ conclusion. 

\cite[Lem.~2]{Fr73a} notes  \eql{heir}{heira} and \eql{heir}{heirb}  are equivalent if both $T_f$ and $T_g$ are doubly transitive, a conclusion of $f$ being an indecomposable polynomial. 

Yet, none of the the \eqref{heir} hypotheses include that the covers attached to $f$ and $g$ have genus 0. Also, we can proceed if desired to an algebraic closure, without regard to ranges over residue class fields. So, for  reducibility of variables separated expressions, we may consider if  \eql{heir}{heirb}, or even \eql{heir}{heira}, might hold, too. 

\subsubsection{Schinzel's problem and group challenges} \label{SchinzG}  Lem.~\ref{pullbackcomps} starts by noting that if $f=f_1\circ f_2$, $g=g_1\circ g_2$ and $f_1(u)-g_1(v)$ is reducible, then so is $f(x)-g(y)$. 

\begin{defn} \label{newlyReducible} Assume $f(x)-g(y)$ is reducible. Also, for no $(f_1,g_1)$ with  either $\deg(f_1)< \deg(f)$ or $\deg(g_1)< \deg(g)$ is $f_1(u)-g_1(v)$ is reducible. Then,  we say $(f,g)$ is {\sl newly reducible}. \end{defn} 

To properly focus on unknowns in Schinzel's problem,  we restrict attention to newly reducible $(f,g)$.  Further, Lem.~\ref{pullbackcomps} lets us conclude that for a newly reducible $(f,g)$, the Galois closures of the covers for $f$ and $g$ are the same. 

Recall the discussion of \cite{So01} in \S\ref{secI.3} asking about groups that occur in nature. If you assume that Schinzel's problem occurs \lq in nature,\rq\ then there is the challenge of non-primitive groups, which aren't close to simple groups. Now  I give two problems that  distinguish Schinzel \eql{equats}{equats4} from  Davenport \eql{equats}{equats3} (as in Conj.~\ref{MullConj}): The {\sl Reduced Equivalence Problem\/} and the $(m,n)$ Problem. 

The former starts like this. Assume  $f,g\in K[x]$, $\deg(f)> 1$,  are reduced equivalent (\S\ref{cont7}; but not affine equivalent over $\bar \bQ$, as in \S\ref{secI.1}). That is, up to affine change in $x$ and $y$, $g(x)=af(x)+b$, $a,b\in \bar \bQ$. Consider two possible events:  

\begin{edesc}  \label{redClos}  \item  \label{redClosa}  No translation of $f$ is affine equivalent to a cyclic polynomial and the covers $f,g:\prP^1_x\to \prP^1_z$ have the same geometric Galois closures; or 
\item  \label{redClosb}  no translation of $f$ is composite with a non-trivial cyclic polynomial and $f(x)-g(y)$ is reducible (\eql{heir}{heirc} holds). 
\end{edesc} 

Prop.~\ref{Gusic}  includes a quick proof of \cite[Thm.~3]{Gu10} with the same condition on $g$ as \eql{checkSeries}{checkSeriesb}, but it asks only  when is the variables separated expression reducible, without concern for the genus of the projective normalization of a component. Recall the branch cycle, $\sigma_\infty$, at $\infty$ for a polynomial cover from $\S\ref{secIII.3}$. As in \S\ref{secII.3} denote the (geometric) Galois closure of the cover for $f$ by $\hat f: \hat X_f \to\prP^1_z$. 
 
\begin{prop} \label{Gusic} We may assume  $a=\zeta_v=e^{2\pi i/v}$, $v\ne 1$, and translating $f$ by a constant, also that $g=\zeta_v f$ if either  \eql{redClos}{redClosa} or   \eql{redClos}{redClosb} holds. Then, $a$ acts as a permutation $u_a$  of the finite branch points,   

If \eql{redClos}{redClosa} holds, then $z\mapsto az+b$ gives a cyclic cover $\mu: \prP^1_z\to \prP^1_{u}$ with group $\lrang{a^*}=\bZ/v$ where the following holds.  The composite cover $\mu \circ \hat f: \hat X_f \to \prP^1_u$ is Galois. If $\sigma_\infty^*\in G_{\mu \circ \hat f}$ is a branch cycle over $\infty$ for $\mu \circ \hat f$, then we can take its natural image in $\lrang{a^*}$  to be $a^*$, and $\sigma_\infty =(\sigma^*_\infty)^{v}$. Denote conjugation by $\sigma_\infty^*$ by $c_{\text{\rm AZ}}$. It has trivial action on $\sigma_\infty$, and no element of $S_n$ represents it. 

Let $Z'$ be a cycle of branch points under $u_a$. If \eql{heir}{heira} (resp.~\eql{heir}{heirb}) holds, then  $\tr(T_f(\sigma_z'))$ is constant for (resp.~ $\tr(T_f(\sigma_z'))>0$ holds, independent of) $z'\in Z$.  \end{prop} 

\begin{proof} Assume \eql{redClos}{redClosa} holds.  Then  the covers given by $f$ and $g$ have  exactly the same branch points. If $a= 1$, then translation by $b$ permutes the finite branch points of $f$. The only translation mapping a finite set in the complex plane into itself is $b=0$. So, this contradicts that $f$ and $g$ are affine inequivalent. 

So, we may assume $a\ne 1$. Substitute $f(x)$ by $f(x)+c$ with $c=b/(1-a)$.  Then, with no loss,  $b=0$. Now our hypothesis says that multiplying by $a$ permutes the finite branch points of $f$. Unless those branch points only consist of $0$ -- so $f$ is a cyclic polynomial contrary to assumption -- then $a$ must be a root of 1. 

Now assume \eql{redClos}{redClosb} holds.  \cite[Prop.~2]{Fr73a}, as  in \eqref{schcond}, says $f_1\circ f_2=f$, and $g_1\circ g_2=g$, where  $f_1$ and $g_1$ satisfy \eql{redClos}{redClosa}; and factors of $f(x)-g(y)$ correspond one-one with those of $f_1(x)-g_1(y)$ with $\deg(f_1)=\deg(g_1)$.  

\cite[Prop.~3.4]{FrM69} says, up to affine equivalence, at most one composition factor, $f_1$ (resp.~$g_1$),  of $f$ (resp.~$g$) has a given degree. So,  we know $g_1=af_1+b$, $(f_1,g_1)$ satisfy \eql{redClos}{redClosa}, and the final conclusion holds in this case, too. 

Assume, again, \eql{redClos}{redClosa} holds to address the 2nd sentence.  Assume the normalization above.  Expand a solution, $x$, of $f(x)=z$ over $z=\infty$ as a Laurent series  in $1/z^{-\frac1 n}$. Express all solutions as $ x(\zeta_n^j/z^{-\frac1 n})\eqdef x_j$, $j=0,\dots,n\nm 1$. 
The hypothesis about $a$ says that  the substitution $\sigma_\infty^*: 1/z^{-\frac1 n} \mapsto \zeta_v/z^{-\frac1 n}$ in all the $x_i\,$s gives elements in the field generated by the $x_j\,$s. The fixed field of $\sigma_\infty^*$ and $G_f$ identifies, with $u=z^{v}$, with $\bC(u)$. Since $\sigma_\infty$ is a power of $\sigma_\infty^*$, the two elements commute. As  in the proof of Lem.~\ref{commuteCyc}, the only elements of $S_n$ commuting with $\sigma_\infty$ (an $n$-cycle) are powers of $\sigma_\infty$. So conjugation by $\sigma_\infty^*$ cannot act through $S_n$. 

Finally, consider a branch point  $z'\in Z$ in the statement. The branch cycle for $az'$ and the cover  $f$ is in the conjugacy class of the branch cycle for $z'$ for the cover $g$.  For example, if  \eql{heir}{heirb} holds, then the positive trace condition must hold simultaneously for both $T_f$ and $T_g$, if it holds for one, etc. \end{proof} 

\begin{guess} \label{gusicConj} If \eql{redClos}{redClosb} holds, but $f(x)-g(y)$ is newly reducible, then $a=-1$, and $\deg(f)=4$ \cite[Conj.]{Gu10}. \end{guess}  \cite{FrGu11} interprets Prop.~\ref{Gusic} entirely in branch cycles. That means it is matter about groups, but here we must face the challenge of dealing with imprimitive groups.  \S\ref{modgenRitt} introduces notation for the Galois closure group of a composite of covers as a subgroup of a wreath product.  In Rem.~\ref{wreath} the whole wreath product occurs. Here, however,  the actual $G_{\mu \circ \hat f}\eqdef G_{f^*}$ is  the smallest subgroup of the full wreath product, $G_f\wr \bZ/v=G_f^v \xs \bZ/v$, satisfying  wreath conditions \eqref{wreathconds}. 

The key element inside  $G_{f^*}$ is the $n\cdot v$-cycle $\sigma_\infty^*$. Akin to the computation in Rem.~\ref{wreath}, identify $v$ copies of $\{1,\dots,n\}$ as $\{1_i,\dots,n_i\}$, $i=1,\dots, v$. With no loss, up to renaming the letters -- using that $(\sigma_\infty^*)^v=\sigma_\infty$ -- you can take $\sigma_\infty^*$ as  
$$ (1_1\,1_2\,\dots\,1_v\,2_1\,\dots\,2_v\,\dots \,n\nm1_1\,\dots \,n\nm1_v\,n_1\,\dots\,n_v).$$ 
Then, as on \cite[p.~47]{Fr70} (see Lem.~\ref{ChebChar}), the conjecture is true if and only if $\sigma_\infty$ generates a normal subgroup in $G$. Exactly then, the other branch cycles acting by conjugation on $\lrang{\sigma_\infty}$ have precisely determinable branch indices; the result is that $f$ is equivalent to a Chebychev (or cyclic) polynomial. 

From Lem.~\ref{ChebChar} we see that the only possibility in this case to assure newly reducible is that $n$ must be even. Yet, even then if $n> 4$, $f=f_1\circ f_2$ with $f_1$ a proper composition Chebychev factor, of degree either odd or 4. So, $g$ has the proper composition factor $-f_1$, and from Lem.~\ref{ChebChar}, $(f,g)$ isn't newly reducible. Note for $n=4$, from Lem.~\ref{ChebChar}, since one finite branch cycle has shape $(2)(2)$ the other of shape $(2)$, \eql{heir}{heirb} does not hold. That is, $(f,g)$ is not a Davenport pair. 

A bigger context for Conj.~\ref{gusicConj} starts with  $f:\prP^1_x\to \prP^1_z$,  $f\in \bC(x)$ and with some torsion $\alpha\in \PGL_2(\bC)$, giving $g\eqdef\alpha\circ f:\prP^1_x\to \prP^1_z$ where $f$ and $g$ have the same Galois closures (as  in \eqref{heir}). 
\begin{prob}   Classify this. Then, restrict to  the subcase where $f$ is a polynomial and decide when $(T_f,T_g)$ could form a Schinzel pair (satisfy \eql{heir}{heirc}). \end{prob} 

The wreath product challenge given by the $(m,n)$ Problem starts with polynomials with simple finite branch points, akin to literature quoted in \S\ref{modgenRitt}. 

\begin{prob}[$(m,n)$ Problem] For a \lq general\rq\ pair $(f',g')$ of polynomials (over the complexes), of respective degrees $m$ and $n$, with $n\ge 3$, is  the following true? 
\begin{triv} \label{*11} No matter what are the nonconstant polynomials $f''(x)$ and $g''(y$),  \\ $f'(f''(x))-g'(g''(y))$ is irreducible \cite[p.~17]{Fr87}. \end{triv} \end{prob} 

\cite[p.~18]{Fr87} has branch cycles for such $(f'(f''(x)),g'(g''(y)))$ of degree 4, given any degree 2 pair $(f',g')$, so that $f'(f''(x))-g'(g''(y))$ reducible. This is essentially the factorization in the case $n=4$ from Lem.~\ref{ChebChar}; also the one case of Conj.~\ref{gusicConj}. That is, the excluded (2,2) problem is false.  

It suffices to take for $(f',g')$ any polynomials of respective degrees $m$ and $n$ ($\ge 3$) giving simple-branched covers, and, outside $\infty$, disjoint branch points. Then, the $(m,n)$ problem holds if, for  nonconstant $f''(x)$ and $g''(y)$ (their degrees are irrelevant), $f'(f''(x))-g'(g''(y))$ is irreducible.

Let $N$ be the least common multiple of $m$ and $n$. Then, the reduction in Thm.~\ref{DS2} shows it suffices to consider $\deg(f'')=kN/m$, $\deg(g'')=kN/n$.

For example, in the (2,3)-problem:  it suffices to consider $f''(x)$ and $g''(y)$ of respective degrees $3k$ and $2k$. \cite[Prop. 2.10]{Fr87} shows neither $k=1$ or 2 gives a contradiction to \eqref{*11}. Still, there was a close  call already with $k=2$ for providing new Schinzel pairs (satisfy \eql{heir}{heirc}), except for a failure of the genus 0 (from Riemann-Hurwitz, \eqref{*10}) condition.

\subsection{Final UM and RET Comments} \label{secVII.4} What attributes would make it clear that I took great advantage from my three years at UM? For me, these come to mind.  I was (almost) never frightened by prestigious mathematicians, or by being on my own in hot-house mathematical environments. Yet, even papers solving long unsolved problems appearing in prestigious journals didn't do much for either myself or those who found those problems attractive.

My career (barely) survived by my interactions with European and Israeli mathematicians, doing what they wanted me to, rather than what my own convictions suggested. Later, I turned to the topics I'd put aside for years.

\subsubsection{UM upon my graduation} \label{grad69} There were over 200 grad students at UM in 1967. I have seen only one from my graduate years more than once after grad school. That was the topologist Bob Edwards who twice sat in on talks of mine at AMS conferences. It would have helped if other UM students, even slightly related, interacted with me from the hundreds of talks I've given, from the many papers to which I've corresponded with -- especially, young -- authors, or the many conferences I've attended or run. Especially for the effort I've put into level-raising and correction of papers for which editors claimed they previously found no referees. 

The three others who got PhDs in 1967 were all analysts, one much more famous than anyone who might be reading this. That was \lq\lq The Unabomber,\rq\rq\ a no-show at the going away party Paul Halmos gave us. You can find a picture of me from years related here --  opposite the page with Grothendieck -- in \cite{HMpicture}. I 'm standing in front of my Schur Conjecture diagram at the end of my 1968 UM lecture on it.

I didn't know about that picture until many years later, just prior to my giving a talk at a conference that, excluding myself, were Harvard affiliated arithmetic geometers in Tempe, Arizona.  Several at that conference were visibly upset that I had maneuvered to give an hour talk. This was thanks to Armand Brumer -- a snowstorm interrupted no-show -- conceding his spot to me.

I discovered Halmos' picture by accident during the coffee break before my talk, while I was purposely off in a side commons room. It was appropriate inspiration -- showing a 25 year-old me, facing the UM audience, in a confident pose. That  helped me handle with equanamity giving my 1987 talk to a likely antagonistic audience. One -- younger than myself -- Harvard faculty member asked me before the talk of my topic. It was a presentation of $G_\bQ$, related, but superior in ways, to that from \cite{FrV92}. His response: \lq\lq Well, that would be a dream come true!\rq\rq\ I never heard another word from him after my presentation, and publication in the conference volume, about the  \lq dream come true.\rq\ 

At the '68 UM talk, Mort Brown (from whom I had algebraic topology) and Jim Kister (a course in vector/micro bundles) had left early while Davenport held forth after my talk. They came up to me later, to explain why they left. They were annoyed by Davenport's remarks, which seemed to suggest that there was nothing new in what I had done. Halmos's picture had a surprisingly sympathetic caption under it about the mathematical direction I seemed to be going, perhaps influenced by how well I had handled Davenport's \lq\lq interrogation.\rq\rq 

Halmos' picture helped me do better than just get through that Tempe Arizona talk. Still, either I, or the Schur Conjecture, must have been funny. Once I saw that picture, I realized it was the answer to a New Yorker cartoon -- containing a version of my Schur Conjecture diagram -- that I had puzzled over years before. It was posted on Paul Kumpel's (a Stony Brook colleague) office door. It  charicatured (I now saw) my  satisfaction with  that  diagram. 

\subsubsection{More on RET?} \label{whenceRET} LeVeque had translated to English Siegel's proof of his Thm. (\S\ref{limitAlt}). That introduced me to $\theta$ functions. Especially, the production from them of an arithmetic form of Riemann's version of Abel's Thm:  {\sl Weil's Decomposition Theorem\/}. Despite its masterful use in the Mordell-Weil Theorem \cite{We28}, you don't see it much  these days. It gave an  apparatus relating function theory and statements about rational points. That topic, led to the influence of Siegel's papers and Riemann upon me.   Springer's book \cite{Sp57}, on Riemann Surfaces, has neither  RET  nor much group theory savvy.  The proof of RET in \cite{chpret4-firsthalf.pdf} is mine. So is the particular use of braids, albeit  braids were long ago in the literature. 

Some mathematicians (several co-writers included) either have no training with analytic continuation, or  like neither it nor paths, etc. One who was in this category, but not a cowriter, had been particularly critical of the value of \cite{Fr77} on a Harvard stage in the late '80s. So, it seems perfectly appropriate that \cite[p.~480, Rem]{Se90} is the residue of my correcting his initial guess at a formula, and informing him he had seen the technique at the  Delange-Pisout-Poutteau talk for \cite{Fr90}. 

Let $R_\bp$ be the completion of the ring of integers of some number field at a non-archimedian prime $\bp$. The (integral domain of) Witt vectors, $\bar R_\bp$, attached to $R_\bp$ contains the latter, and  a generator of its maximal ideal generates the maximal ideal of $\bar R_\bp$. They differ essentially only in that the residue class field of the former is $\bar \bF_p$, rather than $\bZ/p$. Denote by $W_\bp$ the quotient field of $\bar R_\bp$. 

\cite[Thm.~3.3]{Fr99} has a form of Grothendieck's Thm., \cite{Gr59}, emphasizing  it is a result about families of covers attached to a given Nielsen class $\ni(G,\bfC)$ over the base (parameter) space $\Spec(\bar \bR_p)$: a tiny space, but significantly more than one point. Assume $(N_\bfC,p)=1$ (\S\ref{BCLtreat}). The result is that you can form a smooth family with a constant Nielsen class  in either of two situations. 

\begin{edesc} \label{grothEx} \item \label{grothExa}  Start with $f_{W_\bp}: X_{W_\bp}\to \prP^1_z$, a cover over $W_\bp$, {\sl with  $p'$ monodromy group}, but the family ends up over $\bar R_{\bp'}$ a possibly larger Witt vector ring. The family then has a cover equivalent to $f_{W_\bp}$ over its generic point. 
\item \label{grothExb} You start with $f_{\bar \bZ/p}: X_{\bar \bZ/p}\to \prP^1_z$, a tamely ramified cover over $\bar\bZ/p$. The family has this cover over its special point. \end{edesc} Each result  refers to $\prP^1_z$, though the spaces are  over different fields. That is, there is a natural family of $\prP^1_z\,$s reasonably labeled $\prP^1_{z,\bar \bR_p}$. Grothendieck's use of Abhyankar's Lemma in \S\ref{genRitt} produced the change of base in \eql{grothEx}{grothExa}. I understood Grothendieck's theorem from the detailed exposition in \cite{Fu66}, referenced in \cite{Fr70} and discovered in Spring 1968 by accident while I was at IAS. 

Suppose $\Psi: \sT\to \sF\times \prP^1_z$ is a smooth family of $r$ (distinct) branch point covers, with  $\sF$ absolutely irreducible. (Generalizing polynomial families as in \S\ref{polySpaces}.)   Grothendieck's theorem gives the following for tamely ramified covers in positive characteristic, from it holding  in characteristic 0. 

 \begin{triv} \label{constantFam} If the branch points, as a function of $\bp\in \sF$, are constant, then there is an \'etale cover  $\sF'\to \sF$, so that  the  family's pullback over $\sF'$ is  constant. \end{triv} 

In characteristic 0 this reverts to its truth locally in the complex topology. Then,  if the branch points don't move, you don't need to move the classical generators or the base point for them, either. That means, the branch cycle description of the cover doesn't change, and  all covers nearby a given $\bp\in \sF$ are equivalent. 

Prop.~\ref{famWild} includes an analog of \eqref{constantFam} which also holds for wildly ramified covers. All spaces and covers are defined over the algebraic closure of a finite field. We use the phrase \lq\lq in the finite topology\rq\rq\ to mean that we can adjust any morphism by pullback over a finite, not necessarily flat (\S\ref{roleFlatness}), morphism. 

Suppose $f: X\to \prP^1_z$ is a wildly ramified cover. Then, \cite[Iso-trivial Prop.~6.8]{FrMz02} constructs an explicit {\sl configuration space\/}   $\sP_f$ -- generalizing the role of $U_r$ to wild ramification -- with the following property. 

\begin{prop} \label{famWild} Given any irreducible smooth family of covers $\Phi: \sT\to \sP\times \prP^1_z$ containing $f$ at a particular fiber $\bp\in \sP$, then  -- in the finite topology -- there is a morphism (unique  in the finite topology) $\Psi_{\sP,\sP_f}: \sP\to \sP_f$. 

Over the range $\sR_\Psi$ of $\Psi_{\sP,\sP_f}$ there is a finite cover $\sP_\Psi\to \sR_\Psi$ that supports a family of covers of $\prP^1_z$ whose pullback by $\Psi_{\sP,\sP_f}$ is equivalent to $\Phi$. Further, $\Psi_{\sP,\sP_f}$ is constant if and only if $\Phi$ is constant (in the finite topology). \end{prop}  

\subsubsection{Families over the space $\sP_f$}  \label{spacesP_f} Denote the ring of formal power series over $\bar k$ by $\bar k[[z]]$. In constructing $\sP_f$ we must deal with this:  \begin{triv} \label{wildram} There are many more wildly, versus tamely, ramified local (separable) ring extensions of $\bar k[[z]]$. \end{triv} 

Further there is a serious complication with going to the Galois closure. Look again at  \lq\lq grabbing a cover by its branch points\rq\rq\ in \S\ref{HurMonUr}. The construction allowed uniquely continuing  a given cover, with  branch points $\bz_0\in U_r$, to a cover with branch points $\bz\in U_r$ along any path in $U_r$ between $\bz_0$ and $\bz$. The branch cycle description continues along the path. So the geometric monodromy -- generated by the branch cycles -- is locally constant. 

Assume we start with any Nielsen class $\ni(G,\bfC)^*$ of $r$-branch point covers, $*$ indicating absolute or inner equivalence.  Over $\bC$, there is always a Hurwitz space $\sH(G,\bfC)^*$. \cite[\S3-\S4]{Fr77} considers the existence of total families $\Phi: \sT\to \sP\times\prP^1_z$ with fibers  $\sT_\bp\to \bp\times \prP^1_z$ that are covers in $\ni(G,\bfC)^*$.  The proof shows by the nature of  $\sH(G,\bfC)^*$, any such family induces an analytic map $\Psi: \sP\to \sH(G,\bfC)^*$ with  $\Psi(\bp)$ the point representing the equivalence class of the fiber. Prop.~\ref{absFamCovers} notes that if fine moduli conditions hold, then there is a family  over $\sH(G,\bfC)^*$ so that the family $\Phi$ is the pullback  by $\Psi$ of this family.  

That construction also includes Prop.~\ref{repFam}, even without fine moduli.  

\begin{prop} \label{repFam}  For $r\ge 3$, there is an \'etale (unramified) cover $\sP\to \sH(G,\bfC)^*$ supporting a {\sl total representing space}.  That is, in one fell swoop, all covers  in $\ni(G,\bfC)^*$  are in one family over $\sP$, though possibly many times.  \end{prop} 

(\cite[\S 3, Ex.~2]{Fr77} shows $r=2$ does not work.) \cite[Prop.~3]{Fr77} gives a condition that shows even without fine moduli we  can choose $\sP=\sH(G,\bfC)^*$ in Prop.~\ref{repFam}. 

\begin{triv}  \label{primetoG} From Grothendieck: If $(p,|G|)=1$, the conclusions just above are the same over the algebraic closure of $\bZ/p$; ditto the fine moduli condition. \end{triv} 

Now consider the other half of Grothendieck, starting with a Nielsen class  and a tamely ramified cover $\phi_0: X\to \prP^1_z$ in this class --$(N_\bfC,p)=1$ as in \S\ref{secVI.1} -- from characteristic $p$ where possibly $(p,|G|)=p$.  Lifting $p$-adically  does allow comparison with results in the complex topology. You can then analytically continue the lifted cover along a path in characteristic 0. Also, the geometric monodromy  is constant in any smooth family of $r$-branch point covers over an irreducible base. 
\begin{triv} \label{NCp=1} Yet, if you only know $(N_\bfC,p)=1$, you may not be able to reduce modulo $p$. You don't know how \lq\lq far\rq\rq\ in characteristic $p$ the cover extends. \end{triv}

By contrast, even the Galois closure of the quotient fields of wildly ramified extensions can change in a family without moving the branch points. Abelian wild ramification is not a good model for this. 
That is,  without $(p,|G|)=1$, there is no notion of continuing a characteristic $p$ cover  with branch points $\bz_0$ to one with branch points $\bz$; not  even with tame ramification. Indeed, for some $\bz\in U_r$, there  may be  no such cover in the Nielsen class in positive characteristic. An extreme version of being supersingular, akin to how supersingular points occur in the modular curve Nielsen class (Ex.~\ref{modcurves}). 

The space $\sP_f$ in Prop.~\ref{famWild} depends on computing two sets of data from the cover $f$:  {\sl ramificiation\/} and and {\sl regular ramification\/} data (introduced first in \cite[\S1]{Fr74c}). The former  is an array  -- indexed by points $x'\in X$ ramified over $\prP^1_z$. Each element in the array is a Newton polygon attached to a not necessarily Galois extension $\bar k((x^*))/\bar k((z))$ with $x^*$ a uniformizing parameter around $x'$. Regular ramification refers to the convex hull of this. \cite[Lem.~5.1]{FrMz02} gives a rubric based on computing  the number of tame embeddings of $\bar k((x^*))/\bar k((z))$. From the slopes in the regular ramification data one computes the composite ramification index of all the tame embeddings.

Some properties of $\sP_f$ as a {\sl configuration\/} space use \cite{Ga96} as reformulated in \cite[Thm.~6.6]{FrMz02}: wild ramification does have a significant lifting to 0 characteristic using curves with ordinary cusps. Here is the fundamental problem. 

\begin{prob}  What part of $\sP_f$ is in the image of a family of covers with given ramification data.\end{prob}   Our approach, assuming $(|G|,p)=p$,  puts the case of wild ramification and tame ramification under one roof. Problems about Davenport pairs and exceptional covers also fit under one roof, as in \cite{Fr05b}. To solve this problem in positive characteristic, no simple reversion to Galois covers works.  

Continuing \S\ref{secVII.3}, Solomon didn't define the phrase\lq appearing in nature.\rq\ Maybe he won't consider these problems as being \lq in nature.\rq\  My response is to ask: Do any rational functions -- in positive or 0 characteristic -- appear \lq in nature?\rq\  As \S\ref{secI.3} notes, characteristic 0 rational functions are intrinsically impossible in 3 dimensions. The same for electricity and magnetism: Many electromagnetic spheres in the world composed, say,  of protein molecules, interact. Those interactions are mostly from van der Waals attractions, hydrogen and ionic bonds. Are these what we should regard as appearing in nature? Or is it the symmetry groups of molecules or particle arrays by which chemists  interpret quantum mechanics that we should regard as in nature?  If the latter I doubt that the topic is any more restricted to simple groups than should the topics be that I've presented here. 

\begin{appendix}
\renewcommand{\labelenumi}{{{\rm (\teql \alph{enumi})}}} 

\section{Group and cover comments} \label{App.1} Standard field notation for an algebraic closure of a field $K$ is $\bar K$. A finite extension $L/K$ is one in which $L$ is  finite dimensional, as a vector space over $K$. That dimension is $\deg(L/K)\eqdef [L:K]$, the {\sl degree\/} of $L/K$.  Any finite extension of $K$ has a field embedding, as an extension of $K$, in $\bar K$. If $L/K$ is separable, the number of such embeddings is $\deg(L/K)$; all characteristic 0 fields (and finite fields) have only separable extensions. 

The maximal cardinality  of automorphisms of $L/K$ (of $L$ fixed on $K$) is $[L:K]$, a cardinality achieved exactly when $L/K$ is Galois.  A field $K$ is {\sl  perfect \/} if it has only separable finite extensions. In that case, $\bar K/K$ is Galois, in that it is a union of Galois extensions. Denote the projective limit of those groups by $G_K$. We call it the {\sl absolute Galois group\/} of $K$. \cite{FrJ86} distinguishes properties of fields by enhancing Galois theory.  It uses no covering space theory or  fundamental groups. 
 
\subsection{Affine groups and related topics} \label{App.1a} Use the notation of \S\ref{secV.3}. An $n$ dimensional group representation of a group $G$ over a field $K$ is a homomorphism $T: G\to \GL_n(K)$. It's character is the function $\sigma\in G\mapsto \tr(T(\sigma))$: $\tr$ denotes the trace of the matrix. The symmetric group on $\{1,\dots,n\}$, $S_n$, natural embeds in $\GL_n(\bQ)$ by mapping a permutation $\sigma(i)=j_i$, $i=1,\dots,n$,  to the matrix with 1 in all $(i,j_i)$ positions, 0 elsewhere. We can apply $\tr$ to a permutation representation. The result is the number of fixed points of $T(\sigma)$.   

\S\ref{secV.3} has defined $\PGL_n(K)$, and there is similarly $\PSL_n(K)$, the quotient of the matrices of determinant 1 over the field $K$ by its diagonal matrices. The relation between primitive groups and simple groups starts by recognizing that the two most common sets of finite, far from abelian groups, are symmetric groups, $S_n\,$s,  and general linear groups, $\GL_n(\bF_q)\,$s, where $q$ is a power $p^t$ of some prime $p$. For most values of $n$ (and $p$) both are in evident ways close to simple.  We call these groups {\sl almost simple\/} for those values $n\ge 5$ (resp. $n$ and $q$, excluding $n=2$ and $p=2$ or 3) for which $A_n$ (resp.~$\PSL_n(\bF_q)$) is simple \cite[Thm. 4.10]{Ar57}. 

The goals of algebraic covers and  group theory don't match perfectly. For the latter, at the end of the 20th century there was an emphasis on   the simple group classification. This could sometimes strip a group to an essential core, tossing  data of significance for covers. We give the full definition of almost simple, to show what it means to get to that core. Still, by staying with primitive groups --  a concept natural for covers -- App.~\ref{App.2} reminds of a tool sufficient, modulo considerable expertise, for handling covers from knowledge of simple groups. 

According to \cite{GLS}, a {\sl quasisimple\/} group $G$ is a perfect central cover $G \to S$ of a simple group $S$. Here: {\sl cover\/} means onto homomorphism; {\sl perfect\/} means the commutators $g_1g_2g_1^{-1}g_2^{-1}$ in $G$ generate $G$; and {\sl central\/} means the kernel is in the center of $G$. Such a cover is a special case -- because we don't assume $S$ is simple -- of a {\sl Frattini\/} central cover: where the map, if restricted to a proper subgroup of $G$, won't be a cover.  Then, if $S$ is perfect, so is $G$.

A component, $H\le G$, of $G$, is a quasisimple subgroup which has, between $H$ and $G$, a composition series -- a sequence of groups each normal in the next.  The group generated by components and the maximal normal nilpotent subgroup of $G$ is called the {\sl generalized Fitting subgroup}, $F^*(G)$, of $G$.  \cite{GLS} calls a group $G$ almost simple if $F^*(G)$ is quasisimple.

We don't lose the almost simple property if we extend $\PGL_n(\bF_q)$ to $\text{P}\Gamma\text{L}_n(\bF_q)$, the extension given  by  adjoining a Frobenius, $\Fr_p$ ($p$th power map on coordinates),  for $\bF_p$ to $\PGL_n(\bF_q)$. That extends  permutations on lines and hyperplanes (on linear spaces of any dimension). The notation differs  from its use today, but \cite[Chap.~XII]{Ca37}  is where I learned about these groups in graduate school.

A {\sl chief series} of a group $G$ is a maximal series of normal subgroups of $G$ (no possible further refinement of the series with normal subgroups of $G$, \cite[p.~102]{Is94}). 
Supersolvable means  $G$ has a chief series whose consecutive  subquotients have prime order, and then the commutator subgroup of $G$ is nilpotent \cite[p.~133]{Is94}.

An affine group is a subgroup of the full group that combines the  actions of $\GL_n(\bF_q)$ and translations on the vector space $(\bF_{q})^n$ of dimension $n$ over $\bF_q$. The case that arose in Burnside's Theorem (\S\ref{secIV.1}) is $n=1$. 

\subsection{Residue class fields and their relation to general algebra} \label{App.1b}  The {\sl normalization\/} subject described in \S\ref{secII.1} applies to any  finite extension $K$ -- number field -- of $\bQ$. The elements, $O_K$, of $K$ satisfying a monic polynomial over $\bZ$ are called  its integral closure (or its ring of integers). Excluding the 0 ideal, all prime ideals $\bp$ are maximal. So their residue classes, $O_K/\bp$, are fields. 

Indeed, the general idea of normalization is based on starting with an object defined \lq\lq locally\rq\rq\ by an integral domain, and taking its integral closure in a field extension. In our cases, when we are close to Davenport's problem, the field is the function field of an algebraic curve that is a component of an algebraic object defined by a fiber product. 

 \cite{Cox05} attempts to define algebra, sufficiently widely to say how it arises where you might not regard it as naturally related to algebra.  His basic premise is that computations involve addition and multiplication, and sometimes division. That is, you work within a ring, and sometimes a field. Actual computations may limit manipulations by considering a finite set of elements which generate -- by computation -- all the others you use. If, then, you assume the multiplication is commutative -- he does not consider quantum mechanics, or Hopf algebras -- you are working in a polynomial ring. So, it is reasonable to say that such computations fall within algebraic geometry. 

{\sl Elimination theory\/}, a very old topic, was the forerunner of \cite{Cox05}. Until desktop computers, comparing {\sl your\/} mathematical objects with {\sl mine\/} by pure computation was difficult. Yet, that was the central topic of elimination theory.  

\subsection{Group theory in \cite{FrGS93}} \label{App.2} I could have phrased this appendix as a question: How 
could I  -- without formal training in groups -- have possibly understood (been confident of) the group theory in \cite[Part III]{FrGS93}? 

\S\ref{coGS} reminds of the essential results about exceptional polynomials, based on using the {\sl factorization\/} of a monodromy group into a product of a stabilizing group and the inertia group over $\infty$.  \cite[Part III]{FrGS93}  establishes a list of group properties of the Galois closure of $f$. These allow a characterization using the A(schbacher)-O('Nan)-S(cott) Classification  of primitive groups \cite{AOS85}. Excluding (primitive) affine groups, there are four primitive group types. Each is shaped by dropping almost simple groups into particular positions. Three points about this process call for clarification.

\begin{edesc} \label{primpoints} \item \label{primpointsa} Reduction to where $^aG_f$ is primitive (in its natural permutation representation; see \S\ref{coGS}). 
\item  \label{primpointsb}  Unlike the  \eql{equats}{equats2} version of Schur's Conjecture, if $(\deg(f),p)\ne 1$, no immediate version of  \eql{primpoints}{primpointsa} assures the geometric group, $G_f$, is primitive.
\item  \label{primpointsd}  \cite[Part III]{FrGS93} starts by clarifying the definitions in \cite{AOS85}. Then, this combines with  the appropriate factorizations of groups that arose from \cite{LPS}. The result is \eqref{excPoly}. 
\end{edesc} 

The most important addendum is to \eql{primpoints}{primpointsd}. I could not have completed this result alone. Also, rarely has academia found a formula for apportioning the significance and interpretation of such respective contributions. Finally, it was the unanticipated surprises in \eql{excPoly}{excPolyb} that got the attention of others. 

A statement due to Wan, that an exceptional polynomial should have degree prime to $q\nm 1$, was immediate from \cite{FrGS93} before Wan formulated his conjecture.  It wouldn't have occurred to the authors of \cite{FrGS93} to take that conjecture seriously, until we found that others mistakenly thought it meant that elementary methods had achieved our result. Wan's statement told little about exceptional polynomials, not even their degrees. By contrast, \cite{FrGS93} characterized much: Even in the one mystery, the precise monodromy groups in the affine case in \eql{primpoints}{primpointsd},  it has the degree of $f $ a power of the characteristic  (see http://math.uci.edu/paplist-ff/carlitz-quick.html). 

\subsection{What is a cover?} \label{whatCover} Grothendieck's definition of a cover of algebraic varieties is a finite, flat morphism $\phi: X\to Z$. We deal with varieties over a field $K$. Points on these spaces are geometric: with coordinates in some extension of $K$.  Components are defined over an algebraic closure $\bar K$. 

\subsubsection{Role of flatness} \label{roleFlatness} 
Finiteness of $\phi$ allows us to put a measure -- degree -- on the fibers of $X_z$, $z\in Z$, of $\phi$.  For irreducible $X$,  flatness says this degree is constant -- the degree of the function field extension $[K(X):K(Z)]$ -- in $z$ \cite[Prop.~2, p.~432]{Mum66}. For finite morphisms, that characterizes  flatness \cite[Cor.~p.~432]{Mum66}. 

It would simplify many things if we could restrict to unramified covers.  In characteristic 0 these come from topology: A finite index subgroup, $H$, of the fundamental group, $\pi_1(Z)$, produces up to equivalence of covers, an unramified cover $X_H\to Z$. The story, however, of monodromy precision, is exactly about going beyond this limitation, as noted in Thm.~\ref{Chebp}. 

The subtlety is that we use fiber products to mean,  after taking the standard fiber product, you normalize the result (\S\ref{secII.1}). If $\phi$ is finite and $X$ and $Z$ are nonsingular, then $\phi$ is automatically flat \cite[p.~266, 9.3a)]{Har77}. This doesn't extend to weakening nonsingular to normal varieties. \cite[p.~434]{Mum66} has a finite morphism, where $X$ is nonsingular (it is $\afA^2$), and $Z$ is normal, where the fiber degree is 2 over each $z\in Z$ excluding one point where it is 3. 

Suppose each of $\phi_i: X_i\to Z$, $i=1,2$, is a cover. Then the usual fiber product, denoted $X_1\times^\set_Z X_2$ in \S\ref{secII.1}, is also flat (therefore a cover) over $Z$. This follows from base change and transitivity of flatness \cite[p. 253, Prop.~9.1a]{Har77}. Yet, I don't know if the normalization, giving  $X_1\times_Z X_2\to Z$, is also. 

Therefore, \cite[\S1.1]{Fr05b} defines the nonsingular locus of $\phi$:  the complement of the (at least) co-dimension 2 union of the image of the singular locus of $X$and the singular locus of $Z$. It calls a finite morphism exceptional  if restricting $\phi$ over  the nonsingular locus -- the resulting morphism is  a cover -- is exceptional. 

There is a similar definition for Davenport pairs. This is conservative. It doesn't say what to expect over the singular locus, but it suffices for now. 

\begin{prob} \label{monpreciseext} Do the monodromy precision results of Davenport pairs, exceptionality, and more generally pr-exceptionality extend over the singular locus? \end{prob} 

\cite{GTuZ08} asserts an affirmative answer to Prob.~\ref{monpreciseext}  for exceptional covers.  \cite{Fr05b} says it should therefore hold for Davenport pairs, and pr-exceptionality. Their proof is exactly the same as that of \cite[Thm.~1]{Fr74b}, except they declare it works even over the singularity locus. 

\subsubsection{Fiber product universality} \label{fibproduniv} As in \S\ref{secII.1}, consider $X_1\times^\set_Z X_2$. As Grothendieck emphasized, it has the following universal property. Given $\phi_W: W\to Z$,  a finite morphism that factors through $\phi_i$, $i=1,2$, it factors through $X_1\times^\set_Z X_2$. 
\begin{trivl} \label{UnivPropGen} If we restrict our morphisms $\phi$ to normal varieties, then $\phi$ factors through the normalization $X_1\times_Z X_2$ of $X_1\times^\set_Z X_2$. 
\end{trivl}

Certain properties of covers come purely from group theory, using the Galois correspondence between subgroups of the monodromy group and quotients of the Galois closure cover. An example is the see-saw correspondence that produced \cite[Prop.~2]{Fr73a} as in Lem.~\ref{pullbackcomps}, especially \eqref{schcond}. It has nothing to do with the covers being genus 0 curves, or  that they cover $\prP^1_z$ or even that they have dimension 1. I did Lem.~\ref{indecompLem} as an example to show how generally it works. 

The use of  Riemann-Hurwitz is just for curves. Using Abyhankar's Lemma in \eqref{abhyLem} is purely local from tame ramification. So,  assume the fiber product of $f: X\to \prP^1_z$ and $g: Y\to \prP^1_z$ is irreducible. More generally replace $\prP^1_z$ by $Z$. Then, to compute the genus of the fiber product use this (well-known) generalization of \eqref{*10} for R-H with $\geng_Z$ denoting the genus of $Z$:  
\begin{equation} \label{RHGen}     2(\deg(f) + \geng_f - 1) =2\deg(f)\geng_Z+ \sum_{i=1}^r \ind(\sigma_i). \end{equation} 

\section{Classical Generators and Definition Fields}  

\subsection{Classical Generators} \label{classgens} 

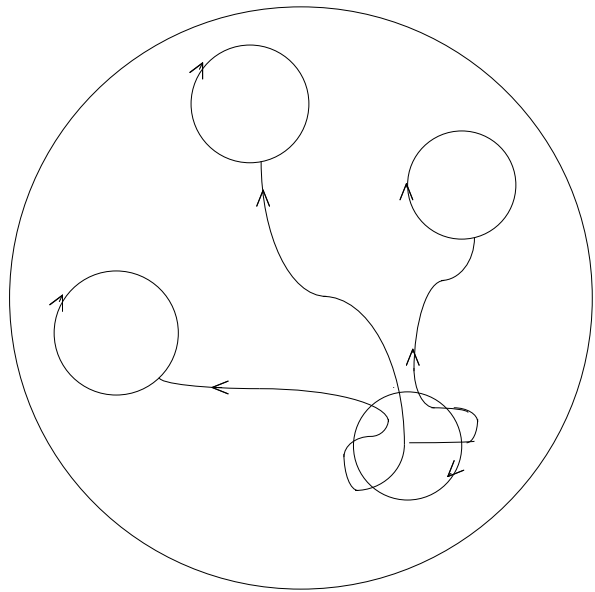
\begin{figure}
\caption{Example Classical Generators }


$$\put(98, 201){$\scriptscriptstyle\bullet$}
\put(85, 175){$\scriptscriptstyle\bullet$}
\put(89, 192){$\scriptscriptstyle\bullet$}
\hbox{\centerline{\includegraphics[scale=1.2]{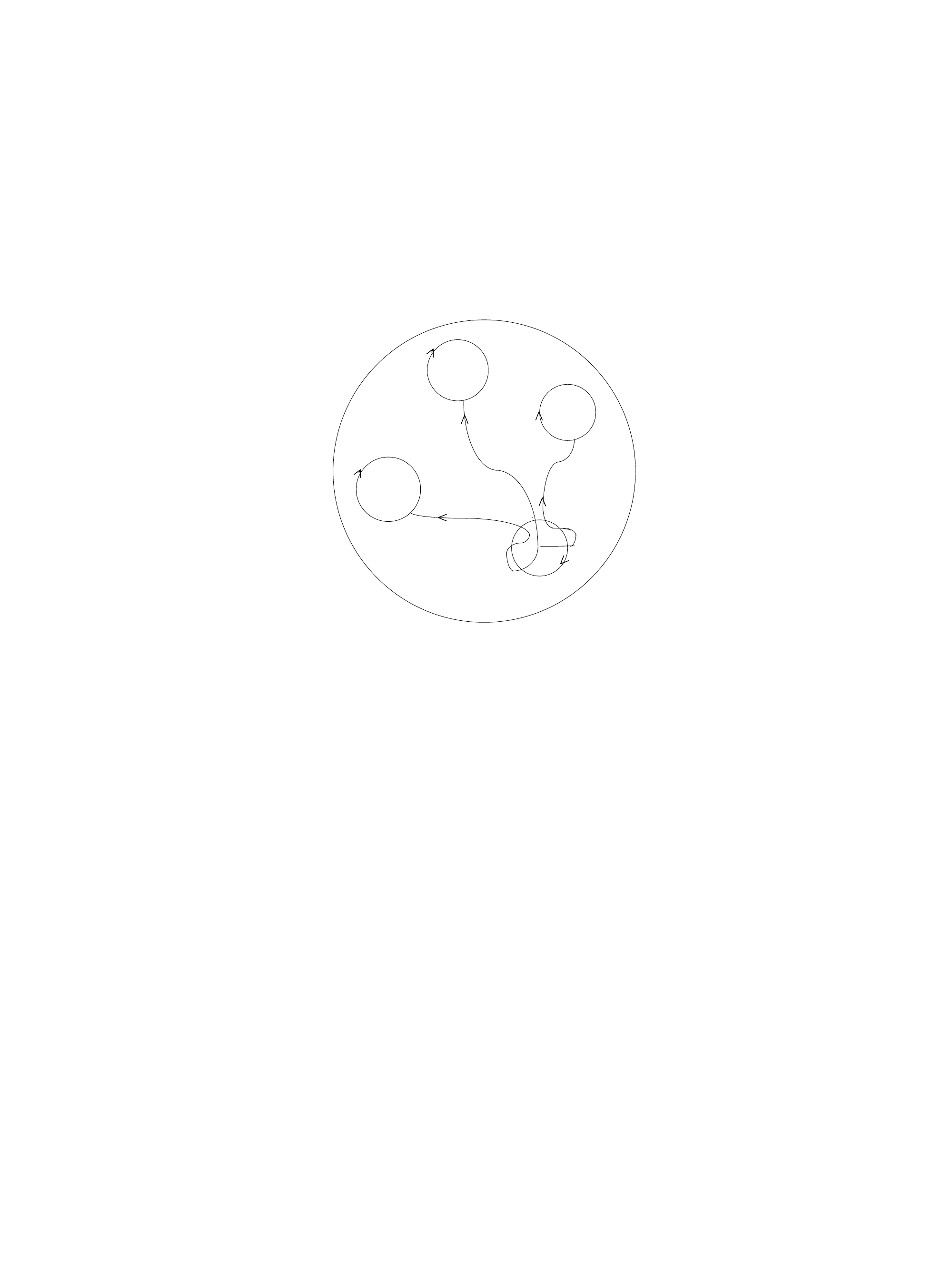}}}
\put(-171, 232){$\scriptscriptstyle\bullet$}
\put(-158, 232) {$\scriptscriptstyle\bullet$}
\put(-142, 225){$\scriptscriptstyle\bullet$}
\put(-154,93){$z_0$}  
\put(-252, 132){$z_1$} 
\put(-200, 221){$z_i$}  
\put(-122, 192){$z_r$} 
\put(-145,93){$\scriptscriptstyle\bullet$}
\put(-255, 132){$\scriptscriptstyle\bullet$}
\put(-203,220){$\scriptscriptstyle\bullet$}
\put(-125, 192){$\scriptscriptstyle\bullet$}
\setbox0=\hbox{$\searrow$} \put(-153, 212){\raise \ht0\hbox
{$\sigma_r^*$}$\searrow$ }  
\setbox0=\hbox{$\nearrow$} \put(-144,84){\lower \ht0\hbox
{$\sigma_0^*$}$\nearrow$ } 
\setbox0=\hbox{$\nearrow$} 
\put(-246, 208){$\sigma_i^*\nearrow$} 
\setbox0=\hbox{$\searrow$} \put(-291, 154){\raise \ht0\hbox
{$\sigma_1^*$}$\searrow$ }  
\put(-202, 107){\lower \ht0\hbox
{$\delta_1$}$\nearrow$ } 
\setbox0=\hbox{$\nearrow$} \put(-215,180){\lower \ht0\hbox
{$\delta_i$}$\nearrow$}   
\setbox0=\hbox{$\leftarrow$} \put(-120, 160){$\leftarrow$\raise \ht0\hbox
{$\delta_r$} } 
\setbox0=\hbox{$\nearrow$} \put(-256,111){\lower \ht0\hbox
{$b_1$}$\nearrow$ } 
\setbox0=\hbox{$\nearrow$} \put(-217, 193){\lower \ht0\hbox
{$b_i$}$\nearrow$ } 
\setbox0=\hbox{$\leftarrow$} \put(-117, 169){$\leftarrow$\raise \ht0\hbox
{$b_r$} }
\put(-239, 115){$\scriptscriptstyle\bullet$}
\put(-201, 201){$\scriptscriptstyle\bullet$}
\put(-119, 169){$\scriptscriptstyle\bullet$}
\setbox0=\hbox{$\nearrow$} 
\put(-178,70){\lower \ht0\hbox{$a_1$}$\nearrow$ }  
\setbox0=\hbox{$\searrow$} 
\put(-166, 116){\raise \ht0\hbox{$a_i$}$\searrow$ }  
\setbox0=\hbox{$\leftarrow$} 
\put(-123,92){$\leftarrow$\raise \ht0\hbox{$a_r$} } 
\put(-159,74){$\scriptscriptstyle\bullet$}
\put(-148, 113){$\scriptscriptstyle\bullet$}
\put(-123,92){$\scriptscriptstyle\bullet$} $$

\end{figure}

Figure 1 explains  {\sl classical generators\/} of the fundamental group, $\pi_1(U_\bz,z_0)$, of the  $r$-punctured sphere, with the punctures given by $\bz=\{\row z r\}$. These are ordered  closed
paths $\delta_i^\sph{\sigma^*_i}^\sph\delta_i^{-1} = \bar\sigma_i$, $i = 1,\dots ,r$.  

Here are their properties. There are discs, $i = 1,\dots ,r$:   $D_i$   with center $z_i$; all disjoint, each excludes $z_0$; 
$b_i$ is  on the boundary of $D_i$.  Their clockwise orientation refers to the  boundary of $D_i$. The path $\sigma^*_i$ has initial and end
point $b_i$;  $\delta_i$ is a simple {\sl simplicial}
path with initial point $z_0$ and end point $b_i$. We also assume 
$\delta_i$ meets none of  $\sigma_1^*,\dots ,\sigma^*_{i-1},\sigma^*_{i+1},\dots 
,\sigma^*_r$,
and it meets
$\sigma^*_i$ only at its endpoint.

There is a crucial condition on meeting the boundary of $D_0$.    First: $D_0$, with center $z_0$, is disjoint from each  $D_1,\dots ,D_r$.  Consider
$a_i$, the first intersection of $\delta_i$ and boundary $\sigma^*_0$ of
$D_0$.  Then,   $\delta_1,\dots
,\delta_r$ satisfy these conditions:  
\begin{edesc} \label{pathcond} \item  \label{pathconda} they are pairwise
nonintersecting, except at $z_0$; and
\item  \label{pathcondb} $a_1,\dots, a_r$ are in order clockwise around $\sigma^*_0$.  
\end{edesc} 

Since the paths are simplicial, \eql{pathcond}{pathconda}  is
independent of  $D_0$, for $D_0$ sufficiently
small.  For any ordering of the collection  $\bz$, many sets of classical generators have orderings corresponding to the order of $\bz$. That means, given branch cycles of a cover there will be several branch cycles descriptions -- up to, say, absolute equivalence -- corresponding to a given cover of $\prP^1_z$ branched over $\bz$. 

\subsection{Hurwitz space definition fields} \label{abs-innTie}  Davenport's problem distinguishes between a cover and its Galois closure. That subtlety  culminates in Thm.~\ref{vectorBundle} distinguishing between parametrizing  Davenport pairs and their Galois closures. While the same space parametrizes both, I will explain the distinction. 

I'll also mend an oversight in \cite[Main Thm.]{FrV91}. In comparing inner and absolute Hurwitz spaces, it didn't appropriately -- as the subject started from the absolute case \cite[Thm.~5.1]{Fr77} -- put their definition fields on the same footing.   

\subsubsection{Inner Hurwitz spaces} \label{innerHur} The space $\sH(G,\bfC)^{\inn}$ parametrizes {\sl  inner equivalence classes\/} of Galois covers $\hat \phi:\hat X\to \prP^1_z$ in the Nielsen class $\ni(G,\bfC)$. Let $(\hat \phi_i,\hat X_i)$ be such covers with an explicit identification $\mu_i$ of $\Aut(\hat X_i/\prP^1_z)$  with $G$, $i=1,2$.   
\begin{defn}  \label{innereq} We say $(\hat \phi_i,\mu_i)$, $i=1,2$, are {\sl inner equivalent\/} if there is a continuous $\psi: \hat X_1\to \hat X_2$,  commuting with $\hat \phi_i$, $i=1,2$, with   $\mu_1\circ \psi^*\circ (\mu_2)^{-1}$ an inner automorphism of $G$. \end{defn} 
 
 Consider this inner analog of expression \eqref{multgps}:    
\begin{equation} \label{innermultgps} \hat M_{\bfC}=\{c\in (\bZ/N_\bfC)^*\mid \exists \beta \in S_r, \C_i^c=\C_{(i)\beta}, i=1,\dots,r\}.\end{equation}  This, too, defines a cyclotomic field, the fixed field of $\hat M_\bfC$ in $\bQ(e^{2\pi i/N_\bfC})$:   $\bQ_{\hat M_\bfC}$.

Given an absolute Nielsen class, \cite[Main.~Thm.]{FrV91} gives three results, using the inner Hurwitz space of a Nielsen class, $\sH(G,\bfC)^{\inn}$. 
\begin{edesc} \label{FrV} \item  \label{FrVa}  There is a natural map $\Psi^{\inn,\abs}: \sH(G,\bfC)^{\inn} \to \sH(G,\bfC)^\abs$:   the class of $\hat \phi: \hat X\to\prP^1_z$ maps to the class of $\phi: \hat X/G(1)\to \prP^1_z$. 
\item  \label{FrVb} The definition field of $(\Psi^{\inn,\abs},\sH(G,\bfC)^\inn)$ is precisely $\bQ_{\hat M_\bfC}$. 
\item  \label{FrVc} Restricting $\Psi^{\inn,\abs}$ to a connected component $\sH'$ of $\sH(G,\bfC)^{\inn}$ gives a Galois cover $\sH'\to \Psi^{\inn,\abs}(\sH')$. Its group is $H\eqdef H_{\sH'}\le  N_{S_n}(G,\bfC)/G$. \end{edesc} 

\begin{proof}[Explaining \eql{FrV}{FrVb}]  
A more precise statement would start: \lq As a moduli space.\rq\ It means  consider the collection of families, $F\in \sF_{G,\bfC}^{*}$, of covers in the Nielsen class $\ni(G,\bfC)^\inn$ defined over $\bar \bQ$. (* is again inner or absolute equivalence.) 

Then, $\gamma\in G_\bQ$ acts on the elements of $F\in \sF_{G,\bfC}^{\inn}$ (through equation coefficients): $F\mapsto  F^\gamma$, giving another family of covers.  \cite[\S4]{Fr77} shows that every cover -- up to equivalence -- in a given Nielsen class appears in a family of covers defined over $\bar\bQ$ parametrized by a finite cover of a Zariski open subset of $U_r$. Further, $F^\gamma$ is in the Nielsen class defined by $(G,\bfC^{c_\gamma})$ with $c_\gamma$  as in \S\ref{secVI.1}. 

Therefore, the collection is fixed under $\gamma$ if and only if the resulting Nielsen class under the equivalence class * is the same as that given by $(G,\bfC)$. That means the respective $\gamma\,$s that fix the families defined by $\abs$ (resp.~$\inn$) equivalence appear from the equation  \eqref{multgps} (resp.~\eqref{innermultgps}). \end{proof} 

The notation of \eql{FrV}{FrVc} indicates that the Galois group of an inner component over an absolute component can vary with the component.  

\subsubsection{Braid Components vs braidable  automorphisms}  \label{missmultiplier} 
\cite[Lem.~3.8]{BiFr82} says, for $h\in G$, there is a  $q\in H_r$ with $(\bg)q=h\bg h^{-1}$: inner automorphisms are  {\sl braidable\/}. Yet, an  $h\in N_{S_n}(G,\bfC)$  may not be (see Ex.~\ref{n=1mod8}).  This is one reason an absolute Hurwitz space may have smaller definition field than its corresponding inner space. 

Denote the braid orbit on $\sH(G,\bfC)^{\inn,\rd}$ corresponding to $\sH'$ by $O'$.  
\begin{defn} \label{braiduatomorphisms}  Define the set of braidable outer automorphisms as follows:  
$$N_{S_n}^\br(G,\bfC)_{O'}  \eqdef\{h \in N_{S_n}(G,\bfC) \mid \exists\ q\in H_r \text{ with  } (\bg)q=h\bg h^{-1}\}.$$ From the above remark, this group contains $G$. \end{defn}

\begin{lem} Per notation, $N_{S_n}^\br(G,\bfC)_{O'}$  depends only on $O'$.  Also, the geometric automorphism group of the cover $\sH'\to \Psi^{\inn,abs}(\sH')$ identifies with $N_{S_n}^\br(G,\bfC)_{O'}/G$. \end{lem} 
 
 \begin{proof} First consider $\bg\in O'$ and $(\bg)q'=\bg'$. Assume $h(\bg)h^{-1}=(\bg)q^*$ for some $q^*\in H_r$. Since conjugation by $h$ and application of $q'$ commute,  
 $$ (h\bg h^{-1})q'=h(\bg')h^{-1}=((\bg)q^*)q'=(\bg')(q')^{-1}q^*q'.$$ That proves the first sentence. 
 
 Now consider the 2nd sentence. The fiber of $\Psi_{\inn,\abs}$ is in one-one correspondence with the elements of $N_{S_n}(G,\bfC)/G$. So, if we restrict to the connected component $\sH'$, the fiber restricts to the action of elements of $N_{S_n}(G,\bfC)$ that are braidable. \end{proof} 
 
 Expression \eqref{multgps} defines $M_{\bfC}$. Here is the generalization of that:  
\begin{equation} \label{orbitgps} \begin{array}{rl}
M_{O'}=\{c\in (\bZ/N_\bfC)^* &\mid \exists \beta \in S_r, h\in N_{S_n} (G,\bfC)^\br_{O'},  \\ 
&h^{-1}\C_i^c h=\C_{(i)\beta}, i=1,\dots,r\}.\end{array} \end{equation} 

The argument explaining \eql{FrV}{FrVb} gives the following. 
\begin{prop} With $H$ and other notation as above, consider the collection $J_H$ of components $\sH'$ (with their maps to $U_r$) with the group of  $\sH'\to \Psi^{\inn,abs}(\sH')$ equal to a subgroup of $N_{S_n}(G,\bfC)/G$ isomorphic to $H$. Then, the collection $J_H$ has  definition field the fixed field in $\bQ(e^{2\pi i/N_{\bfC}})$ of $M_{O'}$. \end{prop} 

Two techniques have located examples of multiple Hurwitz space components: \begin{edesc}  \label{hurcomps} \item \label{hurcompsa} the Fried-Serre Lifting invariant as in \cite[Part II]{Fr95b} and \cite{Se90}; and \item \label{hurcompsb}  unbraidable outer automorphisms as  above (\cite[\S3]{BiFr82} is the first).  
 \end{edesc} 
 If the lifting invariant precisely delineates the components, then -- generalizing the original \cite[Thm.~5.1]{Fr77} result -- the definition fields of those components are known cyclotomic fields. \cite[Main Thm.]{Fr10} uses 3-cycle Nielsen classes to illustrate how effectively  \eql{hurcomps}{hurcompsa}, based on Frattini central extensions (\S\ref{App.1a}; and their kernels,  quotients of Schur multipliers) detects components. Our approach to Schur multipliers (developed along with Modular Towers) has simplified how they appear, removing the intimidating group theory that once accompanied them. 
 
 If unbraidable outer automorphisms precisely delineate the components, then the story is rougher. Still, among the many known examples, the only mysteries  for definition fields are the two described in \cite[\S 9.1]{BFr02}. Each has components whose descriptions come from both types of \eqref{hurcomps}.  Two $j$-line covers of genus 1 are conjugate by  an unbraidable outer automorphism. A particular Inverse Galois conclusion depends on whether they have a nontorsion $\bQ$ point, and this  depends on whether their definition field is $\bQ$ or a quadratic extension of $\bQ$. 
 
All examples we know that have multiple Hurwitz space components can be ascribed to some combination of the limitations posed by the conditions \eqref{hurcomps}. 
 
 \begin{exmp} \label{n=1mod8}  \cite[Ex.~1.5]{Fr10} has the example of the Nielsen class $\ni(A_n,\bfC_{\frac{n\np1}2}^4)^*$, $n\equiv 1 \mod 8$, four repetitions of $\frac{n\np1}2$-cycles.  For $*=\inn$ there are two braid orbits, corresponding to not being able to braid the outer automorphism of $A_n$.  The corresponding  Hurwitz space components have definition field a quadratic extension of $\bQ$.  There is just one absolute component. For $n\equiv 5 \mod 8$ there is just one braid orbit for both absolute and inner classes. \end{exmp}

 \subsubsection{Little use of branch cycles} \label{nobcs}  Why have so few papers that quote \cite{Fr73a}, and related papers, used branch cycles? (A notable exception is M\"uller, say, in \cite{Mu96} and \cite{GMu97}.) Maybe it was the confluence of three historical events that affected all of mathematics, in addition to the lack of training on these topics.  
 
First: In the early 80's  libraries massively moved many journals to archives. This was to make way for the generation of new journal/society generated publicatons. Mathematics, where positions were rapidly disappearing lost heavily in the politics of that process. This deserves further attention, but where is do such topics have a natural publishing venue? It seems the only convenient means to find many of my papers before 1985 (including \cite{Fr73a} and \cite{Fr77}), and even some afterwards, is from their scanning on my web site. 

Occasional  pdf files from journal web sites (\cite{Fr76}, say) are unsearchable, while mine are mostly now. I've used html expositions to improve access -- even beyond searchability -- to what has turned out most significant. 
\LaTeX\ generated pdf's are theoretically searchable. Still,  I've yet to see that turned into minable data, much less a linked database. So far it looks as if html is easier that way. 

Second: I've noted many examples of the following in this paper.  Refereeing is nowhere near the quality to indicate community awareness of what was proved previously, nor what has a history of relating to ongoing research. An author who wants credit for significant results -- according to what it adds to existing literature --  needs hooks to their work. Then, they need ways to get others to use those hooks. This last is too hard right now for those without high prestige connections.  

I don't agree it is the sole responsibility of the author to assure results are correct. That would  mean the author is the most aware of the area's pitfalls, and has no hidden or psychological reasons to mentally avoid subtle points. I've said how wrong this is in public places \cite{Fr07}. 
I note that mathematics is hardly alone in the neglect of its works. No less than Doris Lessing, she of  \lq\lq The Golden Notebook\rq\rq\ fame, has seen it from a far perspective: \lq\lq The shame of the 20th century will be all the research that is left unread on the shelves.\rq\rq

Third: within algebraic geometry, there was a prevailing attitude in the 60's and '70s that it was now time to diminish moduli of curves for the sake of moduli of higher dimensional objects. While number theory wasn't ready for any such move, the field of arithmetic geometry was not well-defined. It  still suffered from sorting those who used vs those who railed against,  Grothendieck's techniques.  

Mumford's research topics were much into curves and their Jacobians (as in \cite{Mum76}), but neither \cite{Har77} nor \cite{Mum66} touched coverings or group theory and certainly not their moduli. Also, they worked entirely over an algebraically closed field, without any profinite aspects, when they didn't emphasize schemes. For example, you would find it difficult even now to place the Branch Cycle Lemma within either book.  \cite{Se92} doesn't have it despite its clear relevance, though its review discussed and used it  \cite[\S3 and \S7]{Fr94b}.  This, too needs a thoughtful perspective, if it is to be available. 

\end{appendix}


\providecommand{\bysame}{\leavevmode\hbox to3em{\hrulefill}\thinspace}
\providecommand{\MR}{\relax\ifhmode\unskip\space\fi MR }
\providecommand{\MRhref}[2]{%
   \href{http://www.ams.org/mathscinet-getitem?mr=#1}{#2}
}
\providecommand{\href}[2]{#2}



\begin{thebibliography}{MFS93}

\bibitem[Abh57]{Abh57} S. Abhyankar, \emph{Coverings of Algebraic Curves}, Am. J. Math \textbf{79} (1957), 825--856.

\bibitem[Abh97]{AbProjPol}
\bysame, \emph{Projective polynomials}, Proc AMS \textbf{125} (1997),
   1643--1650.

\bibitem[Ag08]{Ag08} Ilka Agricola, \emph{Old and New on the Exceptional Group $G_2$}, Notice of the AMS, Vol.~\textbf{55} No.~8, 922--929.

\bibitem[Ah79]{Ah79} L.~Ahlfors, \emph{Introduction to the Theory of Analytic Functions of One}, 3rd edition, Inter. Series in Pure and Applied Math., McGraw-Hill Complex Variable, 1979.

\bibitem[Ait98]{Aitken}
W.~Aitken, \emph{On value sets of polynomials over a finite field},
Finite Fields Appl. \textbf{4} (1998),
   441--449.
 
\bibitem[Art23]{Artin}
E.~Artin, \emph{\"Uber die Zetafuncktionen gewisser algebraischer
Zahlk\"orper}, Math.~Ann. 89 (1923), 147--156.

\bibitem[Ar57]{Ar57} \bysame, \emph{Geometric Algebra}, Inter.~tracts in Pure and App. Math. \textbf{3}, 1957.

\bibitem[AOS85]{AOS85} M. Aschbacher and L. Scott, Maximal subgroups of finite groups, J. Algebra 92 
(1985), 44--80.

\bibitem[AZ01]{AZ01} R.M.~Avanzi, and U.M.~Zannier, \emph{Genus one curves defined by separated variable polynomials and a polynomial Pell equation}, Acta Arith. \textbf{99} (2001), 227--256.

\bibitem[AZ03]{AZ03} \bysame, \emph{The Equation $f(X)=f(Y)$  in Rational Functions $X =X(t)$, $Y =Y(t)$},  Comp.~Math., Kluwer Acad.  \textbf{139} (2003),  263--295. 

\bibitem[Ax68]{Ax68} J.~Ax, \emph{The elementary theory of finite fields}, Annals of Math.~\textbf{88} (1968), 239--271. 

\bibitem[Ax71]{Ax71} \bysame, \emph{A mathematical approach to some problems in number theory}, in AMS Proc.~Symp.~in Pure Math.~\textbf{20} (1971), 1969 Inst.~on No.~Th.~ at Stony Brook, 161--190. 

\bibitem[AxKo66]{AxKo66} J.~Ax and S.~Kochen, \emph{Diophantine problems over local fields III}, (culminating paper of the series), Annals of Math.~\textbf{83} (1966), 437--456.
 
\bibitem[BFr02]{BFr02} P.~Bailey and M.~D.~Fried, \emph{Hurwitz monodromy, spin separation and
higher levels of a Modular Tower}, in  Proc.~of Symp.~in Pure Math. {\bf
70} (2002) eds M.~Fried and Y.~Ihara, \!1999 \!von \!Neumann
Symp., Aug.~16-27,
1999 MSRI, 79--221. Typos are corrected in arXiv:math.NT/0104289 v2 16 Jun 2005. 

\bibitem[BNg06]{BNg06} A.~Beardon and T.~Ng, \emph{Parameterizations of algebraic curves}, Ann. Acad. Sci. Fenn., Math. \textbf{31}, No. 2 (2006), 541--554. 
 
 
\bibitem[Be11]{Be11}T.~Beke, \emph{Zeta functions of equivalence relations over finite fields},  Finite Fields and Their Applications, Vol.~ \textbf{17}, Issue 1, January 2011, 68--80. 

\bibitem[BeShTi99]{BeShTi99} F.~Beukers, T.N.~Shorey and R.~Tijdeman, \emph{Irreducibility of polynomials and arithmetic progressions with equal products of terms}, No.~Th. in Prog.~(Berlin-New York) (ed. J.~Urbanowicz K.~Gyory, H.~Iwaniec,
   ed.), Walter de Gruyter, 1999, Proc. of the Schinzel Festschrift,
   Summer 1997: 11--27. 
   
\bibitem[BiFr82]{BiFr82} R.~Biggers and M.~Fried, \emph{Moduli spaces of covers and the Hurwitz monodromy group}, CrelleÕs Journal 335 (1982), 87Ð121.

\bibitem[BiFr86]{BiFr86} \bysame, \emph{Irreducibility of moduli spaces of cyclic unramified covers of genus $g$ curves}, TAMS Vol.~\textbf{295} (1986), 59--70. 

\bibitem[B99]{B99} Y.F.~Bilu, \emph{Quadratic factors of $f(x)-g(y)$}, Acta Arith. \textbf{90} (1999), 341--355.

\bibitem[BT00]{BT00}  Y.F.~Bilu and R.F.~Tichy, \emph{The diophantine equation $f(x)-g(y)$}, Acta Arith.
\textbf{95} (2000), 261--288.

\bibitem[Bl04]{Bl04} A.~Bluher, \emph{Explicit formulas for strong Davenport pairs}, Act.Arith. \textbf{112.4} (2004),397--403.

\bibitem[Bm78]{Bm78} E.~Bombieri, \emph{On exponential sum in finite fields II}, Inv.~math.~\textbf{47} (1978), 20--39. 

\bibitem[BoSh66]{BoSh66} Z.I.~Borevich and I.R.~Shafarevich, translated by Newcomb Greenleaf, \emph{Number Theory}, Academic Press, 1966.

\bibitem[Ca37]{Ca37} R. Carmichael, Introduction to the Theory of Groups of Finite Order, Dover Publications, 
1956 edition (first published 1937).

\bibitem[CaFr67]{CaFr67} J.~Cassels and A.~Fr\"ohlich, \emph{Algebraic Number Theory}, Thompson Book Co., Wash.~D.C., 1967.

\bibitem[ClZ10]{ClZ10} C.~Fuchs and U.~Zannier, \emph{Composite rational functions expressible with few terms}, preprint as of March 2010. 

\bibitem[C90]{C90} S.D.~Cohen, \emph{Exceptional polynomials and the reducibility of substitution polynomials}, LÕEnseigment Math.~\textbf{36} (1990), 309--318.

\bibitem[CFr95]{CFr95} S.D.~Cohen and M.D.~Fried,
\emph{Lenstra's proof of the Carlitz-Wan conjecture on exceptional polynomials:  an elementary 
version}, Finite Fields Appl.~\textbf{1} (1995), 372--375.

\bibitem[Con78]{Con78} J.B.~Conway, \emph{Functions of a complex variable, 2nd Edition}, Springer-Verlag Grad.~text, 1978. 


\bibitem[CM94]{CM94} S.D.~Cohen and R.W.~Matthews, \emph{A class of exceptional polynomials}, TAMS \textbf{345} (1994), 897--909. 

\bibitem[CGen]{CGen} Paths that are classical generators of the punctured sphere: \\ 
http://math.uci.edu/\~{\ \!\!}mfried/deflist-cov/classicalgens.pdf.  The genus 0 problem for rational functions: http://math.uci.edu/\~{\ \!\!}mfried/deflist-cov/Genus0-Prob.html

\bibitem[Cox05]{Cox05} D.~Cox, \emph{What Is the Role of Algebra in Applied Mathematics?}, Nov.~2005 Notices of the AMS, 1193--1198. 

\bibitem[RET3]{chpfund.pdf} RET Chap. 3: {\sl Complex Manifolds and Covers}: Introduces coordinates on a Riemann surface, and sufficient algebraic geometry to consider manifold compactifications of common Riemann surfaces. Aims directly at introducing Riemann's favorite subject -- necessary for his solution to the Jacobi Inversion Problem -- half-canonical classes. The detail on covering spaces, Galois covers and flat bundles goes beyond what is usual for a truly graduate level book. 

\bibitem[RET4]{chpret4-firsthalf.pdf} RET Chap. 4: {\sl Riemann's Existence Theorem}:  The proof, combinatorics of its use (including Braid and Hurwitz monodromy group manipulations), and the algebra of coordinates attached to Riemann's Existence Theorem. We give a non-traditional approach to Abel's Theorem for genus 1 curves. This treatment of the $j(\tau)$ and 
$\lambda(\tau)$ functions and modular curves of complex variables motivates Chap. 5: {\sl Hurwitz monodromy and the development of Modular Towers}. 

\bibitem[CoCa99]{CoCa99} J.-M. Couveignes and P. Cassou-Nogus, Factorisations explicites de $g(y)
-h(z)$, Acta Arith. 87 (1999), no. 4, 291--317.

\bibitem[CKS76]{CKS76} C.W. Curtis, W.M. Kantor and G.M. Seitz, The 2-transitive permutation 
representations of the finite Chevalley groups, TAMS 218 (1976), 1--59.

\bibitem[DLSc61]{DLSc61} H.~Davenport, D.J.~Lewis and A.~Schinzel \emph{Equations of Form $f(x)=g(y)$}, Quart. J. Math. Oxford (2) 
\textbf{12} (1961), 304--312.

\bibitem[DL63]{DL63} H.~Davenport and D.J.~Lewis, \emph{Notes on Congruences (I)}, Quart. J. Math. Oxford (2) 
\textbf{14} (1963), 51--60.

\bibitem[De99]{De99} P.~D\`ebes, \emph{Arithm\'etique et espaces de modules de rev\^evetements}, No.~Th. in Prog.~(Berlin-New York) (ed. J.~Urbanowicz K.~Gyory, H.~Iwaniec,
   ed.), Walter de Gruyter, 1999, Proc. of the Schinzel Festschrift, 75--102.
   
\bibitem[De09]{De09} \bysame, \emph{Arithm\'etique des rev\^etements de la droite}. chapitres 1-8, 275 pages, 2009 at http://math.univ-lille1.fr/\~\  \!\!de/pub.html (volume no 2). 

\bibitem[DeFr90a]{DeFr90} P.~D\`ebes and M.~Fried, \emph{Rigidity and real residue class  fields}, Acta.~Arith. \textbf{56} (1990), 13--45.

\bibitem[DeFr90b]{DeFr90b} P.~D\`ebes and M.D.~Fried, \emph{Arithmetic variation of fibers in  families: Hurwitz  monodromy criteria for rational points on all  members of the family},  Crelles J. \textbf{409} (1990),  106--137. 

\bibitem[DeFr94]{DeFr94} \bysame, \emph{Nonrigid situations in constructive Galois theory}, PJM \textbf{163} (1994), 81--122. 

\bibitem[DeFr99]{DeFr99} \bysame, \emph{Integral Specialization of families of rational functions}, PJM \textbf{190}, 1999, 75--103. 

\bibitem[DelMu67]{DelMu67} P.~Deligne and D.~Mumford, \emph{The irreducibility of the space of curves of given genus}, IHES No.~\textbf{36}, 75--100. 

\bibitem[Del74]{De74} P.~Deligne, \emph{La conjecture de Weil I}, Publ. Math. IHES \textbf{43} (1974), 273--307. 

\bibitem[Del80]{De80} \bysame, \emph{La conjecture de Weil: II},  Publ. Math. IHES \textbf{52} (1980), 137Ð252. 

\bibitem[Del89]{Del89} \bysame, \emph{Le Groupe fondamental de la Droite Projective Moins Trois Points}, in {\bf Galois Groups over $\bQ$}, MSRI publications \textbf{16}, Springer-Verlag, 79--297. 

\bibitem[Den84]{Den84} J.~Denef, \emph{The rationality of the Poincar\'e series
associated to the $p$-adic points on a variety}, Invent. Math.~\textbf{77} (1984), 1--23.

\bibitem[DeLo01]{DeLo01} J.~Denef and F.~Loeser, \emph{Definable sets, motives and $p$-adic integrals},
JAMS \textbf{14} (2001), 429--469.

\bibitem[Dw66]{Dw66} B.~Dwork, \emph{On the zeta function of a hypersurface III}, Annals of Math.~\textbf{83} (1966), 457--519. 

\bibitem[Ev03]{Ev03}  J.-H.~Evertse \emph{Linear equations with unknowns from a multiplicative group whose solutions lie in a small number of subspaces},  http://front.math.ucdavis.edu/ANT, 11 Dec 2003 Paper: math.NT/0312235.

\bibitem[Fe70]{Fe70} W. Feit, Automorphisms of symmetric balanced incomplete block designs, Math. 
Zeit. 118: (1970), 40--49.

\bibitem[Fe73]{Fe73} \bysame, Automorphisms of symmetric balanced incomplete block designs with 
doubly transitive automorphism groups, J. of Comb. Th. (A) 14: (1973), 221--247.

\bibitem[Fe80]{Fe80} \bysame, Some consequences of the classification of the finite simple groups, Proc. 
of Symp. in Pure Math. 37 (1980), 175--181.

\bibitem[Fr70]{Fr70} M.D.~Fried, On a conjecture of Schur, Mich. Math. J. 17 (1970), 41--55.

\bibitem[Fr73a]{Fr73a} \bysame, \emph{The field of definition of function fields and a problem in the reducibility of 
polynomials in two variables}, Ill.~J.~Math. \textbf{17} (1973), 128--146. Comment: The editors put in, Received 
May 13, 1969; received in revised form March 8, 1972. The revisions consisted of documenting that two 
of my papers had finally been accepted. As the introduction says: "The results of this paper were 
obtained during the academic year 1968--1969 [a preliminary draft from spring and summer '68]. Delay in publication corresponds to 
delay in publication of the applications (for which we'd like to thank the editors and referees of several 
journals). Item [12] in the bibliography is M. Fried and D. Lewis, Solution spaces to Diophantine 
problems, \dots, a response to Lewis' request that I write up an expansion of topics he discussed from this 
paper for his invited AMS hour talk. This never appeared, but the topics were in my 
opening research lecture at A.~Schinzel's 60th birthday celebration in Zakopane ([Fr99] below).

\bibitem[Fr73b]{Fr73b} \bysame, \emph{A theorem of Ritt and related diophantine problems}, Crelles J. \textbf{264}, (1973), 40--55.

\bibitem[Fr74a]{Fr74a} \bysame, \emph{On HilbertÕs irreducibility theorem}, JNT \textbf{6} (1974), 211--232.


\bibitem[Fr74b]{Fr74b} \bysame, \emph{On a theorem of
MacCluer}, Acta Arith.~\textbf{XXV} (1974), 122--127.

\bibitem[Fr74c]{Fr74c} \bysame, \emph{Arithmetical properties of function fields (II): The generalized Schur problem:}, Acta. Arith. XXV (1974), 225--258.


\bibitem[Fr77]{Fr77} \bysame, \emph{Fields of Definition of Function Fields and Hurwitz Families and; Groups as 
Galois Groups}, Communications in Algebra \textbf{5} (1977), 17--82.

\bibitem[Fr78]{Fr78} \bysame, \emph{Galois groups and Complex Multiplication}, Trans.A.M.S. \textbf{235} (1978), 141--162.

\bibitem[Fr80]{Fr80} \bysame, \emph{Exposition on an Arithmetic-Group Theoretic Connection via RiemannÕs 
Existence Theorem}, Proceedings of Symposia in Pure Math: Santa Cruz Conference on Finite Groups, 
A.M.S. Publications \textbf{37} (1980), 571--601.

\bibitem[Fr86]{Fr86} \bysame, \emph{$L$-series on a Galois Stratification}, notes from Lecturing at Yale in Spring 1978, accepted by J.~No. Theory in 1986 (http://www.math.uci.edu/~mfried/paplist-ff/LSeriesGalSt86.pdf). 

\bibitem[Fr87]{Fr87} \bysame, \emph{Irreducibility results for separated variables equations}, Journal of Pure and 
Applied Algebra \textbf{48} (1987), 9--22.

\bibitem[Fr90]{Fr90} \bysame, \emph{Arithmetic of 3 and 4 branch point covers: a bridge provided by noncongruence subgroups of $\SL_2(\bZ)$} Progress in Math. Birkhauser \textbf{81} (1990), 77--117. 

\bibitem[Fr94a]{Fr94} \bysame, \emph{Global construction of general 
exceptional covers,  with motivation for applications to  coding}, 
G.L.~Mullen an P.J.~Shiue,   Finite Fields: Theory, applications
and algorithms,  Cont.~Math. {\bf 168} (1994), 69--100. 

\bibitem[Fr94b]{Fr94b} \bysame, \emph{Enhanced review of J.P. SerreÕs Topics in 
Galois Theory, with examples illustrating braid rigidity}, BAMS 30 \#1 (1994), 124--135. ISBN 0-86720-210-6.  Recent Developments in the Galois Problem, 
Cont. Math., proceedings of AMS-NSF Summer Conference, Seattle 186 (1995), 15--32.

\bibitem[Fr95a]{Fr95a} \bysame, Extension of Constants, Rigidity, and the Chowla-Zassenhaus Conjecture, 
Finite Fields and their applications, Carlitz volume 1 (1995), 326--359.

\bibitem[Fr95b]{Fr95b} \bysame, \emph{Introduction to Modular Towers: Generalizing the relation between dihedral
groups and modular curves}, Proc. of Recent developments in the Inverse Galois
Problem, Pub. AMS, RI, Cont. Math. \textbf{186} (1995), pp. 111--171.

\bibitem[Fr99]{Fr99}
\bysame, \emph{Separated variables polynomials and moduli spaces}, No.~Th. in Prog.~(Berlin-New York) (ed. J.~Urbanowicz K.~Gyory, H.~Iwaniec,
   ed.), Walter de Gruyter, 1999, Proc. of the Schinzel Festschrift,
   Summer 1997: Available from http://www.math.uci.edu/{$\tilde{\phantom
   u}$mfried/\#math}, 169--228.

\bibitem[Fr05a]{Fr05a} \bysame, \emph{Relating two genus 0 problems of John Thompson}, Volume for John 
ThompsonÕs 70th birthday, in Progress in Galois Theory, H. Voelklein and T. Shaska editors 2005 
Springer Science, 51--85. \\ See http://www.math.uci.edu/$\tilde{\  }$mfried/deflist-cov/Genus0-Prob.html 

\bibitem[Fr05b]{Fr05b} \bysame, \emph{The place of exceptional covers among all diophantine relations}, J. Finite Fields \textbf{11} (2005) 367--433, arXiv:0910.3331v1. 

\bibitem[Fr07]{Fr07} \bysame, \emph{Should Journals compensate Referees?}, May 2007 Notices of the AMS, Vol. \textbf{54} (2007), No.6, p.585. 

\bibitem[Fr08]{UM08} \bysame, \emph{Algebraic Equations and Finite Simple Groups}: What I learned from graduate school at the University of Michigan, 1964Ð1967, CONTINUUM Ð 2008, UM Math.~dept. p.~16.

\bibitem[Fr09]{Fr09} \bysame, Riemann's Existence Theorem: An elementary approach to moduli, http://www.math.uci.edu/~mfried/booklist-ret.html

\bibitem[Fr10]{Fr10} \bysame, \emph{Alternating groups and moduli space lifting Invariants}, Arxiv \textbf{\#0611591v4}. Israel J. Math. \textbf{179} (2010), 57--125 (DOI 10.1007/s11856-010-0073-2). 

\bibitem[FrGS93]{FrGS93} M.D.~Fried, R. Guralnick and J. Saxl, Schur Covers and CarlitzÕs Conjecture, Israel J.  Thompson Volume \textbf{82} (1993), 157--225.

\bibitem[FrGu11]{FrGu11} M.D.~Fried and I.~Gusi\'c \emph{Schinzel's Problem: Imprimitive covers and the monodromy method}, preprint March 2011, intended for the 75th birthday volume of A.~Schinzel. 

\bibitem[FrJ86]{FrJ86} M.~Fried and M.~Jarden, \emph{Field arithmetic}, Ergebnisse der Mathematik III, vol. \textbf{11},
Springer Verlag, Heidelberg, 1986 (455 pgs); 2nd Edition 2004 (780 pgs) ISBN 3-540-22811-x. We quote here both the  first and second ed., using [FrJ86]$_1$ and [FrJ86]$_2$ respectively.

\bibitem[FrM69]{FrM69}
M.D.~Fried and R.~E.~MacRae,
\emph{On the invariance of chains of fields}, Ill.~J.~of Math.
   \textbf{13} (1969), 165--171.

\bibitem[FrL87]{FrL87} M.D.~Fried and R.~Lidl, \emph{On dickson polynomials  and R\'edei functions}, 139--149 in: Contributions to General Algebra 5 (Salzburg, 1986), H\"older-Pichler-Tempsky, Vienna, 1987. 

\bibitem[FrMz02]{FrMz02} M.D.~Fried and A.~Mezard, \emph{Configuration Spaces for Wildly Ramified covers}, in Vol. \textbf{70}, Arith. Groups and Noncommut. Alg., M.D. Fried and Y. Ihara eds., AMS publications (2002), 353--376.

\bibitem[FrS76]{Fr76}  M.D.~Fried and G.~Sacerdote, Solving diophantine problems over all residue class fields of a 
number field ...,  Annals Math.  \textbf{104} (1976), 203--233.

\bibitem[FrV91]{FrV91} M.D.~Fried and H.~V\"olklein, \emph{The inverse Galois problem and rational points on moduli spaces},  Math. Ann. \textbf{290}, (1991) 771--800. 

\bibitem[FrV92]{FrV92} \bysame, \emph{The embedding problem over an Hilbertian PAC field}, Annals of Math. \textbf{135} (1992), 469--481.

\bibitem[FrW82]{FrW82} M.~Fried and R.~Whitley, \emph{Effective Branch Cycle Computation}, preprint 1982, contains referee comments, available from http://math.uci.edu/$\tilde{\  }$mfried/paplist-cov/EffCompBrCycles.pdf

\bibitem[Fu66]{Fu66} W. Fulton, \emph{Fundamental group of a curve}, Archive, Princeton University Library, 1966.

\bibitem[Ga96]{Ga96} M.~Garuti, Prolongement de rev\^etements galoisiens en g\'eom\'etrie rigide. \emph{Extension of Galois coverings in rigid geometry}, Comp.~Math. \textbf{104} (1996), no. 3, 305--331. MR 98m:14023

\bibitem[GLS]{GLS} D. Gorenstein, R. Lyons, R. Solomon, The Classification of Finite Simple Groups, 
Number 3, Mathematical Surveys and Monographs, 40 ISBN:0821803913.

\bibitem[Gri70]{Gri70} P.~Griffiths, \emph{Periods of integrals on algebraic manifolds; \dots}, BAMS \textbf{76} (1970), 228--296.

\bibitem[Gr59]{Gr59} A. Grothendieck, \emph{G\'eom\'etrie formelle et g\'eom\'etrie alg\'ebraique}, 
S\'eminaire Bourbaki, 5. 11, no. \textbf{182}, 1958/59. 

\bibitem[G03]{G03} R.~Guralnick, \emph{Monodromy groups of coverings of curves}, Galois groups and fundamental groups, Math. Sci. Res. Inst. Publ., 41, Cambridge Univ. Press,
Cambridge, 2003,  1--46.

\bibitem[GFM99]{GFM99} R.~Guralnick, D.~Frohardt and K.~Magaard, \emph{Genus 0 actions of groups of Lie rank 1}, in Arith.~fund.~groups and noncommutative alg., Proceedings
of Symp. in Pure Math, \textbf{70} (2002) eds M. Fried and Y.
Ihara, 1999 von Neumann Conf., Aug. 16-27, 1999 MSRI, 449Ð-484.

\bibitem[GM98]{GM98} R.~Guralnick and K.~Magaard, \emph{On the minimal degree of a permutation representation}, J.~Alg, \textbf{207} (1998), 127--145.

\bibitem[GMu97]{GMu97} R.~Guralnick and P.~M\"uller, \emph{Exceptional polynomials of affine type}, J. Algebra~\textbf{194}  (1997), 429--454.

\bibitem[GMS03]{GMS03} R.~Guralnick, P.~M\"uller and J.~Saxl, \emph{The rational function analoque of a question of
Schur and exceptionality of permutations representations}, Memoirs of the AMS 162 773 (2003), ISBN 0065-9266.

\bibitem[Gsh07]{GSh07} R.~Guralnick and J.~Shareshian, \emph{Symmetric and Alternating Groups as Monodromy Groups of Riemann
Surfaces I: Generic Covers and Covers with Many Branch Points},  Mem.~AMS. 2007 \textbf{189}, No. 886, 128 pp.


\bibitem[GT90]{GT90} R.~Guralnick and J.~G.~Thompson
\emph{Finite groups of genus zero}, J.~Alg. \textbf{131} (1990), 303--341.

\bibitem[GTuZ08]{GTuZ08} R.~Guralnick, T.J.~Tucker and M.~Zieve, \emph{Exceptional covers and bijections on rational points}, arXiv: 0511276v2. 

\bibitem[Gu10]{Gu10} I.~Gusi\'c, \emph{Reducibility of $f(x) - cf(y)$}, preprint as of June 2010. 

\bibitem[Ha63]{Ha63} M. Hall, \emph{The Theory of Groups}, MacMillan, NY 1963.
\bibitem[HM87]{HMpicture} P.~Halmos, \emph{I Have a Photographic Memory},  Math.~Ass. of America. The AMS allows  this on my web site http://www.math.uci.edu/$\tilde{\ }$mfried/giffiles/fried-HalmosBook.pdf. 

\bibitem[H94]{H94} D. Harbater, \emph{AbhyankarÕs conjecture on Galois groups over curves}, Invent. Math. \textbf{117} (1994), 1--25.

\bibitem[Hal07]{Hal07} T.~Hales, \emph{What is motivic measure?}, the version WhatIsMotivicMeasure.pdf arXiv0511276v2 has an appendix including references to Galois stratification. 

\bibitem[Har77] {Har77} R. Hartshorne, \emph{Algebraic Geometry}, Grad. Texts in Math. \textbf{52}, Springer-Velag, 1977.

\bibitem[Hi1892]{Hi1892} D.~Hilbert, \emph{\"Uber die Irreduzibilit\"at ganzer rationaler Funktionen mit ganzzahligen Koeffizienten}, J.~f\"ur die reine und angewandte Math.~\textbf{110} (1892), 104--129. 

\bibitem[Is94]{Is94} I.M.~Isaacs, \emph{Algebra, a Graduate Course}, Brooks/Cole Publishing, 1994. 

\bibitem[Ka89]{Ka89} V.~Kanev, \emph{Spectral curves, simple Lie algebras, and Prym-Tjurin varieties}, Theta 
function--Bowdoin 1987, Part 1  (1989), (Brunswick, ME, 1987), Proc. Sympos. Pure
Math., \textbf{49}, JAMS, Prov., RI, 627--645.

\bibitem[Ki76]{Ki76} K.~Kiefe, \emph{Sets definable over finite fields: Their zeta functions}, TAMS \textbf{223} (1976), 45--59. 

\bibitem[Kz81]{Kz81} N.M.~Katz, \emph{Monodromy of families of curves: Applications of some results of
Davenport-Lewis}, Sem. on No.~ th., Paris 1979Ð1980, Prog. in Math. 12,
Birkhauser, Boston (1981), 171--195.

\bibitem[KSi08]{KSi08} I.~Kriz and P.~Siegel, \emph{Simple Groups at Play}, July 2008 Sci. Amer., 84--89.

\bibitem[La71]{La71} S.~Lang, \emph{Algebra}, Addison-Wesley, 1971. 


\bibitem[LZ96]{LZ96} H.W.~Lenstra and M.~Zieve, \emph{A family of exceptional polynomials in characteristic 3}, eds.~Cohen and Neiderriter, London Math.~Soc.~Lecture nts. \textbf{233}, CUP (1996), 209--218.

\bibitem[Le64]{Le64} W.J.~LeVeque, \emph{On the equation $y^m=f(x)$}, Acta.~Arith. \textbf{9} (1964), 209--219. 

\bibitem[LSc80]{LSc80} D.J.~Lewis and A.~Schinzel, \emph{Quadratic diophantine equations with parameters}, Acta Arith.~ \textbf{37} (1980), 133-141. 

\bibitem[LMT93]{LMT93} R.~Lidl, G.L.~Mullen and G.~Turnwald, \emph{Dickson Polynomials}, Pitman monographs and
Surveys in pure and applied math. textbf{65}, Longman Scientific, 1993.

\bibitem[LPS]{LPS} M.~Liebeck, C.~Praeger, J. Saxl, \emph{The maximal factorizations of the finite simple 
groups and their automorphism groups}, Mem. AMS 86 \#432 (1990).

\bibitem[Mac67]{Mac67}  C.~MacCluer, \emph{On a conjecture of Davenport and Lewis concerning exceptional polynomials}, Acta. Arith. \textbf{12} (1967), 289--299. 

\bibitem[Mat84]{Mat84} R.~Matthews, \emph{Permutation polynomials over algebraic number fields}, J.~Number Theory, vo..~\textbf{18} no.~3 (1984), 249--260. 

\bibitem[Ma77]{Ma77} B.~Mazur, \emph{Modular curves and the Eisenstein ideal} , IHES 
Publ. Math. \textbf{47} (1977), 33--186.

\bibitem[Me96]{Me96} L.~Merel, \emph{Bornes pour la torsion des courbes elliptiques sur les corps de nombres},
Invent.~Math. \textbf{124} (1996), 437--449.

\bibitem[Mes90]{Mes90} J.-F. Mestre, \emph{Extensions r\'eguli\`eres de $\bQ(t)$ de groupe de Galois $\tilde A_n$}, J. of Alg. \textbf{131} (1990), 483--495.

\bibitem[M\"u95]{Mu95} P. M\"uller, \emph{Primitive monodromy groups of polynomials}, Proceedings of the Recent 
developments in the Inverse Galois Problem conference, vol. \textbf{186}, 1995, AMS Cont. Math series, pp. 
385--401.

\bibitem[M\"u96]{Mu96} \bysame, \emph{Reducibility behavior of polynomials with varying coefficients}, Israel J.~\textbf{94} (1996), 59--91.

\bibitem[M\"u98]{Mu98} \bysame, \emph{Kronecker conjugacy of polynomials}, TAMS \textbf{350} (1998), 1823--1850.

\bibitem[M\"u98b]{Mu98b} \bysame, \emph{$(A_n,S_n)$-realizations by polynomials -- on a question of Fried}, Finite Fields Appl. 4 (1998), 465--468.

\bibitem[M\"u06]{Mu06} \bysame, \emph{The Degree 8 Examples in Davenport's Problem}, preprint  November 30, 2006.

\bibitem[Mu66]{Mum66}  D. Mumford, \emph{The Red Book: Introduction to Algebraic Geometry}, reprinted from 
1966 Harvard Lectures notes by Springer.

\bibitem[Mu76]{Mum76} \bysame, \emph{Curves and their Jacobians}, Ann Arbor, UM Press, 1976.

\bibitem[Ni10]{Ni10} J.~Nicaise, \emph{Relative Motives and the Theory of Pseudo-finite Fields},  Int. Math. Res. (2010), 1--69.

\bibitem[Pa09]{Pa09} F.~Pakovich, \emph{Prime and composite Laurent polynomials}, Bull. Sci. Math, \textbf{133} (2009), 693--732. 

\bibitem[Pa10a]{Pa10a} \bysame, \emph{On the equation $P(f) = Q(g)$, where $P,Q$ are
polynomials and $f, g$ are entire functions}, Amer. J. Math \textbf{132} no. 6 (2010), .

\bibitem[Pa10b]{Pa10b} \bysame, \emph{Algebraic curves $P(x) - Q(y) = 0$ and functional
equations}, First Published on: 29 September 2010. 

\bibitem[Pi1887]{Pi1887} E.~Picard, \emph{D\'emonstration dÕun th\'eor\`eme g\'en\'eral sur les fonctions uniformes li\'ees par une relation alg\'ebrique}, Acta Math. XI. 1--12 (1887).

\bibitem[Ra94]{Ra94} M. Raynaud, \emph{Rev\^etements de la droite affine en caract\`eristique $p > 0$  et conjecture
dÕAbhyankar}, Invent. Math. \textbf{116}(1994), 425--462.

\bibitem[Ri22]{Ri22} J.F.~Ritt, \emph{Prime and composite polynomials}, TAMS \textbf{23} (1922),
51--66.

\bibitem[Sc71] {Sc71} A. Schinzel, \emph{Reducibility of Polynomials}, Int. Cong. of Math. Nice 1970 (1971), 
Gauthier-Villars Žd., 491--496.

\bibitem[Sc82]{Sc82} \bysame, \emph{Selected Topics on Polynomials}, Ann Arbor UM Press, 1982.

\bibitem[Se68]{Se68} J.-P. Serre, \emph{Abelian $\ell$-adic representations and elliptic curves}, 1st ed., McGill University
Lecture Notes, Benjamin, New York $\bullet$ Amsterdam, 1968, written in collaboration
with Willem Kuyk and John Labute; 2nd corrected ed.~by A. K. Peters,
Wellesley, MA, 1998.

\bibitem[Se90]{Se90} \bysame, \emph{Rel\`evements dans $\tilde A_n$}, C.R.~Acad.~Sci. Paris, t.~\textbf{111}, Serial I (1990), 478--482. 

\bibitem[Se92]{Se92}  \bysame, \emph{Topics in Galois Theory}, 1992, Bartlett and Jones Publishers, 

\bibitem[Si29]{Si29} C.L.~Siegel, \emph{\"Uber einige Anwendungen diophantischer Approximationen}, Abh.~Preus.~Akad.~Wiss.~Phys.--Math. Kl, \textbf{1} (1929), 14--67. 

\bibitem[So01] {So01} R.~Solomon, A Brief History of the Classification of the Finite Simple Groups, 
BAMS 38 (3) (2001), 315--352. 

\bibitem[Sp57]{Sp57} G.~Springer, \emph{Introduction to Riemann Surfaces}, Addison-Wesley, 1957. 

\bibitem[Tv64]{Tv64} H.~Tverberg, \emph{A remark on Ehrenfeucht's criterion for the irreducibility of polynomials}, Prace Mat. 8 (1963/64), 117--118. 

\bibitem[Tv68]{Tv68} \bysame, \emph{A Study in Irreducibility of Polynomials}, PhD Thesis, Univ. Bergen, 1968.

\bibitem[Turn95]{Turn} G.~Turnwald, \emph{On Schur's conjecture}, J.~Austral.~Math~Soc.~Ser.~A 
\textbf{58}   (1995), 312--357.
   

\bibitem[vdW35]{vdW}
B.~L.~van der Waerden, \emph{Die Zerlegungs- und Tr{\"a}gheitsgruppe als 
Permutationsgruppen}, Math.~Ann. \textbf{111} (1935), 
731--733.

\bibitem[vdP82]{vdP82} A.J.~van der Poorten, \emph{The growth conditions for recurrence
sequences}, unpublished (1982). 

\bibitem[Ve08]{FVe08} F.~Vetro, \emph{Irreducibility of Hurwitz spaces of coverings with one special fiber and
monodromy group a Weyl group of type $D_d$}, Man.~Math., \textbf{125} , no.~3 (2008), 353--368; doi:10.1007/s00229-007-0153-8.

\bibitem[Ve09]{FVe09} \bysame, \emph{On Hurwitz spaces of coverings with one special
fiber},  PJM   vol. \textbf{240} (2009), No. 2, 383--398.

\bibitem[V87]{V87}
P.~Vojta, \emph{Diophantine approximation and value distribution theory},  Lecture Notes 
in Math.~Springer-Verlag \textbf{1239}, 1987. 

\bibitem[Vo96]{Vo96}  H.~V\"olklein, \emph{Groups as Galois Groups}, Cambridge Studies in Adv.~Math.~\textbf{53}, 1996.

\bibitem[We28]{We28} A.~Weil, \emph{L'arithmetique sur les courbes alg\'ebriques}, Acta Math.~\textbf{52} (1928), 281--315. 

\bibitem[Wo64]{Wo64} K.~Wohlfahrt, \emph{An extension of F.~Klein's level concept}, Ill.~J. Math.
\textbf{8} (1964), 529--535. 

\end{thebibliography}
\end{document}